\documentclass[10pt,a4paper]{article}
\usepackage{amsmath}
\usepackage{amssymb}
\usepackage{amsthm}
\usepackage[pdftex]{color,graphicx}
\usepackage[latin2]{inputenc}
\usepackage[T1]{fontenc}
\usepackage[english]{babel}
\usepackage{subfigure}
\usepackage{hyperref}
\usepackage{bbm}
\newtheorem{tw}{Theorem}[section]
\newtheorem{lm}[tw]{Lemma}
\newtheorem{wn}[tw]{Corollary}
\newtheorem{pr}[tw]{Proposition}
\theoremstyle{definition}
\newtheorem{df}{Definition}[section]
\newtheorem{uw}[df]{Remark}
\newtheorem{ex}[df]{Example}

  \newcommand{\pd}[2]{\frac{\partial #1}{\partial #2}}

\numberwithin{equation}{section}
\numberwithin{figure}{section}

\date{}
\title{On the self-similarity problem for smooth flows on orientable surfaces}
\author{Joanna Ku\l{}aga\footnote{Research partially supported by MNiSzW grant N N201 384834 and Marie Curie "Transfer of Knowledge" program, project MTKD-CT-2005-030042 (TODEQ).}\\{\small Faculty of Mathematics and Computer Science, Nicolaus Copernicus University},\\{\small ul. Chopina 12/18, 87-100 Toru\'n, Poland}\\ {\small e-mail: joanna.kulaga@gmail.com} }

\begin{document}
\bibliographystyle{plain}
\maketitle

\begin{abstract}
On each compact connected orientable surface of genus greater than one we construct a class of flows without self-similarities.
\end{abstract}


\tableofcontents

\section{Introduction}\label{intro}

\subsection{Main result}
In this paper we deal with some ergodic properties of flows on surfaces. More precisely, we consider smooth measure-preserving flows on compact connected orientable surfaces of genus $\mathbf{g}\geq 2$ with a finite number of non-degenerate singular points and no saddle connections. Among them we find a class of flows with no self-similarities, i.e. flows $\mathcal{T}=\{T_t\}_{t\in\mathbb{R}}$ for which there is no $s\in\mathbb{R}\setminus\{-1,1\}$ such that the flows $\mathcal{T}$ and $\mathcal{T}\circ s:=\{T_{st}\}_{t\in\mathbb{R}}$ are measure-theoretically isomorphic. Thus we settle an open question raised in~\cite{FL08}. More precisely, we show that the following holds.
\begin{tw}\label{tw:g}
On any closed compact orientable surface of genus greater or equal two there exists a smooth flow which is not self-similar.
\end{tw}
The problems connected with the notion of self-similarity were studied in the past by numerous authors (e.g. in~\cite{Ka-Th},~\cite{Ma} and~\cite{Ry97}). Let us list here some of the results related to self-similarity which by no means constitute a complete survey. Let us mention first a result of an opposite nature to what will be of our interest in this paper. In~\cite{Ma} B.~Marcus showed that every positive number $s$ is a scale of self-similarity of the horocycle flow on a connected orientable surface of constant negative curvature and finite area. The further studies include investigations of the size of the set $I(\mathcal{T})=\{s\in\mathbb{R}\colon \mathcal{T}\simeq \mathcal{T}\circ s \}$ and some disjointness results (see e.g.~\cite{FL08},~\cite{Ry91} and~\cite{Ry96} or also more recent~\cite{da-ry}). There are also quite a few different examples of flows with no self-similarities, which include mixing rank one flows~\cite{Ry97}, special flows over an ergodic interval exchange transformation under some piecewise absolutely continuous roof functions and special flows over irrational rotations satisfying a certain Diophantine condition under some piecewise constant roof functions~\cite{FL08}.

\subsection{Outline of the proof}\label{se:out}

The main idea of the proof of Theorem~\ref{tw:g} is to use a special representation $T^h=\{T_t\}_{t\in\mathbb{R}}$ of considered systems~\cite{Arnold91},~\cite{Kochergin76},~\cite{Zorich94u}. In this representation the base automorphism $T$ is an interval exchange transformation and the roof function $h$ is smooth, except for a finite collection of points where it has logarithmic singularities, i.e. it is of the form $f+g$, where
\begin{equation*}
f(x)=\sum_{0\leq i\leq r-1}\left(-c^{+}_{i}\log\{x-{\beta}_{i}\}\right)+\sum_{0\leq i\leq r-1} \left(-c^{-}_{i+1}\log\{{\beta}_{i+1}-x\}\right),
\end{equation*}
where ${\beta}_i$ for $0\leq i\leq r-1$ are the discontinuity points of the interval exchange transformation $T$ and $\{\cdot\}$ stands for the fractional part. The constants $c_{i+1}^-$, $c_i^+$ for $0\leq i\leq r-1$ are positive (except for two of them which are equal to zero, for more details see Section~\ref{reduction}), while the function $g$ is piecewise absolutely continuous.

The methods of showing that a flow is not self-similar developed in~\cite{FL08} rely on two properties: the absence of partial rigidity and a condition which is strictly related to the absence of mixing. We will use the following result ($J(\mathcal{T})$ stands for the set of self-joinings of $\mathcal{T}$, $\mathcal{P}(\mathbb{R})$ is the set of all probability Borel measures on $\mathbb{R}$ and $\{\cdot \}^d$ denotes the closure in the weak operator topology).
\begin{lm}[\cite{FL08}]\label{lm:fl08}
Let $\mathcal{T}=\{T_t\}_{t\in\mathbb{R}}$ be a measure-preserving flow on $(X,\mu)$. 
If $\mathcal{T}$ is not partially rigid and $a\int_{\mathbb{R}}T_t\ dP(t)+(1-a)J$ belongs to $\{T_t\colon t\in \mathbb{R}\}^d$ for some $P\in\mathcal{P}(\mathbb{R})$, $0<a\leq 1$ and $J\in J(\mathcal{T})$ then $\mathcal{T}$ is not self-similar.
\end{lm}
As we can see, there are two main ingredients needed to show the absence of self-similarities. One of them is that in the weak closure of time automorphisms we can find an operator of the form $a\int_{\mathbb{R}}T_t\ dP(t)+(1-a)J$. Due to a result from~\cite{FL05} (see Theorem~\ref{tw15}) this condition can be replaced in our situation by the boundedness of the sequence $\{ \int_{D_n} |f^{(q_n)}(x)-a_n|^2\ d \mu (x)\}$, where $D_n$ are appropriately chosen rigidity subsets of the interval $[0,1)$ in the base of the special flow. Important in the process of obtaining an operator of the form $a\int_{\mathbb{R}}T_t\ dP(t)+(1-a)J$ in the weak closure of time automorphisms is the sequence of measures $\left\{\frac{1}{\mu(D_n)}\left(\left(f^{(q_n)}(x)-a_n\right)|_{D_n}\right)_{\ast}\left(\mu|_{D_n}\right)\right\}$, which turns out to be uniformly tight whenever the sequence $\left\{ \int_{D_n} |f^{(q_n)}(x)-a_n|^2\ d \mu (x)\right\}$ is bounded. Recall that $\left(\left(f^{(q_n)}(x)-a_n\right)|_{D_n}\right)_{\ast}\left(\mu|_{D_n}\right)$ stands for the image of measure $\mu|_{D_n}$ via $\left(f^{(q_n)}(x)-a_n\right)|_{D_n}$. A theorem recently proved by C. Ulcigrai in~\cite{0901.4764} ensures that the sequence $\left\{ \int_{D_n} |f^{(q_n)}(x)-a_n|^2\ d \mu (x)\right\}$ is in fact bounded. This condition is strictly connected with the absence of mixing.

The other component needed to prove the absence of self-similarities is the absence of partial rigidity. This will be our main technical concern, i.e. we have to show that there is no $0<u<1$ and no sequence $\{t_n\}_{n\in\mathbb{N}}$ ($t_n \to \infty$) satisfying $\liminf_{n\to\infty}\mu (A\cap T_{-t_n}A)\geq u\mu(A)$ for every measurable set $A$, where $\mu$ is the measure preserved by the flow. As the base automorphism we exploit interval exchange transformations with balanced partition lengths (see the definition in Section~\ref{defin}). When $T$ is an irrational rotation by $\alpha$, the property of balanced partition lengths means that $\alpha$ has bounded partial quotients in its continued fraction expansion. To give also examples of flows without self-similarities over interval exchange transformations of more than two intervals, we show that all interval exchange transformations for which the renormalized Rauzy induction \cite{Rauzy79},~\cite{Veech82} is periodic also have balanced partition lengths.

\subsection{Organization of the remaining part of the paper}
In Section~\ref{defin} we first recall the definitions of \emph{self-similarities} (Section~\ref{se:2.1}) and \emph{partial rigidity} (Section~\ref{se:2.2}). Then we give the necessary information from the \emph{theory of joinings} (Section~\ref{se:2.3}). In Section~\ref{przek} we introduce notation and recall the definition of an \emph{interval exchange transformation}. We explain how to obtain an interval exchange transformation from an interval exchange transformation \emph{on the circle}. We also recall some basic facts connected with the \emph{continued fraction expansion} of irrational numbers. In Section~\ref{se:2.5} we concentrate on the \emph{Rauzy induction}: we recall its definition and also the definition of the \emph{Rauzy cocycle}. The further information is related to the \emph{towers} for interval exchange transformation and the \emph{Rauzy heights cocycle}. Section~\ref{se:2.6} is devoted to interval exchange transformations of \emph{periodic type}. We first recall the definition and in Section~\ref{se:2.6.1} we introduce the notion of \emph{balanced partition lengths}. In Section~\ref{se:2.7} we recall basic information about the \emph{special flows}.

In Section~\ref{reduction} we describe how to obtain a special flow representation of flows on closed compact orientable surfaces, which are given by closed 1-forms, with a finite number of non-degenerate critical points and no saddle connections.

In Section~\ref{rigidity} we show that the flows in some class of special flows over interval exchange transformation under the roof function with symmetric logarithmic singularities are not partially rigid (Theorem~\ref{tw12}). Namely, the claim of Theorem~\ref{tw12} holds whenever the interval exchange transformation in the base has balanced partition lengths.

The main concern in Section~\ref{balparle} is with the interval exchange transformations with balanced partition lengths. We show that this class of IETs includes all IETs of periodic type (see Lemma~\ref{kazdypjestb}).

In Section~\ref{selfsim} we prove the absence of self-similarities (Theorem~\ref{tw:g}) for the considered class of special flows using the results proved in Section~\ref{rigidity}. In Section~\ref{se:6.3} we deal with the problem of absence of spectral self-similarities. We formulate spectral counterparts of the results from~\cite{FL05} which are needed to prove the absence of metric self-similarities (see Theorem~\ref{tw:spco}) which allows us to prove the absence of spectral self-similarities. We give examples of special flows with no spectral self-similarities which can be obtained as a representation of smooth flows (with saddle connections) on surfaces of genus $\mathbf{g}\geq 2$ (see Example~\ref{przykladzik}).

\section{Definitions}\label{defin}

\subsection{Self-similarities}\label{se:2.1}
Let $\mathcal{T}=\{T_t\}_{t\in\mathbb{R}}$ be an ergodic measurable flow on a standard probability Borel space $(X,\mathcal{B},\mu)$. For $s\in\mathbb{R}\setminus\{0\}$ by $\mathcal{T}\circ s$ we denote the flow $\{T_{st}\}_{t\in\mathbb{R}}$.
\begin{df}
If $I(\mathcal{T})=\{s\in\mathbb{R}\colon \mathcal{T}\text{ and }\mathcal{T}\circ s\text{ are isomorphic}\} \subset\{-1,1\}$,
we say that the flow $\mathcal{T}$ has no \emph{self-similarities}. If there exists $s\in I(\mathcal{T})\setminus \{-1,1\}$ we say that $\mathcal{T}$ is \emph{self-similar} with the scale of self-similarity $s$.
\end{df}

\subsection{Joinings}\label{se:2.3}
Let $\mathcal{T}=\{T_t\}_{t\in\mathbb{R}}$ and $\mathcal{S}=\{S_t\}_{t\in\mathbb{R}}$ be measurable flows on $(X,\mathcal{B},\mu)$ and $(Y,\mathcal{C},\nu)$ respectively (by measurability of the flow $\{T_t\}_{t\in\mathbb{R}}$ we mean that the map $t \mapsto \left\langle f\circ S_t,g \right\rangle$ is continuous for all $f,g \in L^2 (X,\mathcal{B},\mu)$). By $\mathcal{J}(\mathcal{T},\mathcal{S})$ we denote the set of all joinings between $\mathcal{T}$ and $\mathcal{S}$, i.e. the set of all $\{T_t\times S_t\}_{t\in\mathbb{R}}$-invariant probability measures on $(X\times Y,\mathcal{B}\otimes \mathcal{C})$, whose projections on $X$ and $Y$ are equal to $\mu$ and $\nu$ respectively. For $\mathcal{J}(\mathcal{T},\mathcal{T})$ we write $\mathcal{J}(\mathcal{T})$. Joinings are in one-to-one correspondence with Markov operators $\Phi\colon L^2(X,\mathcal{B},\mu)\to L^2(Y,\mathcal{C},\nu)$ satisfying $\Phi\circ T_t=S_t \circ \Phi$ for all $t\in\mathbb{R}$. We denote the set of such Markov operators by $J(\mathcal{T},\mathcal{S})$ (as in case of measures we write $J(\mathcal{T})$ for $J(\mathcal{T},\mathcal{T})$). This identification allows us to view $\mathcal{J}(\mathcal{T})$ as a metrisable compact semitopological semigroup endowed with the weak operator topology. We say that $\mathcal{T}$ and $\mathcal{S}$ are disjoint if $\mathcal{J}(\mathcal{T},\mathcal{S})=\{\mu\otimes\nu\}$ (the notion of disjoitness was introduced by H.~Furstenberg in~\cite{Furstenberg67}). Given a flow $\mathcal{T}=\{T_t\}_{t\in\mathbb{R}}$ and a Borel probability measure $P$ on $\mathbb{R}$, we define the operator $\int_{\mathbb{R}}T_t\ dP(t)$ acting on $L^2(X,\mathcal{B},\mu)$ by $\left\langle(\int_{\mathbb{R}}T_t\ dP(t))f,g\right\rangle=\int_{\mathbb{R}}\left\langle T_t f,g \right\rangle dP(t)$ for all $f,g\in L^2 (X,\mathcal{B},\mu)$.

\subsection{Partial rigidity}\label{se:2.2}
\begin{df}
Let $\mathcal{T}=\{T_t\}_{t\in\mathbb{R}}$ be a measurable flow on a standard probability space $(X, \mathcal{B},\mu)$. The flow $\mathcal{T}$ is said to be \emph{partially rigid} along $\{t_n\}_{n\in\mathbb{N}}$ if there exists $0<u\leq 1$ such that
\begin{equation*}
\liminf_{n\to\infty}\mu(A\cap T_{-t_n}A)\geq u\mu(A)\text{ for every }A\in\mathcal{B}.
\end{equation*}
\end{df}

\begin{uw}\label{czesc}
Let $\mathcal{T}=\{T_t\}_{t\in\mathbb{R}}$ be an ergodic flow on a standard probability space $(X,\mathcal{B},\mu)$ which is partially rigid along time sequence $t_n\to \infty$ with rigidity constant $u\in (0,1]$. Then for any subsequence $\left( t_{n_k}\right)_{k\in\mathbb{N}}\subset \left(t_n \right)_{n\in\mathbb{N}}$ such that $T_{t_{n_k}}$ is convergent in weak operator topology there exists $K\in\mathcal{J}(T)$ such that
\begin{equation}\label{eq:jak}
\lim_{k\to\infty}T_{t_{n_k}} = u\cdot \text{Id} + (1-u)\cdot K.
\end{equation}
Indeed, let $n_k\to\infty$ be such a sequence that $T_{t_{n_k}}$ converges. Let $\Phi = \lim_{k\to\infty}{T_{t_{n_k}}}$. For any sets $A,B\in\mathcal{B}$ we have
\begin{equation*}
\lim_{k\to\infty}\mu\left(T_{t_{n_k}}A \cap B\right) \geq \lim_{k\to\infty} \mu \left( T_{t_{n_k}}(A \cap B) \cap (A\cap B) \right) \geq u \cdot \mu(A\cap B).
\end{equation*}
In other words, the following inequality holds for any $A,B\in\mathcal{B}$:
\begin{equation*}
\int_B  (\Phi-u\cdot \text{Id})(\mathbbm{1}_A)\ d\mu \geq 0
\end{equation*}
Therefore, letting $K:=\frac{\Phi-u\cdot\text{Id}}{1-u}$, we obtain
\begin{equation*}
Kf\geq 0 \text{ for any nonnegative function }f\in L^2(X,\mathcal{B},\mu).
\end{equation*}
Moreover,
\begin{itemize}
\item
$K \mathbbm{1} = K^{\ast}\mathbbm{1}=\mathbbm{1}$,
\item
$U_T \circ K = K \circ U_T$,
\end{itemize}
whence $K\in\mathcal{J}(T)$. This means that~\eqref{eq:jak} indeed holds.

On the other hand, whenever
\begin{equation*}
\lim_{k\to\infty}T_{n_{k}} = u\cdot \text{Id} + (1-u)\cdot K
\end{equation*}
for some $K\in\mathcal{J}(T)$, flow $\mathcal{T}$ is partially rigid along $\{t_n\}_{n\in\mathbb{N}}$.
\end{uw}

\subsection{Interval exchange transformations of $r\geq 2$ intervals}\label{przek}
\subsubsection{General definition}
An \emph{interval exchange transformation} (IET) is a piecewise order-preserving isometry of a finite interval. To describe an IET of $r\geq 1$ intervals on $[0,\lambda)$ we need the following data\footnote{We use the notation introduced by Marmi, Moussa and Yoccoz in~\cite{MMY05}.}: a pair of permutations of $r$ symbols $(\pi_0,\pi_1)$ and a vector $\underline{\lambda}=(\lambda_1,\lambda_2,\dots,\lambda_r)$ of lengths ($\lambda_i>0$ for $1\leq i\leq r$, $\sum_{i=1}^{r}\lambda_i=\lambda>0$). For $j=1,\dots,r$ the map $T$ is described by the formula
\begin{equation*}
Tx=x-\sum_{\pi_0(i)<j}\lambda_i + \sum_{\pi_1(i)<\pi_1(\pi_0^{-1}(j))}\lambda_i,\ x\in\left[\sum_{\pi_0(i)<j}\lambda_i,\sum_{\pi_0(i)\leq j}\lambda_i \right).
\end{equation*}
The pair $(\pi_0,\pi_1)$ determines the ordering of the subintervals before and after the map is iterated and $\underline{\lambda}$ is the vector of the lengths of the exchanged intervals. In what follows, we will always consider only \emph{irreducible pairs} $(\pi_0,\pi_1)$, i.e. such that for $1\leq k< r$ 
\begin{equation*}
\pi_0^{-1}(\{1,\dots,k\})\neq\pi_1^{-1}(\{1,\dots,k\})
\end{equation*}
(otherwise we could decompose the IET into two disjoint invariant subintervals and analyse two simpler dynamical systems). We endow the space $[0,\lambda)$ with Lebesgue measure denoted by $m$.

Let $T$ be an IET defined by the combinatorial data $(\pi_0,\pi_1)$ and by the length data $\underline{\lambda}$. Put
\begin{equation*}
\beta_j=\sum_{\pi_0 (i)\leq j}\lambda_i
\end{equation*}
for $0\leq j\leq r$. These are the discontinuities of $T$.\footnote{All of the points $\beta_j$ are called discontinuities, even though $T$ is continuous at $\beta_0$, it is not defined at $\beta_r$ and it may happen that $T$ is continuous at $\beta_j$ for some $0<j<r$.}
\begin{df}
We say that $T$ satisfies the \emph{infinite distinct orbit condition} (IDOC) if the orbits
\begin{equation*}
\mathcal{O}(\beta_j)=\left\{T^n\beta_j,\ n\in \mathbb{N}\right\} \textrm{ for } 1\leq j\leq r-1
\end{equation*}
are infinite and disjoint.
\end{df}
This definition provides a generalization of the irrational rotation on the circle. As it was proved by M.~Keane~\cite{Keane75}, if $T$ fulfills the IDOC, then all its orbits are dense. Moreover, if $\underline{\lambda}$ is rationally independent and the pair $(\pi_0,\pi_1)$ is irreducible, then $T$ satisfies the IDOC. 

\subsubsection{IETCs}
The definition of IETs can be easily transferred to the case of \emph{interval exchange transformations on the circle} (IETCs).
\begin{df}
By an \emph{interval exchange transformation on the circle} (IETC) we understand a map $T\colon \mathbb{T} \to \mathbb{T}$ which is a piecewise orientation-preserving isometry ($\mathbb{T}$ is identified with $\mathbb{S}^1=\{z\in\mathbb{C} \colon |z|=1\}$).
\end{df}
\begin{uw}\label{uw:ietandIETC}
Every IET yields an IETC by the identification of the ends of the interval. The number of the exchanged intervals (arcs in the case of IETCs) remains the same. 

On the other hand, every IETC yields an IET. Indeed, consider an IETC $T$ of $r-1$ arcs. Let us denote by $0$ one of the discontinuity points of $T$ and treat the circle as the interval $[0,1)$. Typically we obtain an IET of $r$ intervals. A point which is mapped by the IETC to $0$ (in the example in Figure~\ref{ietandIETC} denoted by $\beta_3$) becomes an additional discontinuity for the resulting IET.
\end{uw}


\begin{figure}[ht]
\centering
\includegraphics[height=280pt]{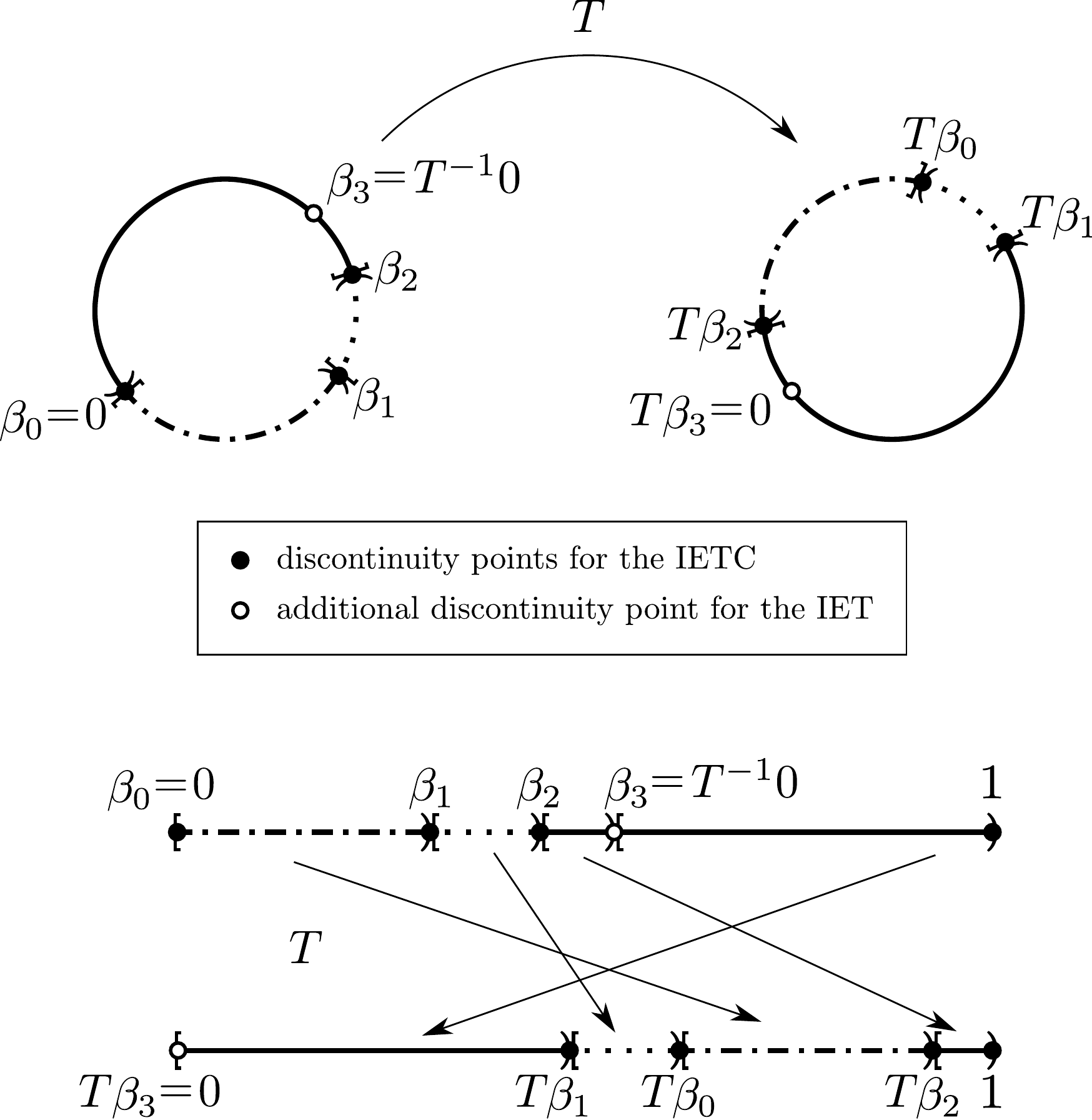}
\caption{IET obtained from an IETC}
\label{ietandIETC}
\end{figure}

\subsubsection{IETs of two intervals}
If $\alpha$ is an irrational number, then we denote by $Tx=x+\alpha$ the corresponding
irrational rotation on $(\mathbb{T}, \mathcal{B}(\mathbb{T}), m)$. The circle $\mathbb{T}=\mathbb{R}/\mathbb{Z}$ is identified with the interval $[0,1)$, the measure $m$ is Lebesgue measure inherited from $[0,1)$. Rotation on the circle is an exchange of two intervals. 

For an irrational $\alpha\in\mathbb{T}$ let $\{q_n\}$ stand for the sequence of its denominators, i.e.
\begin{align*}
p_0=0,\ p_1=1,\ p_{n+1}&=a_{n+1}p_n+p_{n-1},
\\
q_0=1,\ q_1=a_1,\ q_{n+1}&=a_{n+1}q_n+q_{n-1}
\end{align*}
and $[0;a_1,a_2,\dots]$ denotes the continued fraction expansion of $\alpha$.
\begin{df}
Let $\alpha\in\mathbb{T}$ be irrational. It has \emph{bounded partial quotients} if there exists $M>0$ such that $a_n<M$ for all $n\in\mathbb{N}$.
\end{df}

\subsection{Rauzy induction}\label{se:2.5}
Recall the definition of the Rauzy induction map $\mathcal{R}$ on the space of IETs which fulfill the IDOC (the algorithm was introduced and developed by G.~Rauzy and W.~A.~Veech in~\cite{Rauzy79,Veech82}). Let us denote this space by $\mathbf{\Delta}$. For a given IET $T$ exchanging $r$ intervals represented by the triple $(\underline{\lambda},\pi_0,\pi_1)$, set $j_0=\pi_0^{-1}(r)$, $j_1=\pi_1^{-1}(r)$, $I^{(1)}=[0,1-\min(\lambda_{j_0},\lambda_{j_1}))$ and let $\mathcal{R}(T)$ be the induced map on $I^{(1)}$. Due to the IDOC, $\lambda_{j_0}\neq\lambda_{j_1}$. Moreover, we obtain again an IET of $r$ intervals. Let
\begin{equation*}
A(T)=\begin{cases}
      & I+E_{j_0,j_1}\text{ if }\lambda_{j_0}<\lambda_{j_1}, \\
      & I+E_{j_1,j_0}\text{ if }\lambda_{j_1}<\lambda_{j_0},
\end{cases}
\end{equation*}
where $I$ is the identity matrix and $E_{i,j}$ denotes the matrix whose all entries are equal to $0$ except for the $(i,j)$ one which is equal to $1$. This defines the \emph{Rauzy cocycle} $A\colon \mathbf{\Delta} \to SL(r,\mathbb{Z})$ (see~\cite{Zorich97}). The process of inducing on subintervals chosen as described above, can be repeated infinitely many times. Therefore we define $(\underline{\lambda}^{(n)},\pi_0^{(n)},\pi_1^{(n)})=\mathcal{R}^n(\underline{\lambda},\pi_0,\pi_1)$ and $I^{(n)}=[0,1-\min(\lambda_{j_0}^{(n-1)},\lambda_{j_1}^{(n-1)}))$ for $n\geq 0$. The IDOC assures that $\lambda_{j_0}^{(n-1)}$ and $\lambda_{j_1}^{(n-1)}$ are never equal. 
The set of all combinatorial data accessible from the initial one by applying Rauzy induction is called a \emph{Rauzy class}.

\subsubsection{Operations on towers}
Denote by $I_j^{(n)}$, $j=1,\dots,r$, the subintervals exchanged by $\mathcal{R}^nT$. These intervals determine a partition of the given interval $I$ into \emph{towers} $H_j^{(n)}$ ($j=1,\dots,r$), where
\begin{equation}\label{wieze}
H_j^{(n)}=\bigcup_{k=0}^{h_j^{(n)}-1}T^kI_j^{(n)}
\end{equation}
and $h_j^{(n)}$ is the common first return time to the interval $I^{(n)}$ for the points from $I_j^{(n)}$. We call the sets $H_j^{(n)}$ \emph{towers} for $\mathcal{R}^nT$ and the sets $T^kI_j^{(n)}$ the \emph{floors} of the tower $H_j^{(n)}$. Note that once we have fixed $n$, all the floors of all the towers for $\mathcal{R}^nT$ are disjoint:
\begin{equation*}
T^{k_1}I_{j_1}^{(n)} \cap T^{k_2}I_{j_2}^{(n)}=\emptyset
\end{equation*}
for $1\leq j_i\leq r$, $0 \leq k_i \leq h_{j_i}^{(n)}-1$ ($i=1,2$) such that $(j_1,k_1)\neq(j_2,k_2)$.

By \emph{cutting} the tower $H_j^{(n)}$ at the point $x \in I^{(n)}$ we will mean refining the partition into the floors of the towers as follows: if $x \in I_j^{(n)}$, we add to the set of the partition points the set $\left\{x,Tx,\dots,T^{h_j^{(n)}-1}x\right\}$ (see Fig.~\ref{cut}).


\begin{figure}[ht]
\centering
\includegraphics[height=100pt]{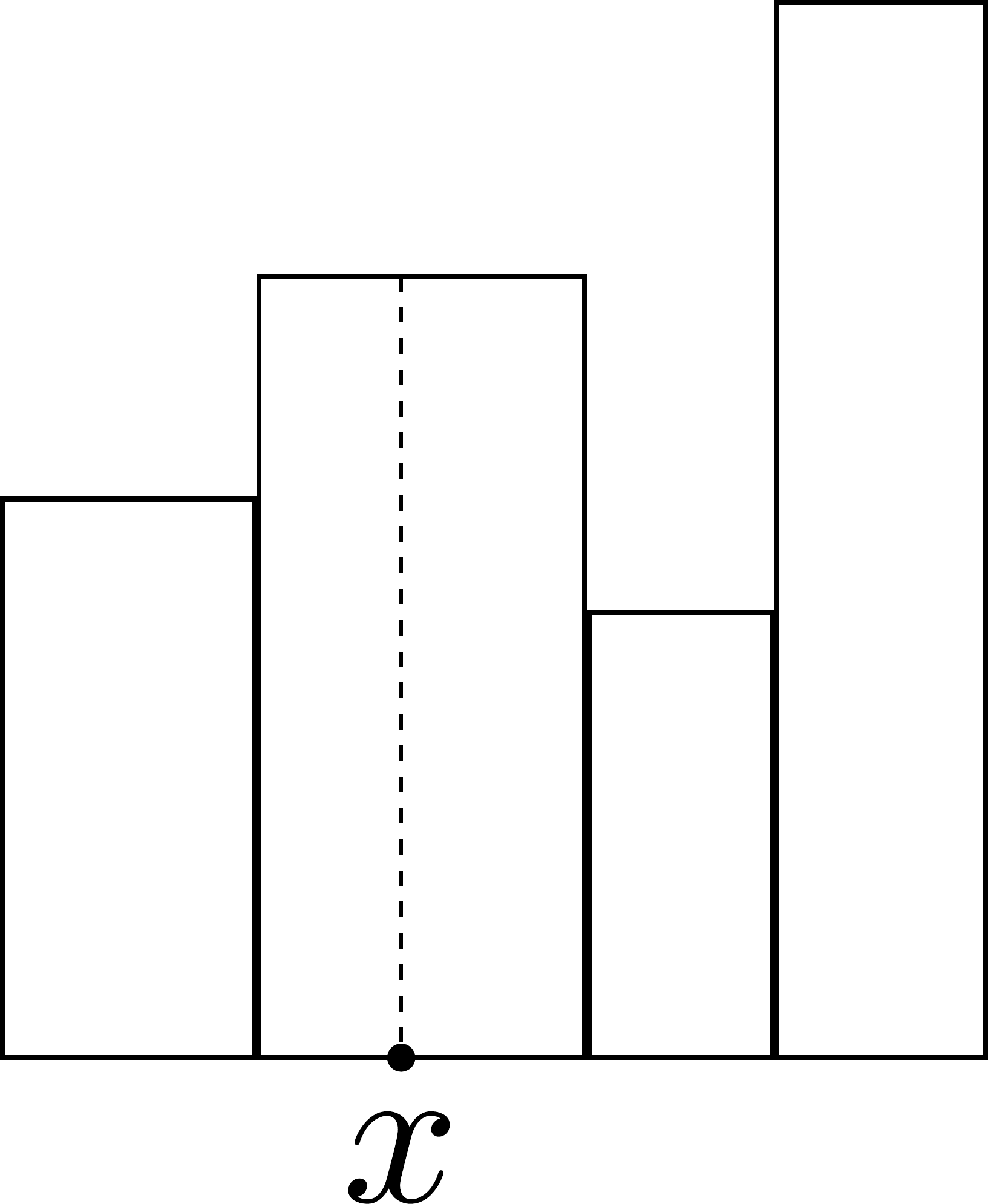}
\caption{Towers cut at $x$}
\label{cut}
\end{figure}

\subsubsection{Rauzy heights cocycle}
Let $\underline{h}^{(0)}$ be the column vector $(1,\dots,1)\in\mathbb{Z}^r$ and $h^{(n)}$ the column vector with heights of the towers for the $n$-th step of Rauzy induction as its entries.  Then we have $h^{(n)}=A(\mathcal{R}^{(n-1)}(T))h^{(n-1)}$ and, denoting by $A^{(n)}$ the product of matrices along the $\mathcal{R}$-orbit of $T$:
\begin{equation*}
A^{(n)}=A(\mathcal{R}^{n-1}(T))\cdot A(\mathcal{R}^{n-2}(T))\cdot \dots \cdot A(\mathcal{R}(T))\cdot A(T),
\end{equation*}
we get
\begin{equation}\label{kocwys}
h^{(n)}=A^{(n)}(1,\dots,1).
\end{equation}
It is the transpose of the cocycle which appears in~\cite{Veech82} and~\cite{Zorich96}, i.e. we can express also the lengths vectors for the induced transformations in terms of the Rauzy cocycle:
\begin{equation*}
\underline{\lambda}^{(n)}=((A^{(n)}(T))^T)^{-1}\underline{\lambda}^{(0)}.
\end{equation*}
For $n\in \mathbb{N}$ let 
\begin{equation*}
h_{\min}^{(n)}=\min_{1\leq j\leq r}h_j^{(n)}
\end{equation*}
and
\begin{equation*}
h_{\max}^{(n)}=\max_{1\leq j\leq r}h_j^{(n)}.
\end{equation*}

\subsection{IETs of periodic type}\label{se:2.6}
\begin{df}
We say that IET $T$ is of \emph{periodic type} if the following two conditions hold:
\begin{itemize}
\item[a)]
the sequence $A(T),\ A(\mathcal{R}T), \dots, A(\mathcal{R}^nT)$ is periodic with some period $p>0$, i.e. $A(\mathcal{R}^{n+p}T)=A(\mathcal{R}^{n}T)$ for all $n\in\mathbb{N}$;
\item[b)]
the period matrix $A^{(p)}(T)$ has strictly positive entries.
\end{itemize}
\end{df}

Examples of IETs of periodic type can be constructed by choosing a closed path on the Rauzy class (for the details we refer to~\cite{SU05}). Moreover, every IET of periodic type can be obtained this way.

If the matrix $B\in SL(r,\mathbb{Z})$  has strictly positive entries, introduce the following quantity (in~\cite{Veech81} there was introduced an analogous definition where the ratios of the entries in the rows was maximized):
\begin{equation*}
\overline{\nu}(B)=\max_{i,j,l}\frac{B_{ij}}{B_{lj}}.
\end{equation*}
Then if $h^{(m+n)}=B\cdot h^{(n)}$, it follows that
\begin{equation}\label{balanced}
\frac{1}{\overline{\nu}(B)}\leq \frac{h_i^{(n+m)}}{h_j^{(n+m)}} \leq \overline{\nu}(B).
\end{equation}
In the case of periodic IETs with period $p$ we will use this fact for $B=A^{(p)}(T)$.

Let $\mathcal{P}$ be a partition of some interval into subintervals. By $\min\mathcal{P}$ and $\max\mathcal{P}$ we denote the minimum and the maximum length of the subintervals determined by this partition. By $\mathcal{P}(a;x_1,\dots,x_k;b)$ we denote the partition of the interval $[a,b)$ by the points $x_1,\dots,x_k$. When there is no ambiguity (e.g. when the considered interval is $[0,1)$) we drop the dependence on the interval $[a,b)$ and write $\mathcal{P}(\{x_i\colon 1\leq i\leq k\})$ for $\mathcal{P}(a;x_1,\dots,x_k;b)$.

\subsubsection{Balanced partition lengths}\label{se:2.6.1}
\begin{df}\label{bal}
Let $T\colon [0,1) \to [0,1)$ be an IET with discontinuity points $0=\beta_0<\beta_1<\dots<\beta_{r-1}<\beta_r=1$.
We say that it has \emph{balanced partition lengths with constant $c>0$} if for any $j\geq 1$ two following conditions hold:
\begin{itemize}
\item[(i)]
\begin{equation*}
\frac{1}{cj}\leq \min\mathcal{P}_j \leq \max\mathcal{P}_j \leq\frac{c}{j},
\end{equation*}
where $\mathcal{P}_j=\mathcal{P}(\{T^{-k}\beta_i\colon 1\leq i\leq r-1,0\leq k\leq j-1\})$;
\item[(ii)]
\begin{multline*}
\frac{1}{cj}\leq \min\mathcal{P}(\{T^{-k+l}\beta_i\colon 0\leq l\leq j-1\})
\\
\leq \max\mathcal{P}(\{T^{-k+l}\beta_i\colon 0\leq l\leq j-1\})\leq\frac{c}{j}
\end{multline*}
for all $0\leq i\leq r-1$ and $0\leq k\leq j-1$.
\end{itemize}
\end{df}

\begin{uw}
Notice that in (i) the partitions under consideration are generated by all the discontinuities whereas in (ii) we treat each discontinuity separately. Moreover, in (ii) we iterate discontinuities both backwards and forwards as opposed to (i) where only backward iterations are taken into account.
\end{uw}

\begin{uw}
Let $T\colon [0,1)\to[0,1)$ be an IET. If the conditions (i) and (ii) of the above definition are fulfilled with different constants, $c_1$ and $c_2$ respectively, then $T$ has balanced partition lengths with constant $c=\max\{c_1,c_2\}$.
\end{uw}

\begin{uw}\label{kolkoodcinek}
Definition~\ref{bal} of balanced partition lengths for IETs can be easily transferred to the case of IETCs. Notice that an IET has balanced partition lengths whenever the corresponding IETC has balanced partition lengths.
\end{uw}

\subsection{Special flows}\label{se:2.7}
Let $T\colon (X,\mathcal{B},\mu)\to(X,\mathcal{B},\mu)$ be an ergodic automorphism of a standard probabilistic space and let $f\in L^1 (X,\mathcal{B},\mu)$ be a strictly positive function. Let $X^f=\{(x,t)\in X\times\mathbb{R} \colon 0\leq t<f(x)\}$. Under the action of the \emph{special flow} $T^f$ each point of $X^f$ moves upwards vertically at the unit speed and we identify the points $(x,f(x))$ and $(Tx,0)$. We put
\begin{displaymath}
f^{(n)}(x) = \left\{ \begin{array}{ll}
f(x)+f(Tx)+\ldots+f(T^{n-1}x) & \textrm{if $n>0$}\\
0 & \textrm{if $n=0$}\\
-(f(T^nx)+\ldots+f(T^{-1}x)) & \textrm{if $n<0$}.
\end{array} \right.
\end{displaymath}
For a formal definition of the special flow, consider the skew product $S_{-f}\colon (X\times \mathbb{R},\mu\otimes m)\to(X\times \mathbb{R},\mu\otimes m)$, where $m$ stands for the Lebesgue measure, given by the equation
\begin{equation*}
S_{-f}(x,r)=(Tx,r-f(x))
\end{equation*}
and let $\Gamma^f$ stand for the quotient space $X\times\mathbb{R}/\sim$, where the relation $\sim$ identifies the points in each orbit of the action on $X\times\mathbb{R}$ by $S_{-f}$. Let $\sigma=\{\sigma_t\}_{t\in\mathbb{R}}$ denote the flow on $(X\times \mathbb{R},\mu\otimes m)$ given by
\begin{equation*}
\sigma_t(x,r)=(x,r+t).
\end{equation*}
Since $\sigma_t \circ S_{-f}=S_{-f}\circ \sigma_t$, we can consider the quotient flow of the action $\sigma$ by the relation $\sim$. This is the special flow over $T$ under $f$ denoted by $T^f$.

\section{Representation as a special flow}\label{reduction}
We will construct a class of flows on surfaces of genus equal or greater than two, with a finite number of singularities, and with no saddle connections. We recall that a saddle connection is a flow orbit which joints two (not necessarily distinct) saddles. In case when the orbit joints the same saddle, the saddle connection is called a \emph{loop saddle connection}.


Consider a closed 1-form $\omega$ on a closed, compact, orientable surface of genus $\mathbf{g}$. Since $\omega$ is closed, it is locally equal to $dH$ for some real-valued function $H$. The flow associated to $\omega$ is locally given by the solutions of the system of differential equations $\dot{x}=\pd{H}{y}$, $\dot{y}=-\pd{H}{x}$. Assume that this flow has a finite number of nondegenerate critical points and that there are no saddle connections. Flows generated by such forms were shown to be minimal by A.~G.~Mayer in~\cite{mayer43}. Moreover, they are isomorphic to special flows over interval exchange transformations of $4\mathbf{g}-4$ intervals on a circle - a closed curve on the surface transversal to the flow. The roof function is smooth, except for a finite number of points (which are the first intersections of the backward orbits of the singularities of the flow with the transversal), where it has logarithmic singularities. The set of such points coincides with the discontinuities of the interval exchange on the \emph{circle} (see the left part of Figure~\ref{figkolko}). For more information on representing flows this way see Section 1.1. in~\cite{Zorich94u}, for the calculations in the case of a torus, see Section 4 in~\cite{Arnold91} and in the general case see Section 3 in~\cite{Kochergin76}.
\begin{figure}[ht]
\centering
\includegraphics[height=150pt]{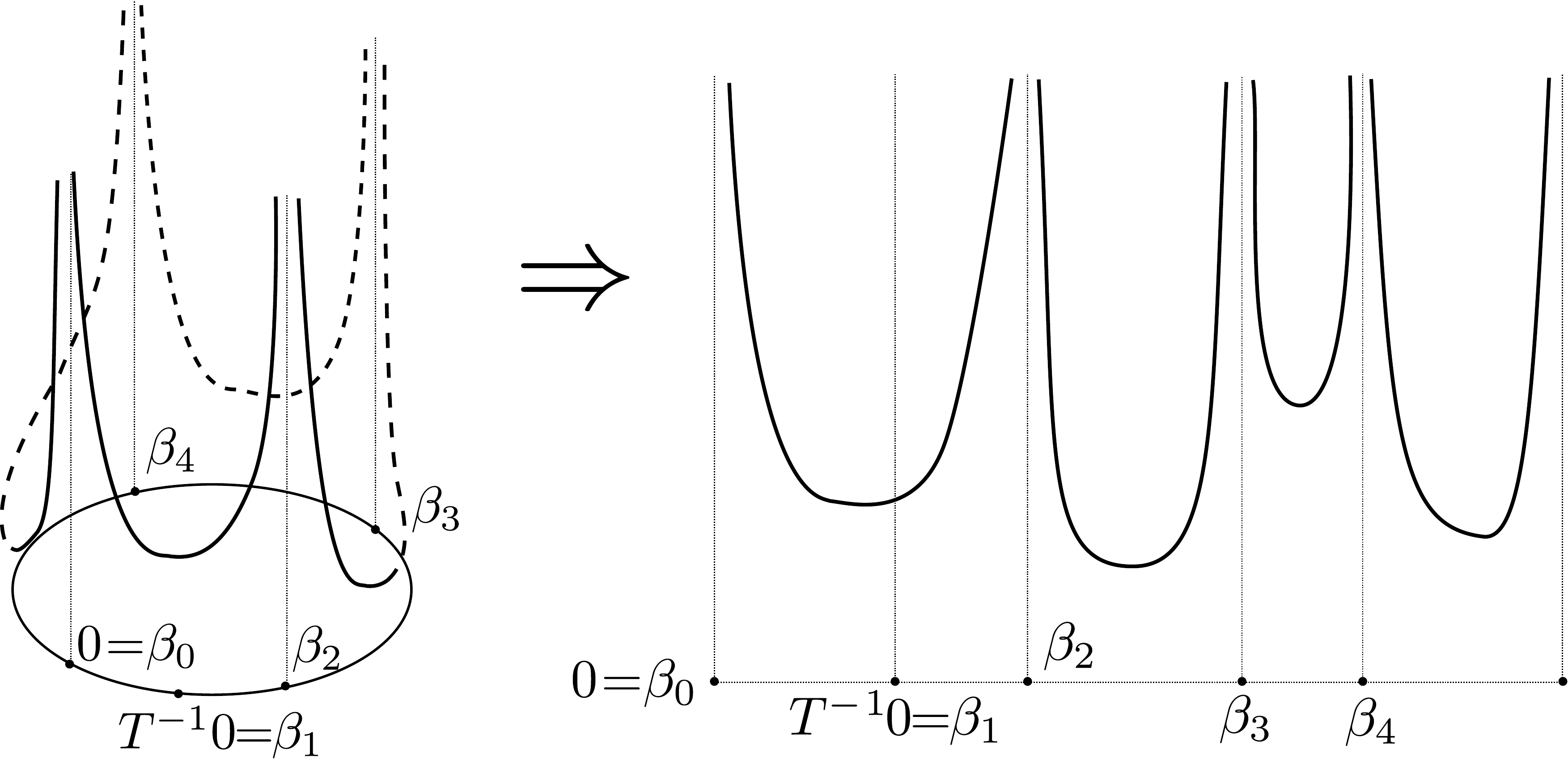}
\caption{Opening the closed transversal}
\label{figkolko}
\end{figure}

In order to use some properties of the IETs on the interval $[0,1)$, we proceed as in Remark~\ref{uw:ietandIETC} (see Figures~\ref{ietandIETC} and~\ref{figkolko}). This results in that one of the discontinuities of the IET (the point which is mapped to $0$ by the IET) is not a discontinuity of the roof function. Both one-sided limits at this point are finite and equal.
It is also reflected in the formula for the roof function which is of the form $f+g$, where $f$ is given by
\begin{equation}\label{fczyste}
f(x)=\sum_{0\leq i\leq r-1}\left(-c^{+}_{i}\log\{x-{\beta}_{i}\}\right)+\sum_{0\leq i\leq r-1} \left(-c^{-}_{i+1}\log\{{\beta}_{i+1}-x\}\right),
\end{equation}
where ${\beta}_i$ for $0\leq i\leq r-1$ are the discontinuity points of the interval exchange transformation $T$ on the interval and $g$ is piecewise absolutely continuous (it is continuous whenever $f$ is so), such that $\min (f+g)>0$. The function $g$ can be represented as a sum $g=g_1+g_2+g_3$, where $g_1$ is absolutely continuous with $g_1(0)=\lim_{x \to 1}g_1(x)$, $g_2$ is linear and $g_3$ is piecewise constant and is continuous whenever $g$ is so. The constants $c_{i+1}^-$, $c_i^+$ for $0\leq i\leq r-1$ are positive, except for $c_{i_0}^{+}=c_{i_0+1}^{-}=0$, where $i_0=\pi_0\circ\pi_1^{-1}(r)=\pi_0\circ \pi_1^{-1}(1)-1$ ($\pi_0$ and $\pi_1$ are the combinatorial data defining $T$, for the definition see Section~\ref{przek}). 

Moreover, since the flow has no saddle connections we have $c_{i}^{+}=c_{i}^{-}$ for $1\leq i \leq r-1$ and $c_0^{+}=c_r^-$ in the definition of $f$. Therefore
\begin{equation}\label{symetry}
\sum_{i=0}^{r-1}c_i^+=\sum_{i=0}^{r-1}c_{i+1}^{-}.
\end{equation}

If the condition~\eqref{symetry} is satisfied for the roof function which is of the form $f+g$ (with $f$ given by~\eqref{fczyste} and $g$ as above), the roof function is said to have \emph{logarithmic singularities of symmetric type} (otherwise they are cold \emph{asymmetric}). All results from Section~\ref{rigidity} hold for the roof function with singularities of \emph{both symmetric and asymmetric} type. In Section~\ref{selfsim} we need to assume that the singularities are of \emph{symmetric} type.\footnote{Singularities of asymmetric type may appear if we admit loop saddle connections.}

To keep the notation as simple as possible, in the remainder of the paper we will additionally assume that $c_{i_0}^+$ and $c_{i_0+1}^{-}$ are strictly positive and we will deal with IETs on the interval $[0,1)$. All the results remain true (with notational changes only) for IETs on the circle which corresponds to the fact that $c_{i_0}^+=c_{i_0+1}^{-}=0$ (see also Remark~\ref{kolkoodcinek}).

\section{Absence of partial rigidity}\label{rigidity}

\subsection{Main result and outline of the proof}
The main result of this section is the following.
\begin{tw}\label{tw12}
Let $T\colon[0,1)\to[0,1)$ be an IET with discontinuity points $0=\beta_0<\beta_1<\dots<\beta_{r-1}<\beta_r=1$ and balanced partition lengths with constant $c>0$. Let
\begin{equation}
f(x)=\sum_{0\leq i\leq r-1}\left(-c^{+}_{i}\log\{x-\beta_{i}\}\right)+\sum_{0\leq i\leq r-1} \left(-c^{-}_{i+1}\log\{\beta_{i+1}-x\}\right),
\end{equation}
where $c_{i+1}^{-}, c_{i}^{+}> 0$ for $0\leq i\leq r-1$ and $g$ is a piecewise absolutely continuous function which is always continuous whenever $f$ is continuous and satisfies the condition $\min (f+g)>0$. Then the special flow $T^{f+g}$ over $T$ under $f+g$ is not partially rigid.
\end{tw}

Our main tool to prove Theorem~\ref{tw12} will be the following lemma which gives a necessary condition for a special flow to be partially rigid.
\begin{lm}[\cite{FL06}]\label{lm5}
Let $T\colon (X,\mathcal{B},\mu) \to (X,\mathcal{B},\mu)$ be an ergodic automorphism and $f\in L^1(X,\mu)$ be a positive function such that $f\geq C > 0$. Suppose that the special flow $T^f$ is partially rigid along the sequence $\{t_n\}_{n\in\mathbb{N}}$ $(t_n \to +\infty)$. Then there exists $0<u\leq 1$ such that for every $0<\varepsilon<C$ we have
\begin{equation}\label{eq:ve}
\liminf_{n \to \infty} \mu \{ x \in X \colon \left( \exists\ {j\in \mathbb{N}}\right) |f^{(j)}(x)-t_n|<\varepsilon\}\geq u.
\end{equation}
\end{lm}
\qed 

Before going into detail let us give the outline of the proof of Theorem~\ref{tw12}. The roof function of the special flow we deal with is a sum of $f$ and $g$. These two functions are of a very different character and this is why we deal with them separately. 

We begin by considering the function $f$ only (i.e. we act as if $g\equiv 0$). In order to apply Lemma~\ref{lm5}, we show first that arbitrary big proportion (less than one) of points from each continuity interval for the base transformation is such that the derivative $f'^{(j)}$ is large enough (see Lemma~\ref{lm3}). The most important property used in the proof of Lemma~\ref{lm3} is that the interval exchange transformation in the base has balanced partition lengths. This will allow us later (in the proof of Lemma~\ref{lm9}) to conclude that the condition~\eqref{eq:ve} does not hold.

What we do next is to perturb the roof function $f$. Every absolutely continuous function $g$ on $[0,1)$ can be decomposed into the sum $g_1+g_2+g_3$, where $g_1$ is absolutely continuous with $g_1(0)=\lim_{x \to 1}g_1(x)$, $g_2$ is linear and $g_3$ is piecewise constant and is continuous whenever $g$ is. 

A perturbation by a linear function has no influence on the claim of Lemma~\ref{lm3} due to Remark~\ref{uw:4.1}. Moreover, Lemma~\ref{wn8} will allow us later (in the proof of Lemma~\ref{lm9}) to conclude that a perturbation by an arbitrary absolutely continuous function doesn't change the situation either.

The next step is to construct a partition of the interval $[0,1)$ (see Lemma~\ref{partycja}). In the proof of Lemma~\ref{lm9} we will work with each subinterval of this partition separately. The situation in each of these subintervals is presented in Figure~\ref{uklad}. As we can see in the figure, the functions $f^{(j)}$ whose graphs cross the $2\varepsilon$-strip around $t$ can be divided into two groups: we treat separately the function which is ``in the middle'' (denoted with a solid line in the figure) and the functions which are at its both sides (denoted with the dashed lines). We apply Lemma~\ref{lm3} to the function which is ``in the middle'' to see that there is ``not too much of it'' in the $\varepsilon$-strip around $t$. The functions which are at its sides cannot fill ''too much'' of the strip either due to convexity (see Figure~\ref{ka} and Lemma~\ref{pochodna}).

\subsection{Technical details}
Let $T$ be an IET with discontinuities $0=\beta_0 < \beta_1 < \dots <\beta_{r-1}<\beta_r=1$. Assume that $T$ fulfills the IDOC and $T$ has balanced partition lengths with constant $c>0$. For $j\geq 1$ consider the partition $\mathcal{P}_j$ (see Definition~\ref{bal}). Denote the partition points in the increasing order by
\begin{equation*}
0=x_0^j<x_1^j<\dots<x_{(r-1)j}^j<1.
\end{equation*}
Let the function $f$ be given by equation~\eqref{fczyste} and $g$ be as described in Section~\ref{reduction}.

Notice that $x_i^j$ ($0 \leq i \leq (r-1)j$) are all discontinuity points of the function ${f}^{(j)}$ ($j\geq 1$). Note also that $f^{(j)}{'}=f'^{(j)}$ whenever both derivatives are well-defined and the sets of discontinuity points of the functions $f'^{(j)}$ and $f^{(j)}$ are equal. Now we will study some basic properties of the function $f'$.
For $j\geq 1$ and $0 \leq i \leq (r-1)j$ let $\Delta^j_{i}=(x^j_i,x^j_{i+1})$ and $d^j_i=x^j_{i+1}-x^j_i$.\label{tajk}
\begin{lm}\label{pochodna}
For every $j\geq 2$ and $0\leq i \leq (r-1)j$ the function $f'^{(j)}$ is strictly increasing on $\Delta_i^j$ with $\lim_{{x\to x_i}^{+}} f'(x)=-\infty$, $\lim_{{x\to x_{i+1}}^{-}} f'(x)=\infty$. The same holds for $f'''^{(j)}(x)$.
\end{lm}
\begin{proof}
We prove the statement by induction. The same arguments remain valid for both $h=f'$ and $h=f'''$. Since $h^{(j+1)}(x)=h^{(j)}(Tx)+h(x)$, the set of discontinuities for $h^{(j+1)}$ consists of two parts: the discontinuities of $h$ and the discontinuities of $h^{(j)}$ iterated backwards one time. The conclusion follows directly from the following two observations:
\begin{itemize}
\item
for every $0\leq i \leq r-1$ the function $h(x)$ is increasing on the interval $[{\beta}_i,{\beta}_{i+1})$,
\item
$\lim_{{x\to {\beta}_i}^{+}} h(x)=-\infty$ and $\lim_{{x\to {\beta}_{i+1}}^{-}} h(x)=\infty$.
\end{itemize}
\end{proof}
\begin{uw}\label{uw:4.1}
If we replace $f$ with $f+g_2$ (where $g_2$ is linear), the assertion of the above lemma remains true. 
\end{uw}
\begin{lm}\label{lm3}
For every $0<\eta<1$ there exists $\delta>0$ such that for every $j\geq 6c^2$, $0\leq i\leq (r-1)j$
\begin{equation*}
m(\{x \in \Delta^j_i \colon |f'^{(j)}(x)|>\delta j \})>\eta d^j_i.
\end{equation*}
\end{lm}
Speaking less formally, we claim that for $j$ large enough on any positive proportion of the interval $\Delta^j_i$ the absolute value of the derivative of $f^{(j)}$, i.e. $\left|f'^{(j)}\right|$, is larger than $\delta j$ for some $\delta>0$.
\begin{proof}
Take $0<\eta<1$. Recall that $T$ has balanced partition lengths and therefore
\begin{equation}\label{bpl}
\frac{1}{cj}\leq d^j_i \leq \frac{c}{j}
\end{equation}
for every $j\geq 1$ and $0\leq i\leq (r-1)j$. Let $M>\max\{\frac{2}{1-\eta},c^2\}$. Put $\delta=\frac{\sum_{i=0}^{r-1}(c^{+}_{i}+c^{-}_{i+1})}{4Mc^3}$. Fix $j\geq 6c^2$ and $0\leq i\leq (r-1)j$. Choose $x_0\in\Delta_i^j$ satisfying $f'^{(j)}(x_0)=0$. We claim that
\begin{equation}\label{e1}
\left|f'^{(j)}(x)\right|>\delta j\mbox{ for }x \in \Delta^j_i\mbox{ such that }\left|x-x_0\right|\geq \frac{d^j_i}{M}.
\end{equation}
Without loss of generality, we will conduct the proof only for $x>x_0$. Since $f'^{(j)}$ is increasing on $\Delta_i^j$, it is enough to show that $f'^{(j)}\left(x_0+\frac{d_i^j}{M}\right)>\delta j$, provided that $x_0+\frac{d_i^j}{M}\in \Delta_i^j$. Let $\overline{x}=x_0+\frac{d_i^j}{M}$. If $\overline{x}\notin\Delta_i^j$ then~\eqref{e1} is trivial. Suppose that $\overline{x}\in\Delta_i^j$. 

We will estimate now $f'^{(j)}(\overline{x})=f'^{(j)}(\overline{x})-f'^{(j)}(x_0)$ from below. Let $0\leq k_0\leq j-1$ be such that $T^{-k_0}\beta_{i_0}$ is the left end of the interval $\Delta_i^j$. Since 
\begin{equation*}
T^{-k}\beta_i \notin [T^{-k_0}\beta_{i_0},\overline{x}] \text{ for }0\leq k\leq j-1\text{ and }0\leq i\leq r-1,
\end{equation*}
it follows that 
\begin{equation}\label{wszystkie}
\beta_i \notin T^k([T^{-k_0}\beta_{i_0},\overline{x}])\text{ for }0\leq k\leq j-1\text{ and }0\leq i\leq r-1.
\end{equation}
Thus $T^k$ is a translation on $[T^{-k_0}\beta_{i_0},\overline{x}]$ for $0\leq k\leq j-1$, i.e.
\begin{equation}\label{takiesame}
T^kx=x+c_k \text{ for } x\in [T^{-k_0}\beta_{i_0},\overline{x}]
\end{equation}
for some constants $c_k$ ($0\leq k\leq j-1$). Therefore the lengths of the intervals of the three partitions by the sets of points $\{T^{k-k_0}\beta_{i_0}\colon 0\leq k\leq j-1\}$, $\{T^kx_0\colon 0\leq k\leq j-1\}$ and $\{T^k\overline{x}\colon 0\leq k\leq j-1\}$ are the same, except for the leftmost and rightmost intervals. The length of the leftmost and rightmost intervals of the two latter partitions can be estimated from above by $2\max \mathcal{P}(\{T^{k-k_0}\beta_{i_0}\colon 0\leq k\leq j-1\})$. Hence, in view of the inequalities~\eqref{bpl}, if $T^{k_1}x_0<T^{k_2}x_0$ then
\begin{align*}
T^{k_2}x_0-T^{k_1}\overline{x}=&(T^{k_2}x_0-T^{k_1}x_0)-(T^{k_1}\overline{x}-T^{k_1}x_0) \geq \frac{1}{cj}-\frac{d^j_i}{M}>\frac{1}{cj}-\frac{d^j_i}{c^2}
\\
=&\frac{1}{c^2}\left(\frac{c}{j}-d^j_i\right)\geq 0,
\end{align*}
as we have chosen $M>c^2$. This means that the interval $[x_0,\overline{x}]$ and its $j-1$ consecutive iterations by $T$ are pairwise disjoint.

Let $\overline{f}\colon[0,1)\to\mathbb{R}$ be given by
\begin{equation*}
\overline{f}(x)=\sum_{i=0}^{r-1}\chi_{[\beta_i,\beta_{i+1})}\left(-c^{+}_{i}\log\{x-\beta_{i}\}-c^{-}_{i+1}\log\{\beta_{i+1}-x\}\right).
\end{equation*}
\begin{figure}[ht]
\centering
\includegraphics[height=300pt]{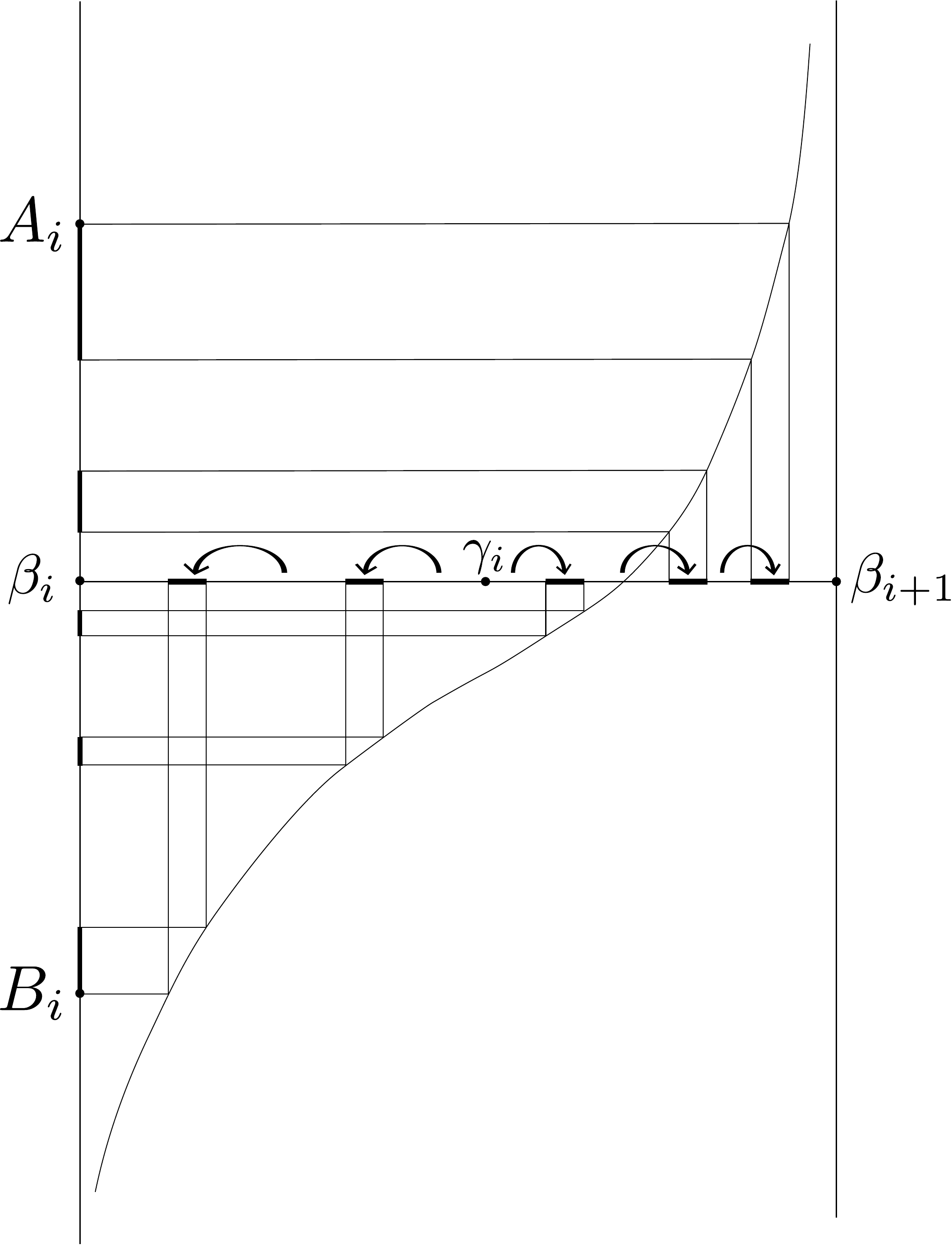}
\caption{The iterations of $[x_0,\overline{x}]$ and the gaps in $(\beta_i,\beta_{i+1})$}
\label{rys1}
\end{figure}

For $0\leq i\leq r-1$ put
\begin{align*}
A_i=&\max\left\{f'(T^k\overline{x})\colon 0\leq k \leq j-1\mbox{ such that }T^k\overline{x}\in[\beta_i,\beta_{i+1})\right\},
\\
B_i=&\min\left\{f'(T^kx_0)\colon 0\leq k \leq j-1\mbox{ such that }T^kx_0\in[\beta_i,\beta_{i+1})\right\},
\end{align*}
i.e. $A_i$ is the image via $f'$ of the right end of the rightmost interval among $T^k[x_0,\overline{x}]$ for $0\leq k\leq j-1$ such that $T^k[x_0,\overline{x}]\subset [\beta_i,\beta_{i+1})$ and $B_i$ is the image via $f'$ of the left end of the rightmost interval among $T^k[x_0,\overline{x}]$ for $0\leq k\leq j-1$ such that $T^k[x_0,\overline{x}]\subset [\beta_i,\beta_{i+1})$  (see Fig.~\ref{rys1}). Moreover, for $0\leq i\leq r-1$ put
\begin{align*}
\overline{A}_i=&\max\left\{\overline{f}'(T^k\overline{x})\colon 0\leq k \leq j-1\mbox{ such that }T^k\overline{x}\in[\beta_i,\beta_{i+1})\right\},
\\
\overline{B}_i=&\min\left\{\overline{f}'(T^kx_0)\colon 0\leq k \leq j-1\mbox{ such that }T^kx_0\in[\beta_i,\beta_{i+1})\right\}.
\end{align*}
Fix $0\leq i\leq r-1$. We claim that
\begin{equation}\label{posredni}
\overline{B}_i\leq \overline{f}'\left(\beta_i+\frac{2c}{j}\right).
\end{equation}
\begin{figure}[ht]
\centering
\includegraphics[height=110pt]{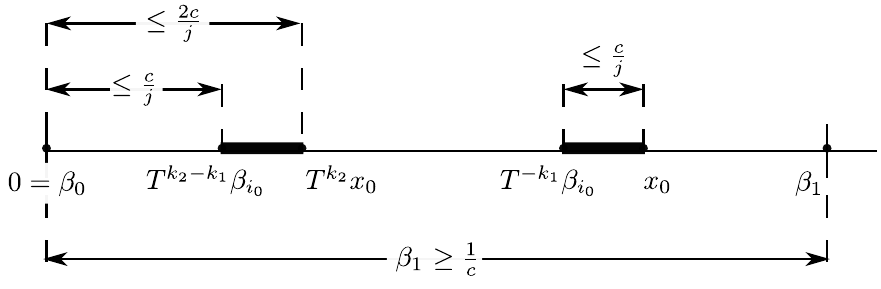}
\caption{$i=0$}
\label{izero}
\end{figure}
Indeed, since $\overline{f}'$ is increasing, it suffices to show that there exists $0\leq k\leq j-1$ such that $T^kx_0\in [\beta_i,\beta_{i+1})$ and $T^kx_0\leq \beta_i+\frac{2c}{j}$. Consider first the case when $i=0$ (see Fig.~\ref{izero}). Let $0\leq k_1 \leq j-1$ and $i_0$ be such that
\begin{equation*}
T^{-k_1}\beta_{i_0}=\max \{T^{-k}\beta_i\colon T^{-k}\beta_i<x_0,\ 0\leq k\leq j-1,\ 0\leq i\leq r-1 \}.
\end{equation*}
\begin{figure}[ht]
\centering
\includegraphics[width=300pt]{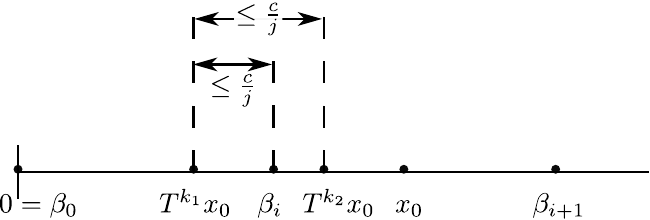}
\caption{$i>0$}
\label{iplus}
\end{figure}
Since $T$ has balanced partition lengths, $x_0-T^{-k_1}\beta_{i_0}\leq \frac{c}{j}$. For the same reason, there exists $0\leq k_2\leq j-1$ such that $T^{k_2-k_1}\beta_{i_0}\leq \frac{c}{j}$. Hence $T^{k_2}x_0\leq \frac{2c}{j}$. Moreover, we have $j\geq 6c^2 >2c^2$, so $\frac{1}{c}>\frac{2c}{j}$ and therefore $T^{k_2}x_0\in [0,\beta_1)$. Consider now the case when $i>0$ (see Fig.~\ref{iplus}). Let $0\leq k_1\leq j-1$ be such that
\begin{equation*}
T^{k_1}x_0=\max\{T^{k}x_0<\beta_i\colon 0\leq k\leq j-1\}
\end{equation*}
and let $0\leq k_2\leq j-1$ be such that
\begin{equation*}
T^{k_2}x_0=\min\{T^{k}x_0>\beta_i\colon 0\leq k\leq j-1\}.
\end{equation*}
Then
\begin{equation*}
T^{k_2}x_0-\beta_i \leq T^{k_2}x_0-T^{k_1}x_0\leq \frac{c}{j}<\frac{2c}{j},
\end{equation*}
where the middle inequality follows from the remarks after~\eqref{takiesame}. The inequality~\eqref{posredni} is therefore proved. Hence
\begin{align}\label{yksi}
\overline{B}_i&\leq \overline{f}'\left(\beta_{i}+\frac{2c}{j}\right)=-\frac{c_i^+}{\beta_{i}+\frac{2c}{j}-\beta_{i}}+\frac{c_{i+1}^-}{\beta_{i+1}-\left(\beta_{i}+\frac{2c}{j}\right)} \nonumber
\\
&=-\frac{c_{i}^+}{\frac{2c}{j}}+\frac{c_{i+1}^-}{\beta_{i+1}-\beta_i-\frac{2c}{j}}.
\end{align}
In a similar way we obtain
\begin{align}\label{kaksi}
\overline{A}_i&\geq \overline{f}'\left(\beta_{i+1}-\frac{2c}{j}\right)=-\frac{c_i^+}{\beta_{i+1}-\frac{2c}{j}-\beta_{i}}+\frac{c_{i+1}^-}{\beta_{i+1}-\left(\beta_{i+1}-\frac{2c}{j}\right)} \nonumber
\\
&=-\frac{c_i^+}{\beta_{i+1}-\beta_i-\frac{2c}{j}}+\frac{c_{i+1}^-}{\frac{2c}{j}}.
\end{align}
Recall that $\beta_{i+1}-\beta_i\geq \frac{1}{c}$ (this follows from~\eqref{bpl} for $j=1$). Since $j\geq 6c^2$ implies $\frac{1}{c}-\frac{2c}{j}>0$, we have
\begin{equation*}
\overline{A}_i\geq -\frac{c_i^+}{\frac{1}{c}-\frac{2c}{j}}+\frac{c_{i+1}^-}{\frac{2c}{j}}\text{ and }
\overline{B}_i\leq -\frac{c_i^+}{\frac{2c}{j}}+\frac{c_{i+1}^-}{\frac{1}{c}-\frac{2c}{j}}.
\end{equation*}
Therefore
\begin{multline}\label{kolme}
\overline{A}_i-\overline{B}_i\geq -\frac{c_i^+}{\frac{1}{c}-\frac{2c}{j}}+\frac{c_{i+1}^-}{\frac{2c}{j}}+\frac{c_i^+}{\frac{2c}{j}}-\frac{c_{i+1}^-}{\frac{1}{c}-\frac{2c}{j}}
\\
=(c_i^{+}+c_{i+1}^{-})\left( \frac{j}{2c}-\frac{1}{\frac{1}{c}-\frac{2c}{j}}\right)\geq\frac{c_i^{+}+c_{i+1}^-}{4}\frac{j}{c},
\end{multline}
where the last inequality follows from the assumption that $j\geq 6c^2$. Note that $f'-\overline{f}'$ is increasing on $(\beta_i,\beta_{i+1})$ and therefore
\begin{equation*}
(A_i-\overline{A}_i)-(B_i-\overline{B}_i)>0.
\end{equation*}
Indeed, since $f'$ and $\overline{f}'$ are both increasing on $(\beta_i,\beta_{i+1})$, the maximal value among $f'(T^k\overline{x})$ ($0\leq k\leq j-1$) and the maximal value among $\overline{f}'(T^k\overline{x})$ ($0\leq k\leq j-1$) are obtained for the same argument $T^k\overline{x}$. The same applies to the minima in the definition of $B_i$ and $\overline{B}_i$. Hence
\begin{equation*}
A_i-B_i>\overline{A}_i-\overline{B}_i\geq\frac{c_i^{+}+c_{i+1}^-}{4}\frac{j}{c}.
\end{equation*}
Adding the inequalities for $0\leq i\leq r-1$, we conclude that
\begin{equation}\label{dwa}
\sum_{i=0}^{r-1}(A_i-B_i)>\frac{1}{4}\sum_{i=0}^{r-1}(c^{+}_i+c^{-}_{i+1})\frac{j}{c}.
\end{equation}
If $0\leq k_1,k_2\leq j-1$ satisfy $\beta_i<T^{k_1}x_0<T^{k_2}x_0<\beta_{i+1}$ and for $0\leq k\leq j-1$ we have $T^kx_0\notin(T^{k_1}x_0,T^{k_2}x_0)$, we say that $(T^{k_1}\overline{x},T^{k_2}x_0)$ is a \emph{gap} in $(\beta_i,\beta_{i+1})$. Since $T$ has balanced partition lengths, in view of~\eqref{takiesame} and~\eqref{bpl} we obtain an upper bound for the lengths of the gaps:
\begin{equation}\label{cztery}
T^{k_2}x_0-T^{k_1}\overline{x} \leq \frac{c}{j}-\frac{d^j_i}{M}=\frac{1}{Mcj}Mc^2-\frac{d^j_i}{M}\leq \frac{d^j_i}{M}(Mc^2-1)
\end{equation}
for any gap $(T^{k_1}\overline{x},T^{k_2}x_0)$.

Fix again $0\leq i\leq r-1$. From Lemma~\ref{pochodna} it follows that the function $f'^{(j)}$ has one inflection point in the interval $(\beta_i,\beta_{i+1})$. Denote it by $\gamma_i$. To each gap in $(\beta_i,\beta_{i+1})$ assign one of the iterations $T^{k}[x_0,\overline{x}]$ in the following way (see Fig.~\ref{rys1}). Consider the gap $(T^{k_1}\overline{x},T^{k_2}x_0)$. There are three cases:
\begin{itemize}
\item
if $T^{k_1}x_0\geq \gamma_i$, we assign $[T^{k_2}x_0,T^{k_2}\overline{x}]$ to the gap $(T^{k_1}\overline{x},T^{k_2}x_0)$,
\item
if $T^{k_1}\overline{x}\leq \gamma_i$, we assign $[T^{k_1}x_0,T^{k_1}\overline{x}]$ to the gap $(T^{k_1}\overline{x},T^{k_2}x_0)$,
\item
if $\gamma_i\in(T^{k_1}\overline{x},T^{k_2}x_0)$ then we split the gap and assign $[T^{k_1}x_0,T^{k_1}\overline{x}]$ to $(T^{k_1}\overline{x},\gamma_i)$ and $[T^{k_2}x_0,T^{k_2}\overline{x}]$ to $(\gamma_i,T^{k_2}x_0)$.
\end{itemize}
If $\gamma_i\in[T^{k}x_0,T^{k}\overline{x}]$ for some $0\leq k\leq j-1$ then $[T^{k}x_0,T^{k}\overline{x}]$ is not assigned to any gap.

From~\eqref{cztery} it follows that the ratio of the length of each gap to the length of the interval which is assigned to it can be estimated from above by $c^2M-1$:
\begin{itemize}
\item
for $(T^{k_1}\overline{x},T^{k_2}x_0)$ with $T^{k_1}x_0\geq \gamma_i$ we have $T^{k_2}x_0-T^{k_1}\overline{x}\leq\frac{d^j_i}{M}(Mc^2-1)=(T^{k_2}x_0-T^{k_2}\overline{x})(c^2M-1)$,
\item
for $(T^{k_1}\overline{x},T^{k_2}x_0)$ with $T^{k_2}\overline{x}\leq \gamma_i$ we have
$T^{k_2}x_0-T^{k_1}\overline{x}\leq\frac{d^j_i}{M}(Mc^2-1)=(T^{k_1}x_0-T^{k_1}\overline{x})(c^2M-1)$,
\item
for $(T^{k_1}\overline{x},T^{k_2}x_0)$ with $\gamma_i\in(T^{k_1}\overline{x},T^{k_2}x_0)$ we have
$\gamma_i-T^{k_1}\overline{x}\leq \frac{d^j_i}{M}(Mc^2-1)=(T^{k_1}\overline{x}-T^{k_1}x_0)(c^2M-1)$ and $T^{k_2}x_0-\gamma_i\leq \frac{d^j_i}{M}(Mc^2-1)=(T^{k_2}\overline{x}-T^{k_2}x_0)(c^2M-1)$.
\end{itemize}

Let $a_k=f'(T^k\overline{x})-f'(T^kx_0)$ ($0\leq k\leq j-1$). Since $f'^{(j)}$ is concave on $(\beta_i,\gamma_i)$ and convex on $(\gamma_i,\beta_{i+1})$, the length of the image by $f'$ of each gap can be estimated from above by $a_k(c^2M-1)$ with $k$ chosen according to the assignment described previously, i.e.
\begin{itemize}
\item
for $(T^{k_1}\overline{x},T^{k_2}x_0)$ with $T^{k_1}x_0\geq \gamma_i$ we have $f'(T^{k_2}x_0)-f'(T^{k_1}\overline{x})\leq a_{k_2}(c^2M-1)$,
\item
for $(T^{k_1}\overline{x},T^{k_2}x_0)$ with $T^{k_2}\overline{x}\leq \gamma_i$ we have $f'(T^{k_2}x_0)-f'(T^{k_1}\overline{x})\leq a_{k_1}(c^2M-1)$,
\item
for $(T^{k_1}\overline{x},T^{k_2}x_0)$ with $\gamma_i\in(T^{k_1}\overline{x},T^{k_2}x_0)$ we have $f'(T^{k_2}x_0)-f'(\gamma_i)\leq a_{k_2}(c^2M-1)$ and $f'(\gamma_i)-f'(T^{k_1}\overline{x})\leq a_{k_1}(c^2M-1)$.
\end{itemize}
This means that the sum of the lengths of the intervals $[A_i,B_i]$, which is the sum of the images of the intervals $[T^kx_0,T^k\overline{x}]\in (\beta_i,\beta_{i+1})$ for $0\leq k\leq j-1$ and the images of the gaps between them in each $(\beta_i,\beta_{i+1})$, can be estimated from above as follows:
\begin{equation}\label{piec}
\sum_{i=0}^{r-1} (A_i-B_i)\leq\sum_{k=0}^{j-1}a_k + \sum_{k=0}^{j-1}a_k(c^2M-1)=\sum_{k=0}^{j-1}a_kc^2M.
\end{equation}

From~\eqref{wszystkie} it follows that $\beta_i\notin [T^{k}x_0,T^k\overline{x}]$ for $0\leq k\leq j-1$, so by~\eqref{dwa} and~\eqref{piec} we obtain
\begin{align*}
f'^{(j)}(\overline{x})&=f'^{(j)}(\overline{x})-f'^{(j)}(x_0)=\sum_{k=0}^{j-1}\left(f'(T^k\overline{x})-f'(T^kx_0)\right)=\sum_{k=0}^{j-1}a_k 
\\
&\geq \frac{\sum_{i=0}^{r-1}\left(A_i-B_i\right)}{c^2M}>\frac{\frac{1}{4}\sum_{i=0}^{r-1}\left(c^{+}_{i}+c^{-}_{i+1}\right)\frac{j}{c}}{c^2M}
=\frac{\sum_{i=0}^{r-1}(c_{i}^{+}+c^{-}_{i+1})}{4Mc^3}j=\delta j.
\end{align*}
and condition~\eqref{e1} indeed holds. Since $M>\frac{2}{1-\eta}$, this means that 
\begin{equation*}
m\left(\left\{x \in \Delta^j_i \colon |f'^{(j)}(x)|>\delta j \right\}\right)>\left(1-\frac{2}{M}\right)d^j_i>\eta d^j_i
\end{equation*}
and the proof is complete.
\end{proof}
\begin{uw}\label{uw:4.2}
We claim that the assertion of the above lemma remains true if we replace $f$ with $f+g_2$ where $g_2$ is linear. Indeed, notice that throughout the proof we have mostly used properties of $f$ which are not affected by adding a linear function to $f$  such as (piecewise) monotonicity or convexity. The only places where we needed an explicit formula for the considered function were~\eqref{yksi} and~\eqref{kaksi}. The estimates of $\overline{A}_i$ and $\overline{B}_i$ are clearly different for $f+g_2$ in place of $f$. However, what is used in the remainder of the proof is~\eqref{kolme} which stays unchanged: to adjust the proof for $f+g_2$ we need to add the same value to the ``new'' $\overline{A}_i$ and $\overline{B}_i$ which cancels out in~\eqref{kolme}.
\end{uw}
\begin{lm}\label{lm4}
Let $\mathcal{H}=\{h_{\alpha}\colon \alpha\in\mathcal{A}\}$ be a family of monotonic, differentiable, convex functions $h_{\alpha}\colon(a_{\alpha},b_{\alpha})\to\mathbb{R}$. Suppose that
\begin{multline}\label{six}
\left(\forall\ {0<\eta<1}\right)\ \left(\exists\ {\widetilde{\delta}>0}\right)\ \left(\forall\ {\alpha\in\mathcal{A}}\right)\\
 \ m\left\{x \in (a_{\alpha},b_{\alpha}) \colon |h_{\alpha}'(x)|>\frac{\widetilde{\delta}}{b_{\alpha}-a_{\alpha}}\right\}>\eta (b_{\alpha}-a_{\alpha}).
\end{multline}
Then
\begin{multline*}
\left( \forall\ {0<\widetilde{\eta}<1}\right)\ \left(\exists\ {\varepsilon>0}\right)\ \left(\forall\ {t>0}\right)\ \left(\forall\ {\alpha\in\mathcal{A}}\right)\\
\ m\left\{x \in (a_{\alpha},b_{\alpha}) \colon |h_{\alpha}(x)-t|<2\varepsilon\right\} \leq \widetilde{\eta}(b_{\alpha}-a_{\alpha}).
\end{multline*}
\end{lm}
\begin{proof}
Fix $0<\widetilde{\eta}<1$. Take $1-\widetilde{\eta}<\eta<1$ and $L>0$ such that $\frac{2\eta}{L}+1-\eta\leq\widetilde{\eta}$. Take $\varepsilon<\widetilde{\delta}\frac{\eta}{2L}$, where $\widetilde{\delta}$ is as in the condition~\eqref{six}. Fix $\alpha\in\mathcal{A}$. Let $t>0$ and
\begin{align*}
A_{\alpha}&=\{x\in(a_{\alpha},b_{\alpha})\colon |h_{\alpha}(x)-t|<2\varepsilon\},
\\
B_{\alpha}&=\left\{x\in(a_{\alpha},b_{\alpha})\colon \left| h_{\alpha}'(x) \right| >\frac{\widetilde{\delta}}{b_{\alpha}-a_{\alpha}}\right\}.
\end{align*}
By assumption $m(A_{\alpha}\cap B_{\alpha}^c)\leq (1-\eta)(b_{\alpha}-a_{\alpha})$. Since function $h_{\alpha}$ is convex and monotone, $B_\alpha$ and $A_\alpha$ are intervals, whence also $A_\alpha\cap B_\alpha$ is an interval. Put
\begin{equation*}
x_{1,\alpha}=\inf(A_{\alpha}\cap B_{\alpha}),\ x_{2,\alpha}=\sup(A_{\alpha}\cap B_{\alpha}).
\end{equation*}
From the mean value theorem
\begin{equation*}
|h_{\alpha}(x_{1,\alpha})-h(x_{2,\alpha})|=|h_{\alpha}'(\xi)(x_{2,\alpha}-x_{1,\alpha})|\text{ for some }\xi\in(x_{1,\alpha},x_{2,\alpha}).
\end{equation*}
Hence $\xi \in B_\alpha$ and we obtain
\begin{equation*}
4\varepsilon\geq|h_{\alpha}(x_{2,\alpha})-h(x_{1,\alpha})|=|h_{\alpha}'(\xi)(x_{2,\alpha}-x_{1,\alpha})|>\frac{\widetilde{\delta}}{b_{\alpha}-a_{\alpha}}|x_{2,\alpha}-x_{1,\alpha}|,
\end{equation*}
which implies
\begin{equation*}
|x_{2,\alpha}-x_{1,\alpha}|<\frac{4\varepsilon(b_{\alpha}-a_{\alpha})}{\widetilde{\delta}}<\frac{2\eta(b_{\alpha}-a_{\alpha})}{L}.
\end{equation*}
It follows that
\begin{align*}
m(A_{\alpha})&=m(A_{\alpha}\cap B_{\alpha}^c)+m(A_{\alpha}\cap B_{\alpha})\leq (1-\eta)(b_{\alpha}-a_{\alpha})+\frac{2\eta(b_{\alpha}-a_{\alpha})}{L}
\\
&=\left(1-\eta+\frac{2\eta}{L}\right)(b_{\alpha}-a_{\alpha})\leq\widetilde{\eta}(b_{\alpha}-a_{\alpha}),
\end{align*}
and the proof is complete.
\end{proof}


In the proof of the next lemma we use the same techniques as in~\cite{CFS82} (see Lemma 2, Ch. 16, $\mathcal{x}$3 for $C^1$-functions in the case of rotations) and in~\cite{FL06} (see Lemma 6.1 for absolutely continuous functions in the case of rotations). One of the properties which we will use in the proof is unique ergodicity of the considered interval exchange transformations. In order to show that the IETs we deal with are indeed uniquely ergodic, let us recall first some definitions introduced by M.~A.~Boshernitzan~\cite{Boshernitzan85}.

\begin{df}
Set $A\subset \mathbb{N}=\{1,2,3,\dots\}$ is said to be \emph{essential} if for any $l\geq 2$ there exists $a>1$ such that the system
\begin{equation*} 
\left\{
\begin{array}{ll}
n_{i+1}>2n_i & \text{for } 1\leq i\leq l-1\\
n_l<a\cdot n_1 & \\
n_i\in A & \text{for }1\leq i\leq l 
\end{array} \right.
\end{equation*}
has an infinite number of solutions $(n_1,n_2,\dots,n_l)$.
\end{df}
\begin{df}
We say that an IET $T$ has Property P if for some $\varepsilon>0$ the set $\left\{n\in\mathbb{N}\colon \min \mathcal{P}_n\geq \frac{\varepsilon}{n}\right\}$ is essential.
\end{df}
\begin{tw}\cite{Boshernitzan85}\label{boo}
Let $T$ be a minimal IET which satisfies Property P. Then $T$ is uniquely ergodic.
\end{tw}

\begin{wn}\label{bosher}
Any IET with balanced partition lengths is uniquely ergodic.
\end{wn}
\begin{proof}
The claim follows directly by Theorem~\ref{boo} and by the definition of balanced partition lenghts.
\end{proof}

\begin{lm}\label{wn8}
Let $T \colon [0,1) \to [0,1)$ be an IET of $r$ intervals with balanced partition lengths with constant $c>0$ and let $g \colon [0,1) \to \mathbb{R}$ be an absolutely continuous function such that $\int_0^1g'(x)dx=0$. Then for any $\varepsilon>0$ there exists $N_0>0$ such that for $n> N_0$, all $0\leq i\leq (n-1)r$ and $x,y\in \Delta_i^n$ the inequality $|g^{(n)}(x)-g^{(n)}(y)|<\varepsilon$ holds.
\end{lm}
\begin{proof}
Fix $\varepsilon>0$. We claim that there exists a $C^1$-function $g_{\varepsilon}\colon [0,1) \to \mathbb{R}$ such that
\begin{equation*}
Var(g-g_{\varepsilon})<\frac{\varepsilon}{2([c^2]+1)}
\end{equation*}
and
\begin{equation*}
\left|\int_0^1 g_{\varepsilon}'(x)\ dx \right|<\frac{\varepsilon}{4c}.
\end{equation*}
Indeed, since $g$ is absolutely continuous, there exists $f\in L^1([0,1))$ and $a\in\mathbb{R}$ such that
\begin{equation*}
g(x)=a+\int_0^x f(y)\ dy
\end{equation*}
for $x\in [0,1)$. Let function $f_{\varepsilon}\in C([0,1))$ be such that
\begin{equation*}
\|f-f_\varepsilon \|_{L^1}<\min \left\{\frac{\varepsilon}{2([c^2]+1)},\frac{\varepsilon}{4c} \right\}
\end{equation*}
and let $g_{\varepsilon}(x)=a+\int_0^x f_{\varepsilon}(y)\ dy$ for $x\in [0,1)$. Then indeed
\begin{equation*}
Var(g-g_{\varepsilon})=\int_0^1 |f(x)-f_{\varepsilon}(x)|\ dx<\frac{\varepsilon}{2([c^2]+1)}
\end{equation*}
and
\begin{multline*}
\left| \int_0^1 g'_{\varepsilon}(x) dx \right|=\left|\int_0^1 g'_{\varepsilon}(x)-g(x)\ dx \right|= \left| \int_0^1 f_{\varepsilon}(x)-f(x)\ dx\right|\\
\leq \int_0^1 |f_{\varepsilon}(x)-f(x)|\ dx<\frac{\varepsilon}{4c}.
\end{multline*}

By Corollary~\ref{bosher}, $T$ is uniquely ergodic and therefore
\begin{equation*}
\left| \lim_{n\to\infty}\frac{1}{n}\sum_{j=0}^{n-1}g_{\varepsilon}'(T^jx) \right|=\left| \int_0^1 g_{\varepsilon}'(x)\ dx\right| <\frac{\varepsilon}{4c},
\end{equation*}
where the convergence is uniform with respect to $x$.\footnote{IETs are homeomorphisms of some Cantor sets, see~\cite{MMY05} S.~Marmi, P.~Moussa, J.-C.~Yoccoz.}  In other words, there exists $N_0\in\mathbb{N}$ such that
\begin{multline}\label{wn8g}
\left|\sum_{j=0}^{n-1}g_{\varepsilon}'(T^jx)\right| \leq \left|\sum_{j=0}^{n-1} g_{\varepsilon}'(T^jx) -n\int_0^1 g_{\varepsilon}'(x)\ dx\right|+\left|n\int_0^1 g_{\varepsilon}'(x)\ dx \right|
\\
\leq n\cdot \frac{\varepsilon}{4c}+n\cdot \frac{\varepsilon}{4c}=n\cdot\frac{\varepsilon}{2c}
\end{multline}
for $n>N_0$ and all $x\in [0,1)$. Fix $n>N_0$, let $0\leq i\leq (n-1)r$ and take $x,y\in\Delta_i^n$, $x<y$. For $0\leq j\leq n-1$ we have $|T^jx-T^jy|=|x-y|$, whence
\begin{equation*}
\sum_{j=0}^{n-1}\left(g_{\varepsilon}(T^jx)-g_{\varepsilon}(T^jy)\right)=\int_x^y \sum_{j=0}^{n-1}g_{\varepsilon}'(T^jz)\ dz.
\end{equation*}
Therefore, by~\eqref{wn8g} and by the assumption that $T$ has balanced partition lengths with constant $c$, we obtain
\begin{equation*}
\left| \sum_{j=0}^{n-1}\left( g_{\varepsilon}(T^jx)-g_{\varepsilon}(T^jy) \right) \right| \leq 
\int_x^y \left| \sum_{j=0}^{n-1}g_{\varepsilon}'(T^jz) \right| dz 
\\
\leq n\frac{\varepsilon}{2c}|x-y|\leq n\frac{\varepsilon}{2c}\frac{c}{n}=\frac{\varepsilon}{2}.
\end{equation*}

Let us consider the following family of intervals:
\begin{equation*}
\mathcal{I}=\{[x,y],[Tx,Ty],\dots,[T^{n-1}x,T^{n-1}y]\}.
\end{equation*}
For every $0 \leq i \neq j \leq n-1$, using the assumption that $T$ has balanced partition lengths, we obtain
\begin{equation*}
|T^ix-T^jx|\geq\frac{1}{cn}.
\end{equation*}
Moreover, for $0\leq i\leq n-1$
\begin{equation*}
\left|T^{i}x-T^iy\right|=\left|x-y\right|\leq\frac{c}{n}.
\end{equation*}
It follows that a point from $[0,1)$ belongs to at most $[c^2]+1$ intervals from the family $\mathcal{I}$. Therefore
\begin{multline*}
|(g^{(n)}(x)-g^{(n)}(y))-(g_{\varepsilon}^{(n)}(x)-g_{\varepsilon}^{(n)}(y))|
\\
\leq \sum_{j=0}^{n-1}|(g-g_{\varepsilon})(T^jx)-(g-g_{\varepsilon})(T^jy)|
\leq\sum_{j=0}^{n-1} Var_{[T^jx,T^jy]}(g-g_{\varepsilon})
\\
\leq ([c^2]+1)Var(g-g_{\varepsilon})<\frac{\varepsilon}{2}.
\end{multline*}
Hence
\begin{equation*}
|g^{(n)}(x)-g^{(n)}(y)|\leq |g_{\varepsilon}^{(n)}(x)-g_{\varepsilon}^{(n)}(y)|+\frac{\varepsilon}{2}<\varepsilon
\end{equation*}
which completes the proof.
\end{proof}
\begin{uw}
Notice that for an absolutely continuous function $g\colon [0,1) \to \mathbb{R}$ the conditions $g(0)=\lim_{x\to 1}g(x)$ and $\int_0^1 g'(x)dx=0$ are equivalent. Notice also that the assertion of the above lemma remains true if we replace $g$ with $g_1+g_3$ where $g_1$ is absolutely continuous satisfying $\int_0^1 g_1'(x)dx=0$ and $g_3$ is piecewise constant and continuos whenever the IET is.
\end{uw}

Let $f,g\colon [0,1) \to \mathbb{R}$ be as described in Section~\ref{reduction} (i.e. $f$ is given by the formula~\eqref{fczyste} and $g=g_1+g_2+g_3$ where $g_1$ is absolutely continuous with $g(0)=\lim_{x \to 1}g(x)$, $g_2$ is linear and $g_3$ is piecewise constant and is continuous whenever $g$ is so). Fix 
\begin{equation}\label{hutu}
0<\varepsilon<\frac{1}{3}\min(f+g)
\end{equation}
and let $N_0\in\mathbb{N}$ be as in the assertion of Lemma~\ref{wn8}. Now we will describe a procedure of choosing a partition of the interval $[0,1)$ into $\Delta_1,\dots,\Delta_n$, which depends on the functions $f$ and $g$, on the parameter $t$ and on $\varepsilon$. 
We will use this partition in the proof of the main theorem. To make clear what functions or parameters we mean, we will indicate it in the parentheses: $\Delta_i(f,g,t,\varepsilon)$.

Let
\begin{equation*}
j_0=\max\{j\in\mathbb{N}\colon \left(\exists\ {x\in [0,1)}\right) \left|f^{(j)}(x)+g^{(j)}(x)-t\right|<\varepsilon\}.
\end{equation*}
Since $\min(f+g)>0$, $j_0$ is finite and therefore determines a partition of the interval $[0,1)$ into subintervals $\Delta_0^{j_0},\dots, \Delta_{(r-1)j_0}^{j_0}$ (for the definition of these subintervals see page~\pageref{tajk}). For $0\leq i\leq (r-1)j_0$ set
\begin{equation*}
j_i^{j_0}=\max\left\{j\in\mathbb{N}\colon \left(\exists\ {x\in \Delta_{i}^{j_0}}\right) \left|f^{(j)}(x)+g^{(j)}(x)-t\right|<\varepsilon \right\}\in \mathbb{N} \cup \{-\infty\}
\end{equation*}
(we put $j_i^{j_0}=-\infty$ if the set is empty). Let 
\begin{equation*}
P=\{x_i^{j_0}\colon 0\leq i\leq (r-1)j_0\}.
\end{equation*}
We are interested in the strip $[0,1)\times (t-\varepsilon,t+\varepsilon)$ and this is why the partition determined by $P$ might be too fine for our purposes, i.e. all the functions $f^{(j)}+g^{(j)}$ for $j\geq N$ such that $\left(f^{(j)}+g^{(j)}\right)(\Delta_i)\cap (t-\varepsilon,t+\varepsilon)\neq \emptyset$ might be continuous at the endpoints of $\Delta_i$ for some $i$. Therefore we remove now some points from $P$. The procedure consists of three steps. After each of them, by abuse of notation, we still denote the reduced set of the partition points by the same letter $P$.

\begin{figure}[ht]
\centering
\includegraphics[width=300pt]{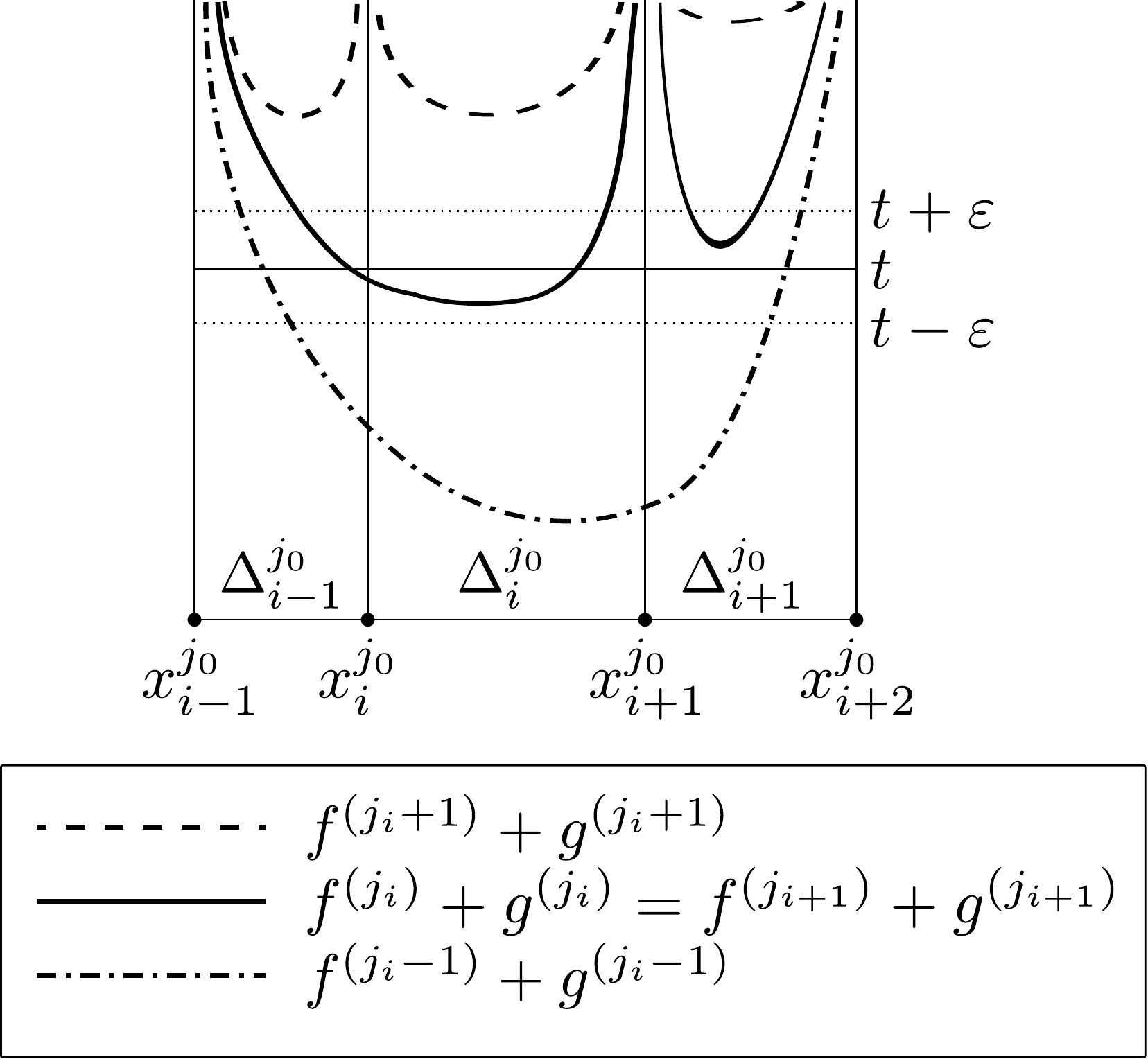}
\caption{Assume $j_{i-1}=j_i=j_{i+1}$. \emph{Step 1.} Remove $x_i^{j_0}$. Don't remove $x_{i+1}^{j_0}$}
\label{step1-1}
\end{figure}
\begin{figure}[ht]
\centering
\includegraphics[width=300pt]{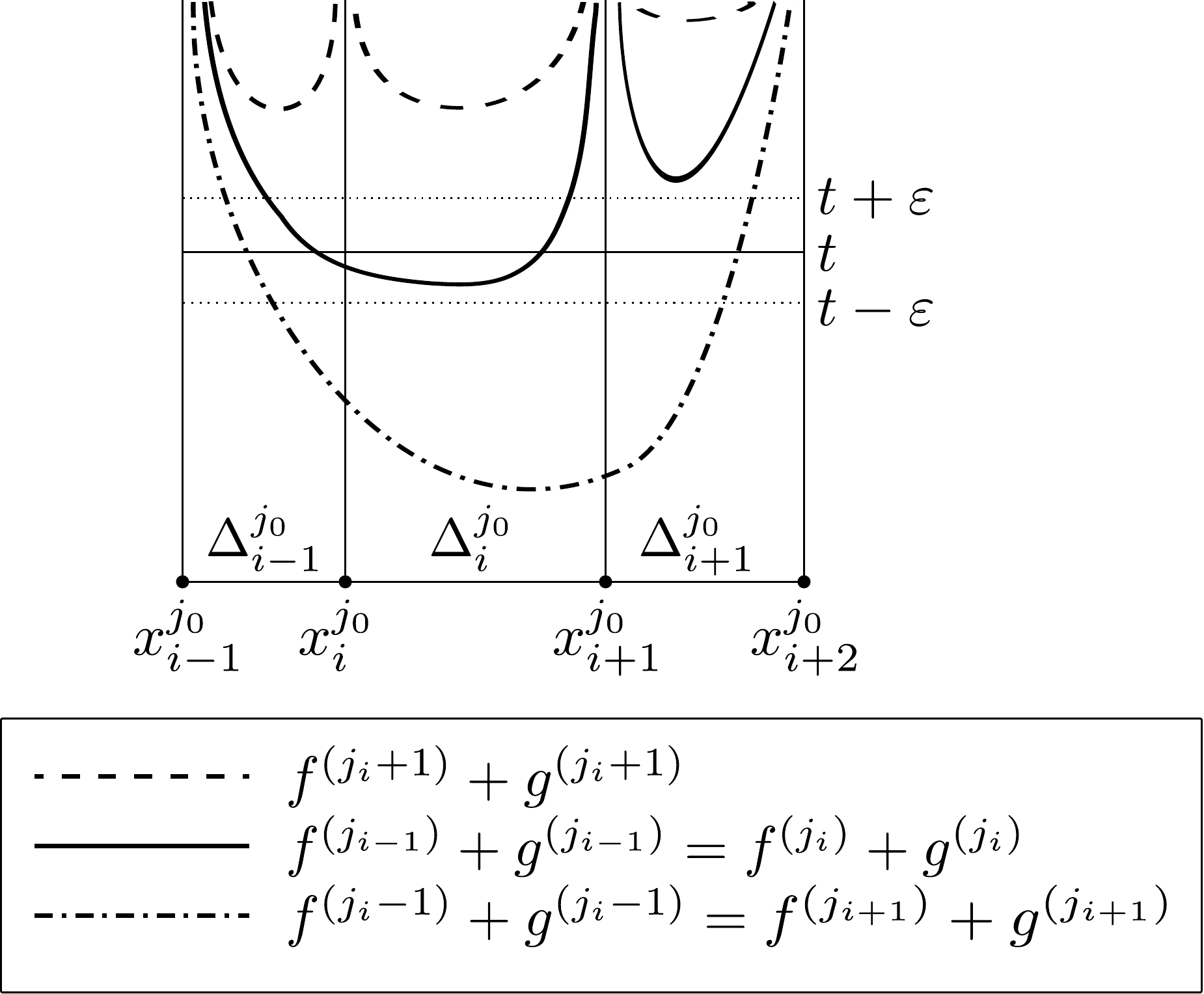}
\caption{Assume $j_{i-1}=j_i\neq j_{i+1}$. \emph{Step 1.} Remove $x_i^{j_0}$. \emph{Step 2.} Remove $x_{i+1}^{j_0}$}
\label{step1-2}
\end{figure}

\begin{figure}[ht]
\centering
\includegraphics[width=200pt]{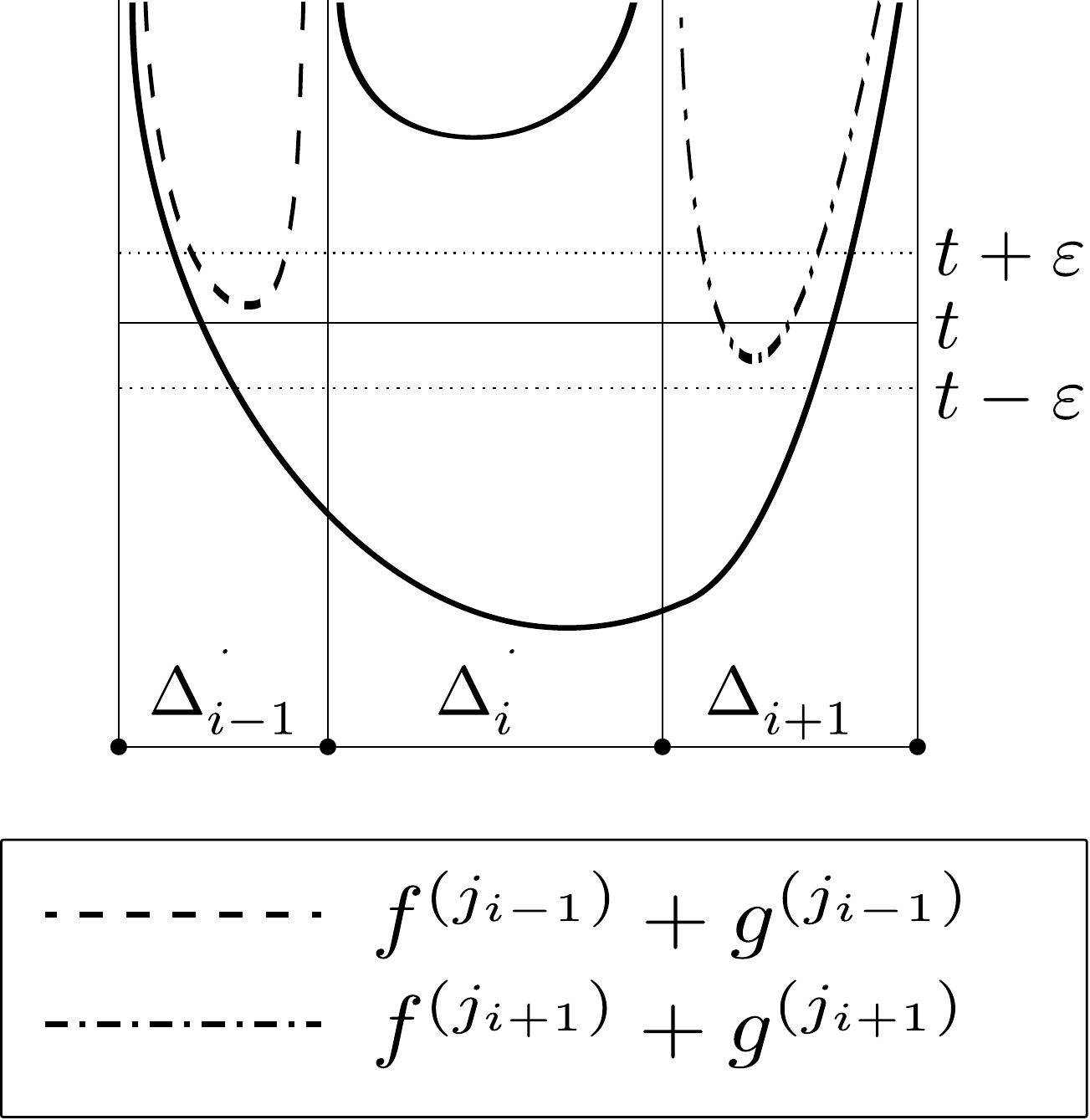}
\caption{\emph{Step 3.} Don't remove any of the points $\inf\Delta_i$, $\sup\Delta_i$}
\label{step3-1}
\end{figure}
\begin{figure}[ht]
\centering
\includegraphics[width=200pt]{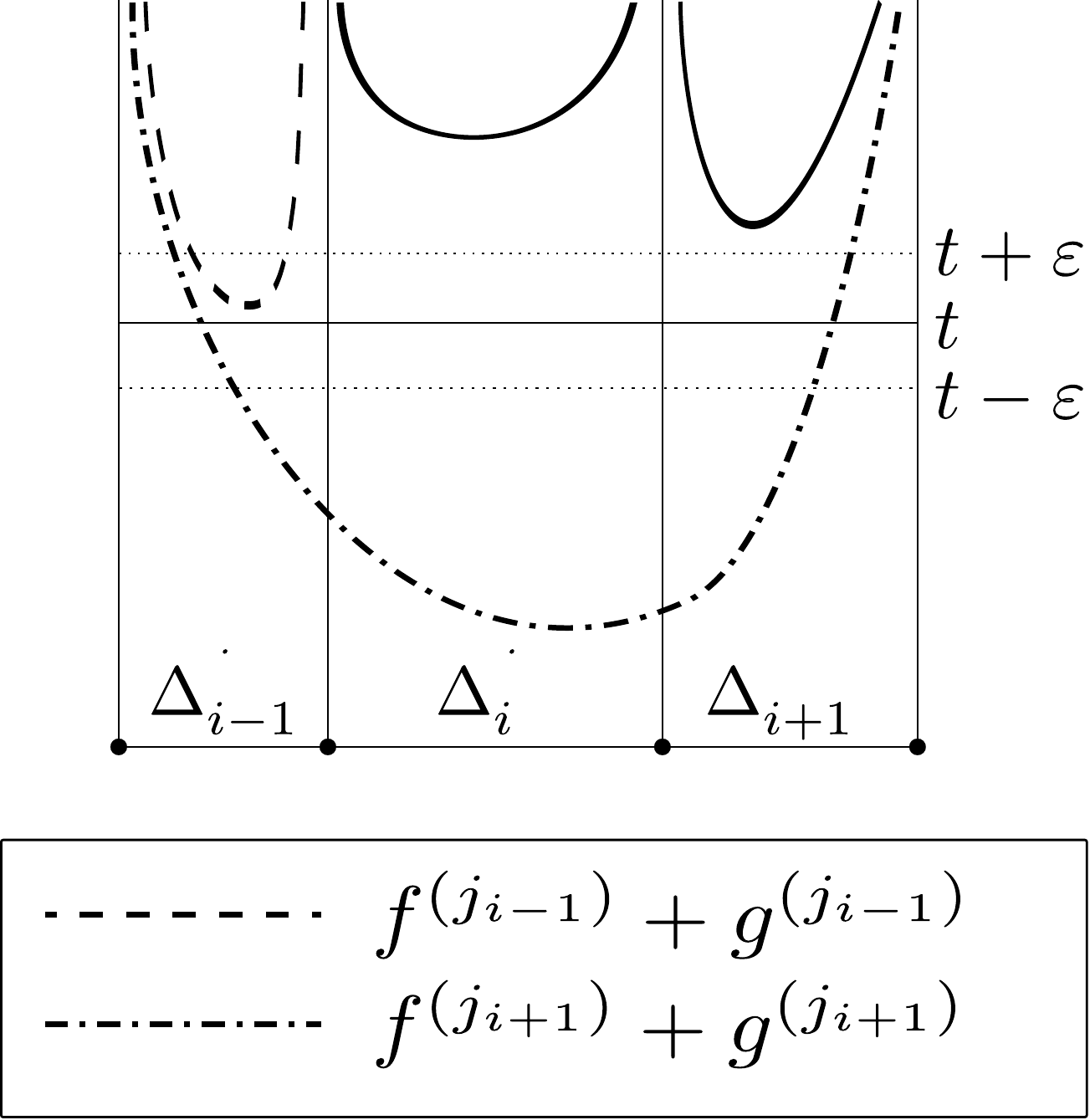}
\caption{\emph{Step 3.} Remove both $\inf\Delta_i$ and $\sup\Delta_i$}
\label{step3-2}
\end{figure}


\emph{Step 1} (see Fig.~\ref{step1-1} and~\ref{step1-2}).
Find all $0\leq i< (r-1)j_0$ such that $j_i^{j_0}=j_{i+1}^{j_0}=-\infty$ or $j_i^{j_0}=j_{i+1}^{j_0}>-\infty$ and function $f^{(j_i^{j_0})}+g^{(j_i^{j_0})}$ is continuous at $x_i^{j_0}$. Remove points $x_i^{j_0}$ from $P$ for all such $i$'s.

\emph{Step 2} (see Fig.~\ref{step1-2}).
Find all $0\leq i< (r-1)j_0$ such that $j_i^{j_0}\neq j_{i+1}^{j_0}$, $j_i^{j_0}, j_{i+1}^{j_0}>-\infty$ and at least one of the functions $f^{(j_i^{j_0})}+g^{(j_i^{j_0})},\ f^{(j_{i+1}^{j_0})}+g^{(j_{i+1}^{j_0})}$ is continuous at $x_i^{j_0}$. Remove points $x_i^{j_0}$ from $P$ for all such $i$'s.

\emph{Step 3} (see Fig.~\ref{step3-1} and~\ref{step3-2}).
To describe what to do in the last step of the construction, denote first the intervals of the partition determined by $P$ from the left to the right by $\Delta_1, \dots,\Delta_n$, where $n$ is the number of elements of $P$. For $1\leq i\leq n$ denote by $d_i$ the length of the interval $\Delta_i$ and put
\begin{equation}\label{joti}
j_i=\max \left\{j \in\mathbb{N}\colon \left(\exists\ {x \in \Delta_i}\right) \left|f^{(j)}(x)+g^{(j)}(x)-t\right|<\varepsilon\right\} \in\mathbb{N}\cup\{-\infty\}.
\end{equation}

For $1\leq i\leq n-1$ we claim that either $j_i,j_{i+1}\in\mathbb{N}$ and by \emph{Step 1} and \emph{Step 2} the point $\sup \Delta_i=\inf \Delta_{i+1}$ is a discontinuity for both $f^{(j_i)}+g^{(j_i)}$ and $f^{(j_{i+1})}+g^{(j_{i+1})}$, or exactly one of the numbers $j_i,j_{i+1}$ is equal to $-\infty$. Indeed, suppose that one of the functions $f^{(j_i)}+g^{(j_i)}$ or $f^{(j_{i+1})}+g^{(j_{i+1})}$ is continuous at $\sup \Delta_i=\inf \Delta_{i+1}$ and $j_i,j_{i+1}\in\mathbb{N}$. Notice that each interval $\Delta_i$ is a union of subintervals of the form $\Delta_s^{j_0}$ and $j_i \geq j_s^{j_0}$ for any number $s$ such that $\Delta_s^{j_0}\subset \Delta_i$. Therefore by \emph{Step 1} or \emph{Step 2} of the construction we would have removed point $\sup \Delta_i=\inf \Delta_{i+1}$ from set $P$. Hence whenever one of the functions $f^{(j_i)}+g^{(j_i)}$ or $f^{(j_{i+1})}+g^{(j_{i+1})}$ is continuous at $\sup \Delta_i=\inf \Delta_{i+1}$ then at least one of the numbers $j_i, j_{i+1}$ is equal to $-\infty$. If $j_i=j_{i+1}=-\infty$, we would have removed point $\sup \Delta_i=\inf \Delta_{i+1}$ from set $P$ by \emph{Step 3} of the construction, whence exactly one of the numbers $j_i,j_{i+1}$ is equal to $-\infty$. Now we concentrate our attention on $1\leq i\leq n$ such that $j_i=-\infty$. For simplicity of notation put $j_0=j_{n+1}=0$. By \emph{Step 3}, we have $j_{i-1}, j_{i+1}\in\mathbb{N}$ whenever $j_i=-\infty$. If $f^{(j_{i-1})}+g^{(j_{i-1})}$ is continuous at $\inf \Delta_i$ or $f^{(j_{i+1})}+g^{(j_{i+1})}$ is continuous at $\sup\Delta_i$, we remove both $\inf\Delta_i$ and $\sup\Delta_i$ from $P$. The construction is complete and again by abuse of notation, we continue to denote the intervals of the partition determined by $P$ by $\Delta_1,\dots,\Delta_n$ and their lengths by $d_1,\dots,d_n$. The numbers $j_i$ are still defined by formula~\eqref{joti} for the \emph{new} intervals $\Delta_i$. 

For $j\in\mathbb{N}$ and $1\leq i\leq n$ set
\begin{align*}
\Delta_{i,j}(f,g,t,\varepsilon)&=\left\{x\in\Delta_i\colon \left|f^{(j)}(x)+g^{(j)}(x)-t\right|<\varepsilon\right\},
\\
\Delta_{i,j}^{-}(f,g,t,\varepsilon)&=\Delta_{i,j}(f,g,t,\varepsilon)\cap \left\{x\in[0,1)\colon f'^{(j)}(x)<0\right\},
\\
\Delta_{i,j}^{+}(f,g,t,\varepsilon)&=\Delta_{i,j}(f,g,t,\varepsilon)\cap \left\{x\in[0,1)\colon f'^{(j)}(x)>0\right\}.
\end{align*}
If there is no ambiguity, we will write briefly $\Delta_{i,j}$, $\Delta_{i,j}^{-}$ and $\Delta_{i,j}^{+}$. 
\begin{lm}\label{partycja}
The partition of $[0,1)$ into $\Delta_1,\dots,\Delta_n$ described above satisfies the following properties:
\begin{enumerate}\label{wlasnosci}
\item
each interval $\Delta_i$ of the partition is a finite union of maximal intervals on which $f^{(j_i)}+g^{(j_i)}$ is continuous;
\item
for each interval $\Delta_i$ of the partition with $j_i>-\infty$ and for every $N_0< j\leq j_i$ there exists a unique number $0\leq q\leq (r-1)j$ ($q=q(i,j)$) such that $\Delta_q^j\cap\Delta_{i,j}\neq\emptyset$.
\end{enumerate}
\end{lm}
\begin{proof}

\emph{Property 1.}
Notice that by construction the endpoints of $\Delta_i$ are discontinuity points for $f^{(j_i)}+g^{(j_i)}$ (otherwise we would have removed them from $P$ - see \emph{Step 1} or \emph{Step 2} of the construction of the partition). Therefore the partition has required Property 1.

\emph{Property 2.}
Suppose, for contradiction, that there exist $y,y'\in\Delta_i$ ($y\leq y'$), $j,j'\in\mathbb{N}$ ($N_0\leq j,j'\leq j_i$) such that
\begin{itemize}
\item
\begin{equation*}
\lim_{x\to y^{-}}f^{(j)}(x)=\lim_{x\to y'^{+}}f^{(j')}(x)=+\infty,
\end{equation*}
\item
\begin{equation*}
\Delta_{i,j}\cap (0,y)\neq \emptyset,\ \Delta_{i,j'}\cap (y',1)\neq \emptyset,
\end{equation*}
\item
there is no point $z\in (0,y)$ such that
\begin{equation*}
\lim_{x\to z^{-}}f^{(j)}(x)=+\infty \text{ and }\Delta_{i,j}\cap (0,y)=\Delta_{i,j}\cap (0,z)
\end{equation*}
\item
and no point $z'\in(y',1)$ such that
\begin{equation*}
\lim_{x\to z'^{+}}f^{(j')}(x)=+\infty \text{ and }\Delta_{i,j'}\cap (y',1)=\Delta_{i,j'}\cap (z',1).
\end{equation*}
\end{itemize}
Assume also that $y$ and $y'$ are the closest such points, meaning that the condition $\Delta_{i,k}\cap (y,y')\neq\emptyset$ for some $k\geq N_0$ implies that the function $f^{(k)}+g^{(k)}$ is continuous on $(y,y')$ (see Fig.~\ref{bla}).

\begin{figure}[ht]
\centering
\includegraphics[height=150pt]{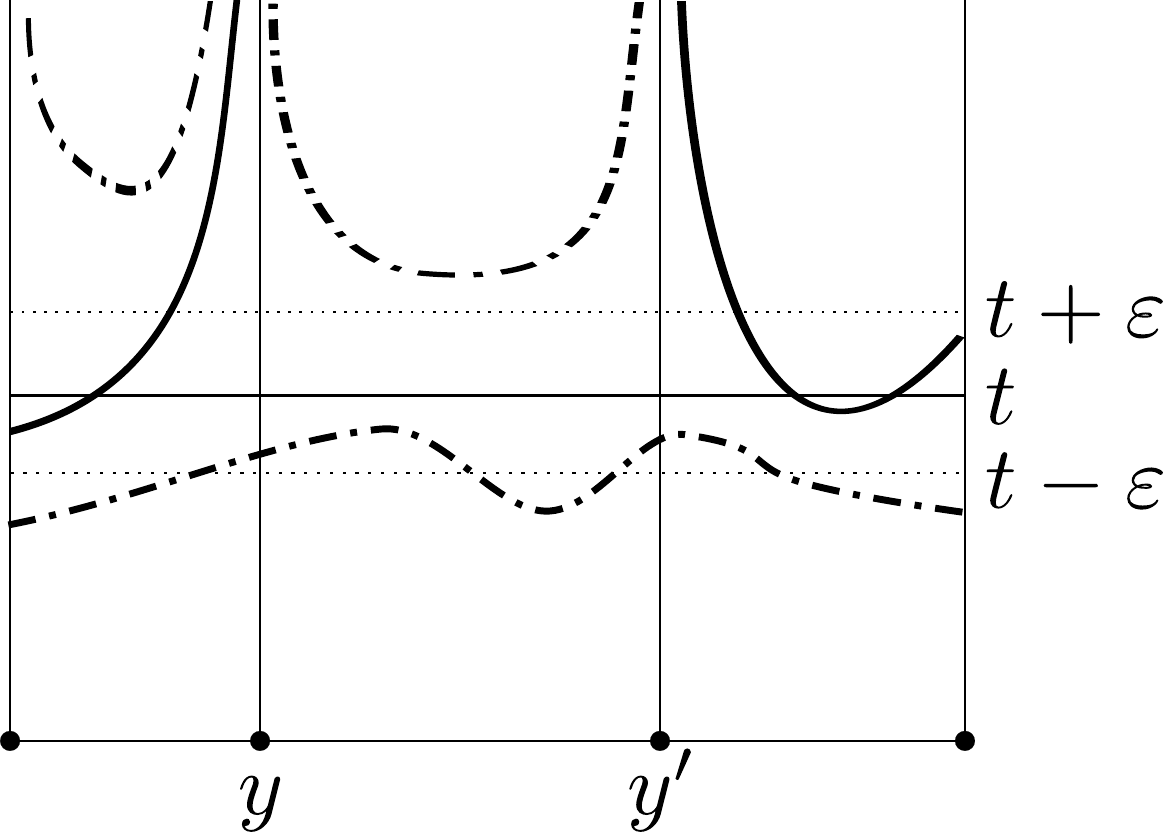}
\caption{A situation in the proof of \emph{Property 2}}
\label{bla}
\end{figure}

We claim that
\begin{equation}\label{pustawy}
\Delta_{i,k}\cap (y,y')=\emptyset\text{ for all }k\geq N.
\end{equation}
Suppose, to derive a contradiction, that this is not true and for some $k\geq N_0$ there exists $z\in (y,y')$ such that $f^{(k)}(z)+g^{(k)}(z)\in(t-\varepsilon,t+\varepsilon)$. Without loss of generality we may assume that $f^{(k)}$ is increasing on $(z,y')$ (if this is not the case, then it is decreasing on $(y,z)$). Let $\overline{y}=\inf\{x\geq y'\colon f^{(j')}(x)+g^{(j')}(x)\in(t-\varepsilon,t+\varepsilon)\}$. Notice that $k<j,j'$ (otherwise $f^{(k)}$ would be discontinuous at $y$ or $y'$ and this would contradict our choice of $y$ and $y'$). We have
\begin{equation*}
f^{(k)}(\overline{y})+g^{(k)}(\overline{y})\leq f^{(j')}(\overline{y})+g^{(j')}(\overline{y})-\min(f+g)=t+\varepsilon-\min(f+g).
\end{equation*}
On the other hand
\begin{equation*}
f^{(k)}(\overline{y})+g^{(k)}(\overline{y})\geq f^{(k)}(z)+g^{(k)}(z)-\varepsilon>t-2\varepsilon,
\end{equation*}
where the first inequality follows from Lemma~\ref{wn8} and from the fact that $f^{(k)}$ is increasing on $(z,\overline{y})$. Hence
\begin{equation*}
t-2\varepsilon<t+\varepsilon-\min(f+g)
\end{equation*}
and this is impossible since $3\varepsilon<\min(f+g)$ (see~\eqref{hutu}, page~\pageref{hutu}). Therefore~\eqref{pustawy} holds. In view of the construction of $P$ (\emph{Step 3}) this is however impossible and the proof is complete.
\end{proof}

For $1\leq i\leq n$ and $N_0<j\leq j_i$ such that $\Delta_{i,j}\neq\emptyset$ pick $x_{i,j}\in\Delta_{i,j}$ and let $g_{i,j}=g^{(j)}(x_{i,j})$ for some $x\in\Delta_{i,j}$. For $N<j\leq j_i$ put
\begin{equation*}
\widetilde{\Delta}_{i,j}^{-}=\left\{ x\in\Delta_i\colon \left|f^{(j)}(x)+g_{i,j}-t\right|<2\varepsilon \text{ and } f'^{(j)}(x)<0 \right\}\cap \Delta_{q(i,j)}^j
\end{equation*}
and
\begin{equation*}
\widetilde{\Delta}_{i,j}^{+}=\left\{ x\in\Delta_i\colon \left|f^{(j)}(x)+g_{i,j}-t\right|<2\varepsilon \text{ and } f'^{(j)}(x)>0 \right\}\cap \Delta_{q(i,j)}^j,
\end{equation*}
where $q(i,j)$ is the unique number $q$ such that $\Delta_q^j\cap\Delta_{i,j}\neq\emptyset$ (such a number exists by Property 2. from Lemma~\ref{partycja}). Let
\begin{equation*}
\widetilde{\Delta}_{i,j}=\widetilde{\Delta}_{i,j}^+ \cup\widetilde{\Delta}_{i,j}^-.
\end{equation*}
\begin{uw}\label{uwaganiewiemczypotrzebna}
Let $j \geq N_0$. Notice that $\Delta_{i,j}\subset \widetilde{\Delta}_{i,j}$. Indeed, let $x \in \Delta_{i,j}$. By Property 2 in Lemma~\ref{partycja} it follows that $x\in \Delta_{q(i,j)}^j$. Therefore and by Lemma~\ref{wn8} we have
\begin{equation*}
\left|g^{(j)}(x)-g_{i,j} \right|<\varepsilon.
\end{equation*}
Moreover, by the assumption that $x\in \Delta_{i,j}$,
\begin{equation*}
\left| f^{(j)}(x)+g^{(j)}(x)-t\right|<\varepsilon.
\end{equation*}
Hence
\begin{equation*}
\left|f^{(j)}(x)+g_{i,j}-t \right|\leq \left| f^{(j)}(x)+g^{(j)}(x)-t\right|+\left|g^{(j)}(x)-g_{i,j} \right|<2\varepsilon
\end{equation*}
and our claim follows.
\end{uw}
Let $A, B\subset [0,1)$. We write $A \leq B$ if for every $a\in A$ and every $b\in B$ we have $a\leq b$. In particular, $A\leq B$ if $A=\emptyset$ or $B= \emptyset$.
\begin{lm}\label{lm4.7}
If $\varepsilon<\frac{1}{6}\min (f+g)$ then the condition $N_0< j<j'\leq j_i$ implies 
\begin{equation}\label{ukl}
\widetilde{\Delta}_{i,j}^{-}\leq \widetilde{\Delta}_{i,j'}^{-}\leq\widetilde{\Delta}_{i,j'}^{+}\leq \widetilde{\Delta}_{i,j}^{+}
\end{equation}
for all $1\leq i\leq n$.
\end{lm}
\begin{proof}
We will show that $\widetilde{\Delta}_{i,j}^{-}\leq \widetilde{\Delta}_{i,j'}^{-}$ for $N_0<j<j'\leq j_i$. The proof of the remaining part of the statement is analogous.
Suppose for contradiction that there exist $x\in\widetilde{\Delta}_{i,j}^{-}$, $x'\in\widetilde{\Delta}_{i,j'}^{-}$ such that $x'<x$. Let $y\in{\Delta}_{i,j}\subset\widetilde{\Delta}_{i,j}$, $y'\in{\Delta}_{i,j'}\subset\widetilde{\Delta}_{i,j'}$ be such that $g^{(j)}(y)=g_{i,j}$, $g^{(j')}(y')=g_{i,j'}$ (see Remark~\ref{uwaganiewiemczypotrzebna}). By Lemma~\ref{wn8} and since $j<j'$, we have
\begin{multline*}
f^{(j')}(x')+g^{(j')}(y')
\\
\geq f^{(j')}(x')+g^{(j')}(x')-\varepsilon \geq f^{(j)}(x')+g^{(j)}(x')+\min (f+g)-\varepsilon.
\end{multline*}
Therefore
\begin{multline}\label{p31}
f^{(j)}(x')+g^{(j)}(x')\leq f^{(j')}(x')+g^{(j')}(y')-\min(f+g)+\varepsilon
\\
=\left(f^{(j')}(x')+g_{i,j'}-t \right)+\left(t+\varepsilon-\min(f+g) \right)
\\
<2\varepsilon+ \left(t+\varepsilon-\min(f+g) \right)=t+3\varepsilon-\min(f+g),
\end{multline}
where the right inequality follows from $x' \in \widetilde{\Delta}^{-}_{i,j'}$.
There are two cases: either $f^{(j)}+g^{(j)}$ is continuous on $[x',x]$ or it is not. In the first case we have
\begin{equation}\label{p32}
f^{(j)}(x')+g^{(j)}(x')>f^{(j)}(x)+g^{(j)}(x)-\varepsilon\geq t-3\varepsilon
\end{equation}
(the inequalities follow from the fact that $f^{(j)}$ is decreasing at $x$, so from Lemma~\ref{pochodna} and from Lemma~\ref{wn8} it is decreasing also on $[x',x]$). In the second case by Property 2. in Lemma~\ref{partycja},
\begin{equation}\label{p33}
f^{(j)}(x')+g^{(j)}(x')=f^{(j)}(x')+g_{i,j}+g^{(j)}(x')-g_{i,j} >t-3\varepsilon.
\end{equation}
Hence from~\eqref{p31},~\eqref{p32} and~\eqref{p33} we obtain
\begin{equation*}
t-3\varepsilon<t+3\varepsilon-\min(f+g),
\end{equation*}
which is a clear contradiction with the choice of~$\varepsilon$ (see~\eqref{hutu}, page~\pageref{hutu}).
\end{proof}

\begin{lm}\label{lm9}
For each $\widehat{\eta} >0$, $\varepsilon < \min(\min(f+g),\frac{1}{4}\min f)$ and each $N \in \mathbb{N}$ there exists $T_0>0$ such that for all $t>T_0$ the following inequality holds:
\begin{equation*}
m \left\{x\in[0,1) \colon \left( \exists\ {j\leq N}\right) |f^{(j)}(x)+g^{(j)}(x)-t|<\varepsilon\right\}<\widehat{\eta}.\footnote{Recall that $m$ stands for the Lebesgue measure.}
\end{equation*}
\end{lm}
\begin{proof}
Fix $\widehat{\eta}>0$, $\varepsilon < \min(\min(f+g),\frac{1}{4}\min f)$ and $N\in\mathbb{N}$. Let
\begin{equation*}
c_j=\max_{0\leq i\leq(r-1)j}\min_{\Delta_{i}^{j}} f^{(j)}-2\varepsilon.
\end{equation*}
Since $f^{(j+1)}\geq f^{(j)}+\min f$, we have $c_{j+1}\geq c_j+\min f$. Put
\begin{equation*}
j_0(t)=\max\{j\in\mathbb{N}\colon c_j<t\}.
\end{equation*}
Let $\delta>0$ be such that for $x,y\in \Delta_i^k$ the condition $|x-y|<\delta$ implies $|g^{(k)}(x)-g^{(k)}(y)|<\varepsilon$ for $k\leq N$ and all $0\leq i \leq(r-1)j$. Set $G=\max(\max g,0)$ and take $M>0$ such that the following holds:
\begin{align}
M&> \frac{2Nrc}{\widehat{\eta}},\label{nr1a}
\\
M&> cN,\label{nr1b}
\\
M&>\frac{c}{\delta},\label{nr1c}
\end{align}
where constant $c>1$ is the same as in the definition of balanced partition lengths. Let $T_0=c_{M+1}+NG$ and fix $t>T_0$. Letting $j_0:=j_0(t-G)$ by~\eqref{nr1a} we obtain
\begin{equation}\label{nr2a}
j_0\geq j_0 (c_{M+1})= M.
\end{equation}
Moreover, by~\eqref{nr1b} we have
\begin{equation}\label{nr2b}
j_0\geq M>cN\geq N.
\end{equation}
We claim that
\begin{equation*}
\left\{x\in\Delta_i^{j_0}\colon \left(\exists\ {j\leq N}\right)\left|f^{(j)}(x)+NG-t\right|<2\varepsilon\right\}=\emptyset
\end{equation*}
for $0\leq i\leq (r-1)j_0$ such that $f^{(j)}$ is continuous on the closure $cl\left(\Delta_i^{j_0}\right)$ of $\Delta_i^{j_0}$. Indeed, suppose that this is not the case and take $\overline{x}\in cl\left(\Delta_i^{j_0}\right)$ such that $\left|f^{(j)}(\overline{x})+NG-t\right|<2\varepsilon$. Without loss of generality, we may assume that $f'^{(j)}(\overline{x})>0$ (otherwise we have $f'^{(j)}(\overline{x})<0$ and instead of looking at $\Delta_{i+1}^{j_0}$ in what follows, we look at $\Delta_{i-1}^{j_0}$). 
\begin{figure}[ht]
\centering
\includegraphics[height=220pt]{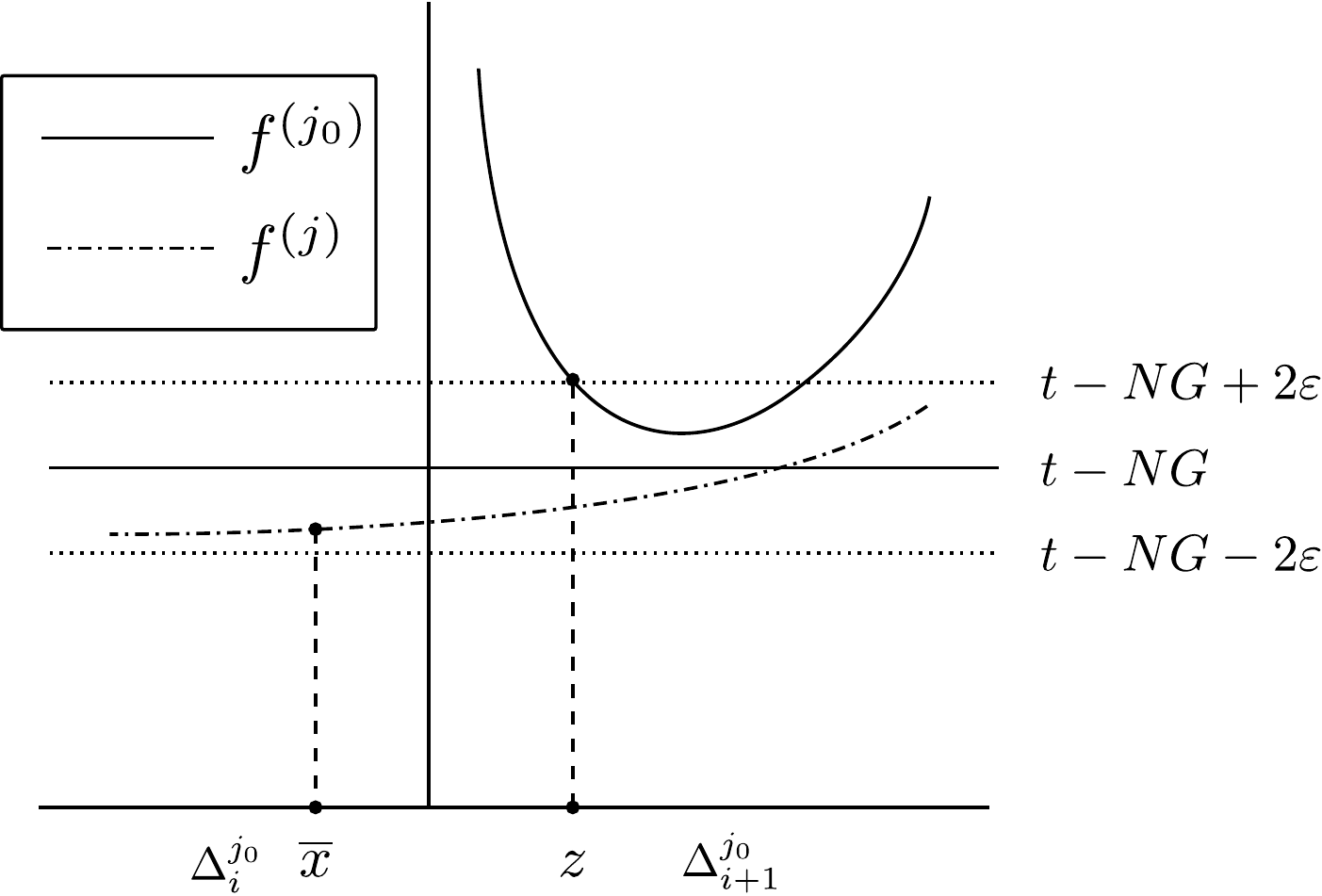}
\caption[Ilustracja dowodu...]{The situation in the intervals $\Delta_i^{j_0}$ i $\Delta_{i+1}^{j_0}$.}
\label{2obok}
\end{figure}
There exists $x\in\Delta_{i+1}^{j_0}$ such that $f'^{(j_0)}(x)<0$ and $\left|f^{(j_0)}(x)+NG-t\right|<2\varepsilon$. Let
\begin{equation*}
z=\inf\left\{x\in\Delta_{i+1}^{j_0}\colon \left|f^{(j_0)}(x)+NG-t\right|<2\varepsilon\right\}.
\end{equation*}
Since $j_0\geq N>j$, we have (see Figure~\ref{2obok})
\begin{equation*}
t-2\varepsilon\leq f^{(j)}(z)+NG\leq f^{(j_0)}(z)+NG-\min f=t+2\varepsilon-\min f,
\end{equation*}
whence $\min f<4\varepsilon$, which is impossible by choice of $\varepsilon$. Since $f^{(N)}$ has $N(r-1)+1\leq Nr$ discontinuities and $f^{(j)}$ for $j<N$ is continuous whenever $f^{(N)}$ is continuous, at most $2Nr$ of the intervals $cl\left(\Delta_i^{j_0}\right)$ have a nonempty intersection with the set 
$$\left\{x\in[0,1)\colon \left(\exists\  {j\leq N}\right)\left|f^{(j)}(x)+NG-t\right|<2\varepsilon\right\}.$$
Hence, by~\eqref{nr2a} and~\eqref{nr1a} and using the assumption that the considered IET has balanced partition lengths, we obtain
\begin{multline}\label{krka}
m\left\{x\in[0,1)\colon \left(\exists\ {j\leq N}\right) \left|f^{(j)}(x)+NG-t\right|<2\varepsilon\right\}\leq 2Nr \max_{0\leq i\leq (r-1)j_0}m(\Delta_i^{j_0})
\\
\leq 2Nr\frac{c}{j_0}\leq 2Nr\frac{c}{M}<\widehat{\eta}.
\end{multline}
\begin{figure}[ht]
\centering
\includegraphics[height=220pt]{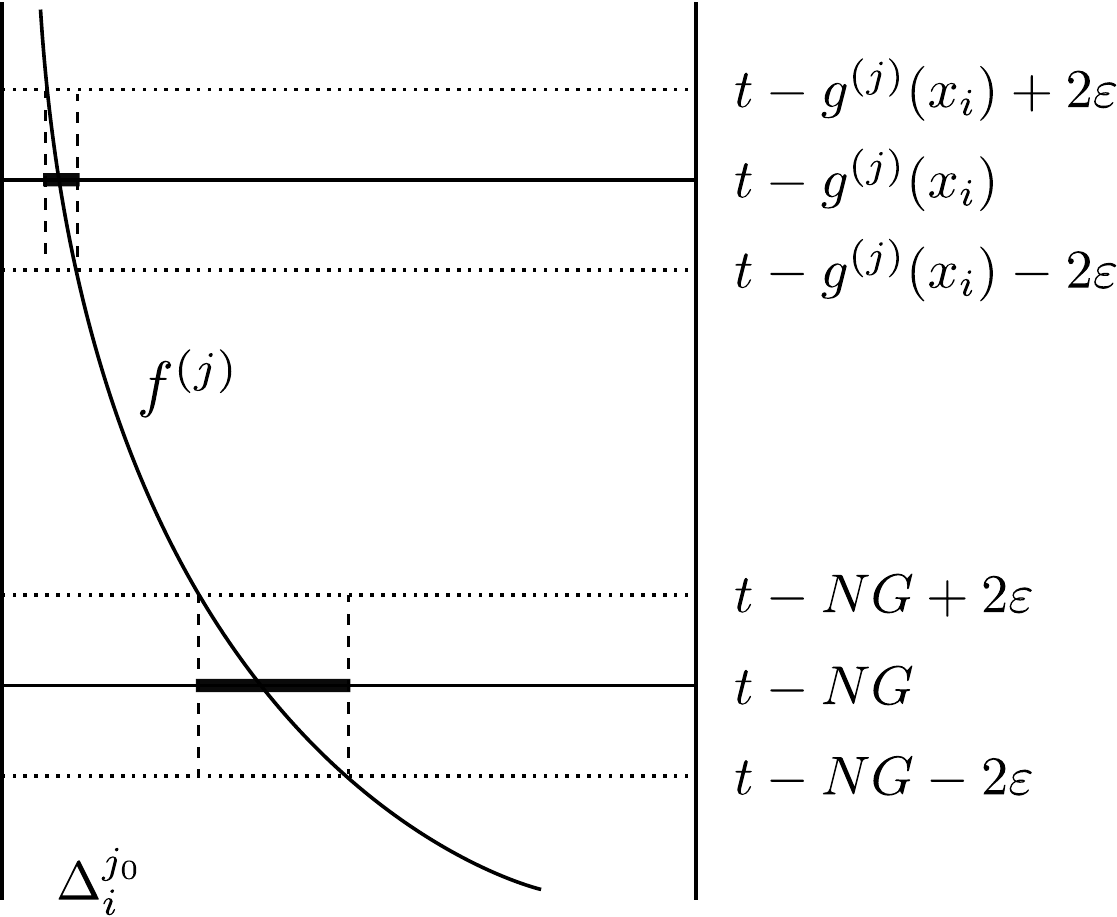}
\caption{The situation in $\Delta_i^{j_0}$}
\label{wypuklosc}
\end{figure}
For $1\leq i\leq (r-1)j_0$ pick $x_i\in \Delta_i^{j_0}$. Since $g^{(j)}(x_i)\leq NG$ for $j\leq N$, by convexity of $f^{(j)}$ we have
\begin{multline}\label{lm91g}
m\left\{x\in\Delta_i^{j_0}\colon \left|f^{(j)}(x)+NG-t\right|<2\varepsilon\right\}
\\
\geq m \left\{x\in\Delta_i^{j_0}\colon \left|f^{(j)}(x)+g^{(j)}(x_i)-t\right|<2\varepsilon\right\},
\end{multline}
where $x_i\in\Delta_i^{j_0}$ (see Figure~\ref{wypuklosc}) and
\begin{multline}\label{lm92g}
\left\{x\in\Delta_i^{j_0}\colon \left|f^{(j)}(x)+g^{(j)}(x)-t\right|<\varepsilon\right\}
\\
\subset \left\{x\in\Delta_i^{j_0}\colon \left|f^{(j)}(x)+g^{(j)}(x_i)-t\right|<2\varepsilon\right\}.
\end{multline}
Indeed, to justify~\eqref{lm92g} notice that by~\eqref{nr2a} and~\eqref{nr1c} we have
\begin{equation*}
m(\Delta_i^{j_0})\leq \frac{c}{j_0}\leq \frac{c}{M}<\delta,
\end{equation*}
so for $x\in\Delta_i^{j_0}$ it holds $\left|g^{(j)}(x)-g^{(j)}(x_i)\right|<\varepsilon$ and
\begin{equation*}
f^{(j)}(x)+g^{(j)}(x_i)=f^{(j)}(x)+g^{(j)}(x)-g^{(j)}(x)+g^{(j)}(x_i)\in(t-2\varepsilon,t+2\varepsilon),
\end{equation*}
provided that $f^{(j)}(x)+g^{(j)}(x)\in(t-\varepsilon,t+\varepsilon)$. Therefore, by~\eqref{lm92g} and~\eqref{lm91g}
\begin{multline*}
m\left\{x\in\Delta_i^{j_0}\colon \left|f^{(j)}(x)+g^{(j)}(x)-t\right|<\varepsilon\right\}
\\
\leq m\left\{x\in\Delta_i^{j_0}\colon \left|f^{(j)}(x)+NG-t\right|<2\varepsilon\right\}.
\end{multline*}
Hence
\begin{multline}\label{odin}
m\left\{x\in[0,1)\colon \left(\exists\ {j\leq N}\right) \left|f^{(j)}(x)+g^{(j)}(x)-t \right|<\varepsilon \right\}
\\
\leq \sum_{i} \sum_{j\leq N} m \left\{ x\in \Delta_i^{j_0}\colon \left|f^{(j)}(x)+g^{(j)}(x)-t \right|<\varepsilon\right\}
\\
\leq \sum_{i} \sum_{j \leq N} m\left\{ x\in \Delta_i^{j_0}\colon \left| f^{(j)}(x)+NG-t \right|<2\varepsilon\right\}.
\end{multline}
Notice that from $\varepsilon<\frac{1}{4}\min f$ it follows that the sets $\left\{x\in[0,1)\colon \left|f^{(j)}(x)+NG-t\right|<2\varepsilon\right\}$ are pairwise disjoint for $j\in\mathbb{N}$ (so in particular for $j\leq N$), whence
\begin{multline}\label{dba}
\sum_{i} \sum_{j \leq N} m\left\{ x\in \Delta_i^{j_0}\colon \left| f^{(j)}(x)+NG-t \right|<2\varepsilon\right\}
\\
=m\left(\bigcup_{i} \bigcup_{j \leq N} \{x\in\Delta_i^{j_0} \colon \left| f^{(j)}(x)+NG-t\right|<2\varepsilon \}\right)
\\
=m \left\{x\in [0,1) \colon \left(\exists\ {j \leq N}\right) \left| f^{(j)}(x)+NG-t\right|<2\varepsilon\right\}.
\end{multline}
The assertion follows from~\eqref{odin},~\eqref{dba} and~\eqref{krka}.
\end{proof}


\begin{proof}[Proof of Theorem~\ref{tw12}]
We claim that for any $\overline{\eta}\in(0,1)$ there exist $\varepsilon>0$ and $t_0>0$ such that 
\begin{equation}\label{klejm}
m\left(\{x\in[0,1)\colon \left(\exists\ {j\in\mathbb{N}}\right)\left|f^{(j)}(x)+g^{(j)}(x)-t\right|<\varepsilon\right\}\leq \overline{\eta}
\end{equation}
for $t>t_0$. This ensures that for any sequence $t_n\to \infty$ there exists $n_0\in\mathbb{N}$ such that for $n>n_0$
\begin{equation*}
m\left\{x\in[0,1)\colon \left(\exists\ {j\in\mathbb{N}}\right)\left|f^{(j)}(x)+g^{(j)}(x)-t_n\right|<\varepsilon\right\}\leq \overline{\eta},
\end{equation*}
which implies 
\begin{equation*}
\liminf_{n\to\infty} m \left\{x\in[0,1)\colon \left(\exists\ {j\in\mathbb{N}}\right) \left|f^{(j)}(x)+g^{(j)}(x)-t_n\right|<\varepsilon \right\}\leq \overline{\eta}.
\end{equation*}
By Lemma~\ref{lm5}, this means that the special flow $T^{f+g}$ is not partially rigid along any sequence $\{t_n\}$. Therefore, we are left to prove the claim~\eqref{klejm}.

Fix $\overline{\eta}\in(0,1)$ and take $\overline{\eta}>\widetilde{\eta}>0$, $\widehat{\eta}$ and $K\in\mathbb{N}$ such that
\begin{equation}\label{ety}
\widetilde{\eta}+\widehat{\eta}+\frac{1}{K}<\overline{\eta}.
\end{equation} 
It follows from Lemma~\ref{lm3} and from the left inequality in~\eqref{bpl} that condition~\eqref{six} in Lemma~\ref{lm4} holds for $\mathcal{H}=\{f^{(j)}\colon \Delta_i^j\to\mathbb{R}\colon j\geq 6c^2,\ 0\leq i\leq (r-1)j\}$ and $\widetilde{\delta}=\frac{\delta}{c}$. Let $\varepsilon$ be as in the assertion of Lemma~\ref{lm4}, making it smaller if necessary, such that
\begin{equation}\label{eps}
\varepsilon<\frac{\min (f+g)}{2K}
\end{equation}
and
\begin{equation*}
\varepsilon<\frac{1}{6}\min(f+g).
\end{equation*}
Then, by Lemma~\ref{lm4}, for $j\geq 6c^2$ and $0\leq i\leq(r-1)j$ we have
\begin{equation}\label{smile}
m\left\{x\in\Delta_i^j\colon \left|f^{(j)}(x)-s\right|<2\varepsilon\right\}\leq \widetilde{\eta}d_i^j
\end{equation}
for all $s>0$. Let $N_0\in\mathbb{N}$ be as in the assertion of Lemma~\ref{wn8}. Put
\begin{equation}\label{kropa}
t_0=\max_{0\leq i\leq(r-1)\lceil 6c^2 \rceil} \{\min_{\Delta_{i}^{\lceil 6c^2 \rceil}} (f^{(\lceil 6c^2\rceil)}+g^{(\lceil 6c^2\rceil)}-\varepsilon)\},
\end{equation}
where $\lceil \cdot \rceil$ stands for the ceiling function and fix $t>t_0$. 
Consider the partition of $[0,1)$ into subintervals $\Delta_i=\Delta_i(f,g,t,\varepsilon)$, where $1\leq i\leq n$, described previously. 
By Lemma~\ref{lm4.7} we have
\begin{equation*}
\widetilde{\Delta}_{i,j}^{-}\leq \widetilde{\Delta}_{i,j'}^{-}\leq\widetilde{\Delta}_{i,j'}^{+}\leq \widetilde{\Delta}_{i,j}^{+}
\end{equation*}
\begin{figure}[ht]
\centering
\includegraphics[height=220pt]{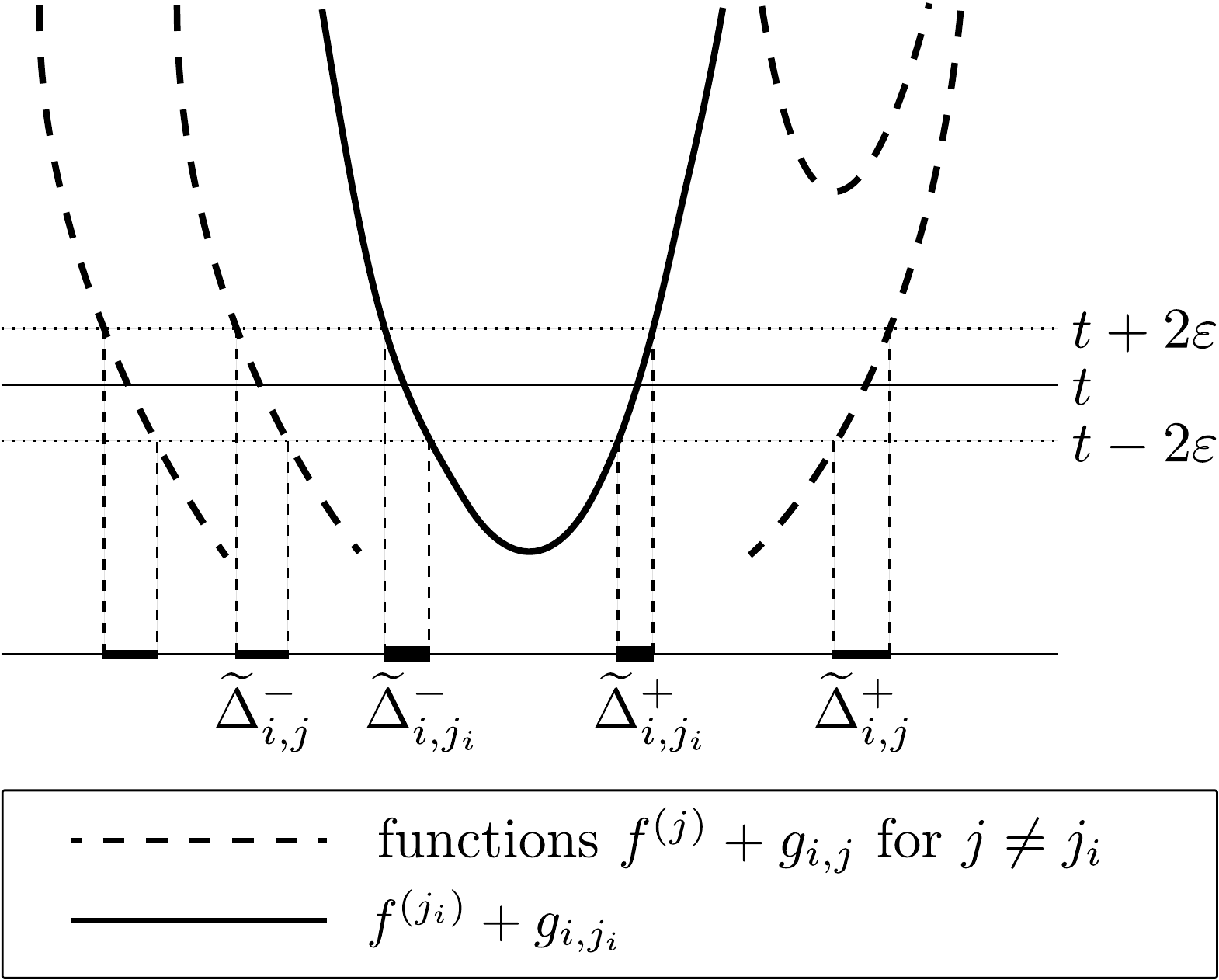}
\caption{Pairwise disjoint sets $\widetilde{\Delta}_{i,j}^{-}$, $\widetilde{\Delta}_{i,j}^{+}$}
\label{uklad}
\end{figure}
for all $N_0<j<j'\leq j_i$ and $1\leq i\leq n$ (see Fig.~\ref{uklad}). We will show that
\begin{equation}\label{12g1}
j_i\geq 6c^2\text{ for }1\leq i\leq n.
\end{equation}
Indeed, notice that from~\eqref{kropa} it follows that for all $0\leq i\leq (r-1)\lceil 6c^2 \rceil$ 
\begin{equation*}
t>t_0\geq \min_{\Delta_{i}^{\lceil 6c^2 \rceil}} (f^{(\lceil 6c^2\rceil)}+g^{(\lceil 6c^2\rceil)}-\varepsilon),
\end{equation*}
whence there exist $x_i\in\Delta_{i}^{\lceil 6c^2 \rceil}$ such that
\begin{equation*}
f^{(\lceil 6c^2\rceil)}(x_i)+g^{(\lceil 6c^2\rceil)}(x_i)<t+\varepsilon.
\end{equation*}
This, together with Lemma~\ref{pochodna}, implies that there exist $x'_i\in\Delta_i^{\lceil 6c^2 \rceil}$ satisfying
\begin{equation}\label{12g2}
\left|f^{(\lceil 6c^2 \rceil)}(x'_i)+g^{(\lceil 6c^2 \rceil)}(x'_i)-t\right|<\varepsilon.
\end{equation}
Hence, for $0\leq i\leq (r-1)\lceil 6c^2 \rceil$
\begin{equation}\label{12g22}
\max\{j\in\mathbb{N}\colon \left(\exists\ {x\in\Delta_i^{\lceil 6c^2\rceil}}\right) \left|f^{(j)}(x)+g^{(j)}(x)-t\right|<\varepsilon\}\geq \lceil 6c^2 \rceil.
\end{equation}
Suppose that for some $1\leq i_0\leq n$ we have $j_{i_0}<6c^2\leq \lceil 6c^2 \rceil$. Then each interval $\Delta_{q}^{j_{i_0}}$ ($0\leq q\leq (r-1)j_{i_0}$) consists of a finite number of intervals of the form $\Delta_i^{\lceil 6c^2 \rceil}$. Hence~\eqref{12g22} is contradictory to the definition of $j_{i_0}$ (recall that $\Delta_{i_0}$ is a union of intervals of the form $\Delta_q^{j_{i_0}}$) and~\eqref{12g1} has been shown.

Consider first the case where $j_i>N_0$ for all $1\leq i\leq n$. We claim that the following three inequalities hold:
\begin{align}
m(\Delta_{i,j_i})&\leq \widetilde{\eta}d_i,\label{12a}
\\
m\left(\bigcup_{N_0<j<j_i}\Delta_{i,j}\right)&\leq\frac{1}{K}d_i,\label{12b}
\\
m\left(\bigcup_{1\leq i\leq n}\bigcup_{j\leq N_0}\right)\Delta_{i,j}&\leq \widehat{\eta}.\label{12c}
\end{align}
For $N<j\leq j_i$ such that $\Delta_{i,j}\neq\emptyset$ let $g_{i,j}=g^{(j)}(x)$ for some $x\in\Delta_{i,j}$, as in the beginning of the proof of Lemma~\ref{lm4.7}. 
Using Lemma~\ref{partycja}, choose $\Delta=\Delta_q^{j_i}\cap\Delta_i$ so that $\Delta_{i,j_i}=\Delta_q^{j_i}\cap\Delta_{i,j_i}$. Notice that
\begin{equation*}
g_{i,j_i}+\min_{\Delta}f^{(j_i)}<t+\varepsilon.
\end{equation*}
Therefore
\begin{equation*}
t-g_{i,j_i}\geq \min_{\Delta}f^{(j_i)}-\varepsilon>0.
\end{equation*}
Similarly, $t-g_{i,j}>0$ for $N_0<j<j_i$. Hence~\eqref{12g1} implies that
\begin{equation}\label{12g3}
m\left(\widetilde{\Delta}_{i,j_i}\right)\leq \widetilde{\eta}d_i.
\end{equation}
By Lemma~\ref{wn8}, for $x,y\in\Delta$ it holds that
\begin{equation*}
\left|g^{(j_i)}(x)-g^{(j_i)}(y)\right|<\varepsilon.
\end{equation*}
Therefore $\Delta_{i,j_i}\subset \widetilde{\Delta}_{i,j_i}$. Indeed, if $x\in\Delta_{i,j_i}$, then 
\begin{equation*}
f^{(j_i)}(x)+g_{i,j_i}=f^{(j_i)}(x)+g^{(j_i)}(x)+g_{i,j_i}-g^{(j_i)}(x)\in (t-2\varepsilon,t+2\varepsilon).
\end{equation*}
Hence and by~\eqref{12g3} we have shown that~\eqref{12a} is true.

Now we will prove that~\eqref{12b} also holds. As before, for $N_0<j<j_i$ we have $\Delta_{i,j}\subset \widetilde{\Delta}_{i,j}$. Therefore it suffices to prove that
\begin{equation}\label{12b'}
m(\cup_{N_0<j<j_i}\widetilde{\Delta}_{i,j})\leq\frac{1}{K}d_i.
\end{equation}
We will use Lemma~\ref{lm4.7}. Notice that by Lemma~\ref{pochodna}, $f^{(j)}$ is convex on each interval where it is continuous. Therefore for $N_0<j<j_i$ such that $\widetilde{\Delta}^{-}_{i,j}\neq\emptyset$ by mean value theorem we have
\begin{equation*}
\frac{m(\widetilde{\Delta}_{i,j}^{-})}{x_{j+1}-x_j}\leq\frac{2\varepsilon}{f^{(j+1)}(x_{j+1})+g_{i,j+1}-f^{(j)}(x_{j+1})-g_{i,j}},
\end{equation*}
where $x_j=\inf\widetilde{\Delta}_{i,j}^{-}$ for $N_0<j\leq j_i$ (see Fig.~\ref{ka}). 
\begin{figure}[ht]
\centering
\includegraphics[height=200pt]{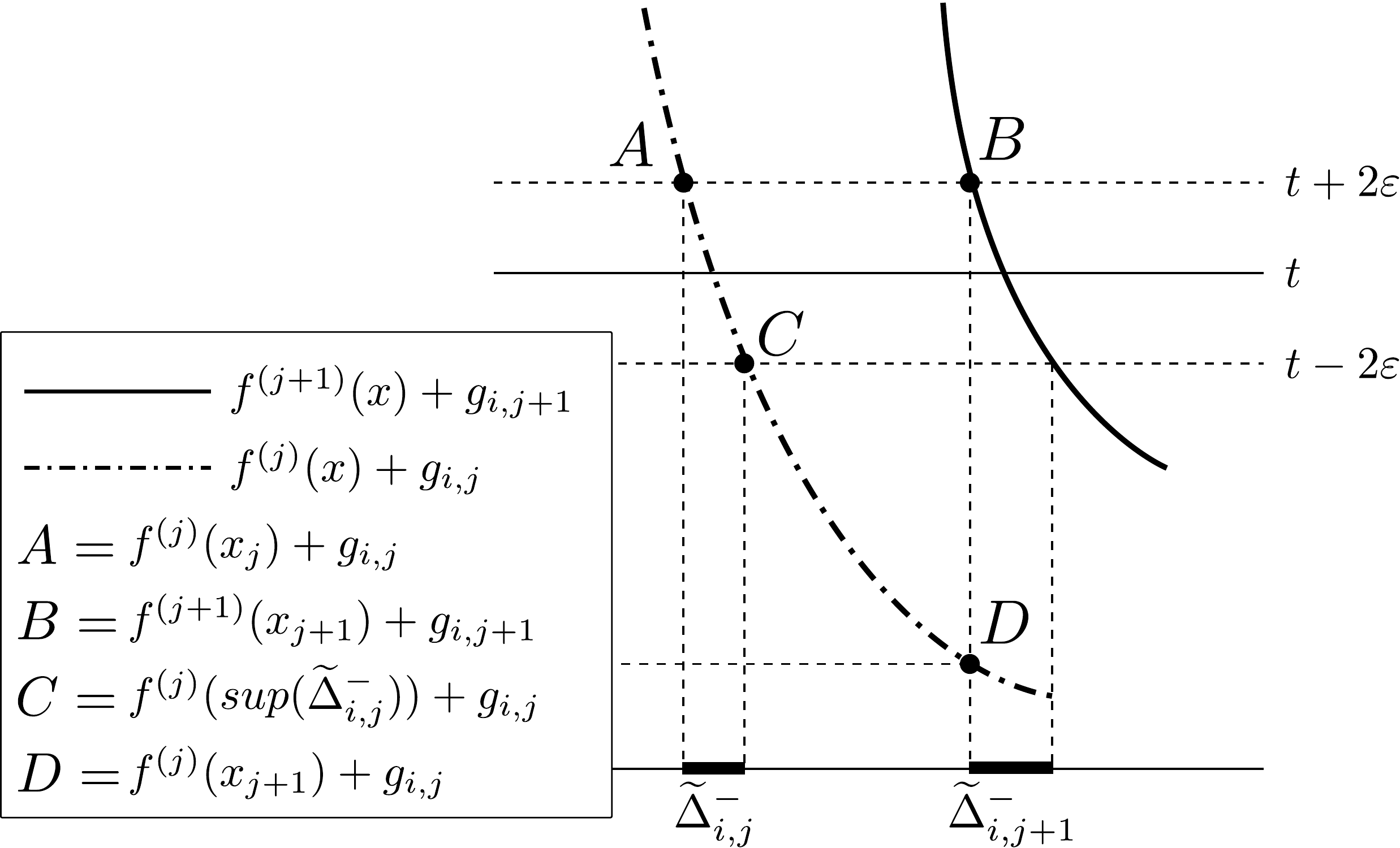}
\caption{A situation in the proof of Theorem~\ref{tw12}}
\label{ka}
\end{figure}
Now we estimate the denominator from below. It follows from Lemma~\ref{wn8} that
\begin{multline*}
f^{(j+1)}(x_{j+1})+g_{i,j+1}-f^{(j)}(x_{j+1})-g_{i,j}
\\
=f^{(j+1)}(x_{j+1})+g^{(j+1)}(x_{j+1})-(f^{(j)}(x_{j+1})+g^{(j)}(x_{j+1}))
\\
+g_{i,j+1}-g^{(j+1)}(x_{j+1})+g^{(j)}(x_{j+1})-g_{i,j}\geq \min(f+g)-2\varepsilon.
\end{multline*}
Therefore by~\eqref{eps} for $N_0<j<j_i$ we have 
\begin{equation*}
m\left(\widetilde{\Delta}_{i,j}^{-}\right)\leq\frac{2\varepsilon}{\min (f+g)}(x_{j+1}-x_j)<\frac{1}{K}\left(x_{j+1}-x_j\right),
\end{equation*}
whence (by adding the inequalities for $N_0<j<j_i$ and by $\inf\Delta_i<x_j$  for $N<j\leq j_i$ such that $\widetilde{\Delta}_{i,j}\neq\emptyset$)
\begin{equation*}
m(\cup_{N_0<j<j_i}\widetilde{\Delta}_{i,j}^{-})<\frac{1}{K}(x_{j_i}-\inf\Delta_i).
\end{equation*}
In the same way,
\begin{equation*}
m(\cup_{N_0<j<j_i}\widetilde{\Delta}_{i,j}^{+})<\frac{1}{K}(\sup\Delta_i-\overline{x}_{j_i}),
\end{equation*}
where $\overline{x}_{j_i}=\sup\widetilde{\Delta}_{i,j_i}^{+}$. Hence
\begin{equation*}
m(\cup_{N_0<j<j_i}\widetilde{\Delta}_{i,j})<\frac{1}{K}d_i+\frac{1}{K}(x_{j_i}-\overline{x}_{j_i})<\frac{1}{K}d_i,
\end{equation*}
i.e.~\eqref{12b'} holds and so does~\eqref{12b}. By Lemma~\ref{lm9}, \eqref{12c} is also true. Therefore by~\eqref{ety}
\begin{equation*}
m(\cup_{1\leq i\leq n}\cup_{j\in\mathbb{N}}\Delta_{i,j})<\widetilde{\eta}+\frac{1}{K}+\widehat{\eta}<\overline{\eta}.
\end{equation*}

If for some $1\leq i\leq n$ we have $j_i\leq N_0$, then
\begin{equation*}
\Delta_{i,j_i}\subset \cup_{1\leq i\leq n}\cup_{j\leq N_0}\Delta_{i,j},
\end{equation*}
we obtain the same result and so the claim follows.
\end{proof}

\section{IETs with balanced partition lengths}\label{balparle}

Let $T\colon [0,1)\to[0,1)$ be an irrational rotation on the circle $\mathbb{T}=\mathbb{R}/\mathbb{Z}$. It is well-known that a necessary and sufficient condition for $\alpha$ to have bounded partial quotients is that the rotation by $\alpha$ on the circle has balanced partition lengths. The main concern in this section is with giving more examples of interval exchange transformations with balanced partition lengths. In particular, we show that every IET which is of periodic type has balanced partition lengths.

\begin{uw}\label{uwag}
Let $x_0<x_1<\dots<x_{k-1}<x_k$ and consider $\mathcal{P}=\mathcal{P}(x_0;x_1,\dots,x_{k-1};x_k)$. Suppose that
\begin{equation*}
\frac{1}{c}\leq \frac{x_{i+1}-x_i}{x_{j+1}-x_j}\leq c
\end{equation*}
for some $c\geq 1$ for all $0\leq i,j\leq k-1$. Then for $0\leq j\leq k-1$ we have
\begin{equation*}
\frac{k}{c}\leq \sum_{i=0}^{k-1}\frac{x_{i+1}-x_i}{x_{j+1}-x_j}\leq kc,
\end{equation*}
which is equivalent to 
\begin{equation*}
\frac{k}{c}\leq \frac{x_k-x_0}{x_{j+1}-x_j}\leq kc.
\end{equation*}
Hence
\begin{equation*}
(x_k-x_0)\frac{1}{ck}\leq x_{j+1}-x_j \leq (x_k-x_0)\frac{c}{k}
\end{equation*}
for all $0\leq j\leq k-1$ and
\begin{equation*}
(x_k-x_0)\frac{1}{ck}\leq \min\mathcal{P}\leq\max\mathcal{P}\leq (x_k-x_0)\frac{c}{k}.
\end{equation*}
\end{uw}

\begin{lm}\label{kazdypjestb}
Every IET of periodic type has balanced partition lengths.
\end{lm}

Before we prove the above lemma, let us recall some notation from Section~\ref{defin}.
Recall that $A\colon \mathbf{\Delta} \to SL(r,\mathbb{Z})$ stands for the Rauzy cocycle, $\mathcal{R}$ stands for the Rauzy induction map and $\mathcal{R}^{n}(T)\colon I^{(n)}\to I^{(n)}$ for $n\geq 0$. Recall that for an IET $T$ of periodic type the sequence $A(T),\ A(\mathcal{R}T), \dots, A(\mathcal{R}^nT)$ is periodic with some period $p>0$ and the period matrix $A^{(p)}(T)$ has strictly positive entries. For a matrix $B\in SL(r,\mathbb{Z})$ with strictly positive entries recall that
\begin{equation*}
\overline{\nu}(B)=\max_{i,j,l}\frac{B_{ij}}{B_{lj}}.
\end{equation*}
We will also need inequalities~\eqref{balanced}, i.e.
\begin{equation*}
\frac{1}{\overline{\nu}(B)}\leq \frac{h_i^{(n+m)}}{h_j^{(n+m)}} \leq \overline{\nu}(B)
\end{equation*}
which hold whenever $h^{(m+n)}=B\cdot h^{(n)}$.

\begin{proof}[Proof of Lemma~\ref{kazdypjestb}]
Suppose that $T$ is of periodic type. Let $p$ be a period of the Rauzy matrices such that $A^p(T)$ has only strictly positive entries. Denote the period matrix $A^{(p)}(T)$ by $B$. Put 
$\rho=\left|(B^T)^{-1}\underline{\lambda}^{(0)}\right|$. We will prove now that condition (i) of Definition~\ref{bal} is fulfilled. Note that IETs of periodic type automatically satisfy the IDOC since Rauzy induction is well-defined for all steps. Therefore they are also minimal and we can choose $K_1\in \mathbb{N}$ such that for $1\leq i\leq r$ there exist $1\leq k_i \leq K_1$ satisfying
\begin{equation}\label{kai}
T^{-k_i}\beta_{i} \in I_{i}^{(0)},
\end{equation}
where $I_{i}^{(0)}$ is the leftmost interval exchanged by $T$.

We will show that there exists $M_1\in\mathbb{N}$ such that
\begin{equation}\label{wysok}
(K_1+1)h^{((M+1)p)}_{\max}\leq h^{((M+M_1)p)}_{\min}
\end{equation}
for every $M \in \mathbb{N}$.
Let $M_1\in\mathbb{N}$ satisfy
\begin{equation}\label{M0}
\frac{K_1+1}{r^{M_1-1}}\overline{\nu}(B) \leq 1.
\end{equation}
Since $h^{((M+1)p+p)}=B\cdot h^{((M+1)p)}$, for every $1\leq i\leq r$,
\begin{equation}\label{De}
h^{((M+1)p+p)}_i \geq r h^{((M+1)p)}_{\min}.
\end{equation}
Moreover, from~\eqref{balanced} we have
\begin{equation}\label{row}
h^{((M+1)p)}_{\max}\leq \overline{\nu}(B) h^{((M+1)p)}_{\min}.
\end{equation}
Therefore
\begin{multline}\label{dluga}
(K_1+1)h^{((M+1)p)}_{\max} \leq (K_1+1)\overline{\nu}(B)h^{((M+1)p)}_{\min} 
\\
\leq (K_1+1)\frac{1}{r^{M_1-1}}\overline{\nu}(B) h^{((M+M_1)p)}_{\min}\leq h_{\min}^{((M+M_1)p)}
\end{multline}
where the left inequality follows from~\eqref{row}, the middle one is obtained by iterating~\eqref{De} $M_1-1$ times and the right one is a consequence of~\eqref{M0}.
This implies~\eqref{wysok}.

Fix $j\in\mathbb{N}$. Let $M \in \mathbb{N}$ satisfy
\begin{equation}\label{zarowka}
(K_1+1)h^{(Mp)}_{\max}\leq j < (K_1+1)h^{((M+1)p)}_{\max}.
\end{equation}
From~\eqref{wysok} we have
\begin{equation}\label{tri}
j \leq h^{((M+M_1)p)}_{\min}.
\end{equation}

Now we will obtain a lower bound for $\min\mathcal{P}_j$. Cut the towers for $\mathcal{R}^{((M+M_1)p)}$ at the points $\rho^{M+M_1} \cdot T^{-1}\beta_i$ ($1\leq i\leq r-1$). Let $\mathcal{P}$ stand for the partition of the interval $[0,1)$ after cutting the towers. We claim that now the set of the partition points of $\mathcal{P}$ includes the set $\{T^{-s}\beta_i\colon 0\leq s\leq j-1,\ 1\leq i \leq r-1\}$. Indeed, the discontinuity points for the induced IET $\mathcal{R}^{((M+M_1)p)}(T)$ are the first iterations of the initial discontinuity points $\beta_i$ ($1\leq i \leq r-1$) via $T^{-1}$ which are in $I^{((M+M_1)p)}$. This means that the points $\beta_i \colon 1\leq i\leq r-1$ belong to the set of the left ends of the floors of the towers. Otherwise, after some iterations via $T^{-1}$ we would get that the discontinuity points of the new transformation $\mathcal{R}^{((M+M_1)p)}(T)$ are \emph{inside} the intervals exchanged by it, which is impossible. In view of the inequality~\eqref{tri} we obtain
\begin{equation*}
\#\left(\left\{T^{-s}\beta_i\colon 0\leq s\leq j-1\right\}\cap I^{\left(\left(M+M_1\right)p\right)}\right)\leq 1
\end{equation*}
for every $1\leq i \leq r$.
Therefore the partition $\mathcal{P}$ is finer than $\mathcal{P}_j$ and
\begin{equation*}
\min\mathcal{P} \leq \min\mathcal{P}_j.
\end{equation*}
We will now estimate $\min\mathcal{P}$ from below. Let $c_2>1$ be such that
\begin{equation*}
\frac{1}{c_2 r} \leq \min\mathcal{P}_2
\end{equation*}
($\mathcal{P}_2$ is the same as $\mathcal{P}_j$ for $j=2$).
Hence, by the definition of $\mathcal{P}$
\begin{equation*}
\min\mathcal{P} = \rho^{M+M_1}\min\mathcal{P}_2 \geq \rho^{M+M_1} \frac{1}{c_2r}.
\end{equation*}
This is therefore also the lower bound which we were looking for:
\begin{equation}\label{est1}
\min\mathcal{P}_j\geq \rho^{M+M_1} \frac{1}{c_2r}.
\end{equation}

Now we will estimate $\max\mathcal{P}_j$ from above. Consider the towers for $\mathcal{R}^{(Mp)}(T)$. From the left inequality in~\eqref{zarowka} and the definition of $K_1$, in each floor of each tower there is at least one partition point of $\mathcal{P}_j$. Therefore from the definition of $\rho$
\begin{equation*}
\max\mathcal{P}_j\leq 2 \rho^{M}\max (\mathcal{P}_1).
\end{equation*}
Let $c_{1}>1$ satisfy
\begin{equation}\label{c1}
\max (\mathcal{P}_{1}) \leq \frac{c_{1}}{r}.
\end{equation}
Hence
\begin{equation}\label{est2}
\max{\mathcal{P}_j}\leq 2 \rho^M \max\mathcal{P}_{1} \leq 2\rho^M \frac{c_1}{r}.
\end{equation}
Combining~\eqref{est1} and~\eqref{est2} we obtain
\begin{equation}\label{wspolm}
\rho^{M+M_1}\frac{1}{c_2r}\leq \min\mathcal{P}_j\leq \max\mathcal{P}_j \leq 2\rho^{M}\frac{c_1}{r}.
\end{equation}
Therefore the ratio of the lengths of the intervals of the partition $\mathcal{P}_j$ is between
\begin{equation*}
2\rho^M\frac{c_1}{r}\cdot\frac{c_2r}{\rho^{M+M_1}}=\frac{2c_1c_2}{\rho^{M_1}} \textrm{ and } \frac{\rho^{M_1}}{2c_1c_2}.
\end{equation*}
Hence, as $\mathcal{P}_j$ is a partition into $(r-1)j+1$ subintervals, $(r-1)j+1<rj$ and $\frac{1}{(r-1)j+1}>\frac{r}{j}$, in view of Remark~\ref{uwag} we have
\begin{equation*}
\frac{1}{j}\cdot \frac{2c_1c_2}{r\rho^{M_1}}\leq \min\mathcal{P}_j\leq \max\mathcal{P}_j<\frac{1}{j}\cdot\frac{r\rho^{M_1}}{2c_1c_2}.
\end{equation*}
Thus, we have proved that IETs of periodic type fulfill condition (i) of Definition~\ref{bal}. 

The proof of condition (ii) is similar. We use the notation introduced in the first part of the proof. Fix $0\leq i_0\leq r-1$ and let $K_2$ be a natural number such that for $1\leq i\leq r$ there exist $1\leq k_{i}^{-},k_{i}^{+}\leq K_2$ such that
\begin{equation*}
T^{-k_{i}^{-}}\beta_{i_0}\in I_i^{(0)} \text{ and } T^{k_{i}^{+}}\beta_{i_0}\in I_i^{(0)}
\end{equation*}
(such a number $K_2$ exists from the minimality\footnote{Since $T$ is of periodic type, all steps of Rauzy induction are well-defined. Therefore $T$ satisfies IDOC, whence also $T^{-1}$ satisfies IDOC. This implies minimality of $T^{-1}$.}). Fix $j\geq 1$ and $0\leq k_0\leq j-1$.
As in the first part of the proof, there exists $M_2\in\mathbb{N}$ such that for every 
$M\in\mathbb{N}$
\begin{equation}\label{wysok1}
\left(K_2+1\right)h_{\max}^{\left(\left(M+1\right)p\right)}\leq \frac{1}{4}h_{\min}^{((M+M_2)p)}.
\end{equation}
Let $M\in\mathbb{N}$ satisfy
\begin{equation}\label{zarowka1}
(K_2+1)h_{\max}^{(Mp)}\leq \left\lfloor \frac{j-1}{2}\right\rfloor < (K_2+1)h_{\max}^{((M+1)p)},
\end{equation}
where $\lfloor \cdot \rfloor$ is the floor function. From the right inequality in~\eqref{zarowka1} and from ~\eqref{wysok1} we have 
\begin{equation}\label{tri1}
j\leq h_{\min}^{((M+M_2)p)}.
\end{equation}

Now we will obtain a lower bound for $\min\mathcal{P}\left(\{T^{-k_0+l}\beta_{i_0}\colon 0\leq l\leq j-1\}\right)$. Cut the towers for $\mathcal{R}^{((M+M_2)p)}$ at the points $\rho^{M+M_2}T^{-1}\beta_{i_0}$ and $\rho^{M+M_2}T\beta_{i_0}$. Let $\mathcal{P}_{i_0}$ stand for the partition of the interval $[0,1)$ after cutting the towers. Since the point $\beta_{i_0}$ is a left end of some floor of some tower for $\mathcal{R}^{((M+M_2)p)}(T)$ (see the first part of the proof), in view of the inequality~\eqref{tri1} we obtain
\begin{equation*}
\#\left(\left\{T^{-s}\beta_{i_0}\colon 0\leq s\leq j-1\right\}\cap I^{((M+M_2)p)}\right)\leq 1
\end{equation*}
and
\begin{equation*}
\#\left(\left\{T^{s}\beta_{i_0}\colon 0\leq s\leq j-1\right\}\cap I^{((M+M_2)p)}\right)\leq 1.
\end{equation*}
Therefore the partition $\mathcal{P}_{i_0}$ is finer than $\mathcal{P}(\{T^{k_0+l}\beta_{i_0}\colon 0\leq l\leq j-1\})$
and 
\begin{equation*}
\min\mathcal{P}\left(\left\{T^{k_0+l}\beta_{i_0}\colon 0\leq l\leq j-1\right\}\right)\geq \min\mathcal{P}_{i_0}.
\end{equation*}
Let $\widetilde{c}_2>1$ be such that
\begin{equation*}
\frac{1}{\widetilde{c}_2r}\leq \min\mathcal{P}(\{T\beta_{i_0},T^{-1}\beta_{i_0},\beta_{i}\colon 0\leq i\leq r-1 \}).
\end{equation*}
Hence, by the definition of $\mathcal{P}^{L}_{i_0}$
\begin{equation*}
\min\mathcal{P}_{i_0}=\rho^{M+M_2}\min\mathcal{P}\left(\left\{T\beta_{i_0},T^{-1}\beta_{i_0},\beta_i\colon 0\leq i\leq r-1\right\}\right) \geq \rho^{M+M_2}\frac{1}{\widetilde{c}_2r}.
\end{equation*}

Now we will estimate $\max\mathcal{P}(\{T^{-k_0+l}\beta_{i_0}\colon 0\leq l\leq j-1\})$ from above. Consider the towers for $\mathcal{R}^{(Mp)}(T)$ and cut them at the points $\rho^{M} T^{-s}\beta_{i_0}$ and $\rho^{M} T^s\beta_{i_0}$ ($1\leq s\leq K_2$). Among the points $\{T^{-k_0+l}\beta_{i_0}\colon 0\leq l\leq j-1\}$ there are either at least $\left\lfloor \frac{j-1}{2}\right\rfloor$ backward iterations of $\beta_{i_0}$ or at least $\left\lfloor \frac{j-1}{2}\right\rfloor$ forward iterations of $\beta_{i_0}$. In either case we conclude from the left inequality in~\eqref{zarowka1} and the definition of $K_2$ that in each floor of each tower there is at least one point of the form $T^{-k_0+l}\beta_{i_0}$ ($0\leq l\leq j-1$). Notice that each floor of each tower is an interval. Therefore and from the definition of $\rho$
\begin{equation*}
\max\mathcal{P}(\{T^{-k_0+l}\beta_{i_0}\colon 0\leq l\leq j-1\})\leq 2\rho^{M}\max\mathcal{P}_1\leq 2\rho^M\frac{c_1}{r}
\end{equation*}
($c_1$ was defined in~\eqref{c1}). To end the proof of condition (ii) we apply the same arguments as in the end part of the proof of (i).
\end{proof}

\section{From absence of partial rigidity to absence of self-similarities}\label{selfsim}
\subsection{Weak convergence and ``non-stretching'' of Birkhoff sums}
An important tool for us will be the following result, which will allow us to use Lemma~\ref{lm:fl08}. 

\begin{tw}[\cite{FL05}]\label{tw15}\footnote{For more details concerning this theorem see Section~\ref{se:out}.}
Let $T\colon (X,\mathcal{B},\mu)\to(X,\mathcal{B},\mu)$ be an ergodic automorphism and $f\in L^2 (X,\mu)$ a positive function for which there exists $c>0$ such that $0<c\leq f(x)$ for a.a. $x\in X$. Suppose that $\{D_n\}$ is a sequence of Borel subsets of $X$, $\{q_n\}$ is an increasing sequence of natural numbers, and $\{a_n\}$ is a sequence of real numbers such that
\begin{itemize}
\item
$\mu(D_n)\to a >0$ as $n\to\infty$,
\item
$\mu(D_n\triangle T^{-1}D_n)\to 0$ as $n \to\infty$,
\item
$\sup_{x \in D_n} d(x,T^{q_n}x) \to 0$,
\item
the sequence $\{ \int_{D_n} |f^{(q_n)}(x)-a_n|^2 d \mu (x)\}$ is bounded,
\item
$\frac{1}{\mu(D_n)}\left(\left(f^{(q_n)}(x)-a_n\right)|_{D_n}\right)_{\ast}\left(\mu|_{D_n}\right) \to P$ weakly in $\mathcal{P}(\mathbb{R})$ the set of probability Borel measures on $\mathbb{R}$,
\item
the sequence $\{(T^f)_{a_n}\}$ converges in the weak operator topology.
\end{itemize}
Then for some $J \in J(T^f)$, $\{(T^f)_{a_n}\}$ converges weakly to the operator
$a \int_{\mathbb{R}} (T^f)_t d P(t) + (1-a) J$.
\end{tw}

We claim, that Theorem~\ref{tw15} is applicable in our case, i.e. where the roof function is given by $f+g$ (function $f$ is defined by~\eqref{fczyste} and the equality~\eqref{symetry} holds, i.e. the singularities are of \emph{symmetric type}, $g$ is piecewise absolutely continuous and continuous whenever $f$ is so). As sets $D_n$ we take the ``rigidity sets'' $C_n$ constructed by~C.~Ulcigrai in~\cite{0901.4764}. They are a modification of the sets used by A.~Katok in~\cite{Katok80} to show that IETs are never mixing. C.~Ulcigrai considers a more general class of flows than us, namely IETs which admit so-called balanced return times (for the definition and more details we refer to~\cite{0901.4764}). It is shown that there exist a sequence $\{C_n\}_{n\in\mathbb{N}}$ of measurable subsets $C_n\subset [0,1)$, a sequence $\{q_n\}_{n\in\mathbb{N}}$, $q_n\in\mathbb{N}$, a sequence of finite partitions $\{\xi_n\}_{n\in\mathbb{N}}$ of $[0,1)$ and $M>0$ such that
\begin{itemize}
\item[(i)]
$m(C_n)>a$ for some positive constant $a$,
\item[(ii)]
for any $F\in\xi_n$, $T^{q_n}(F\cap C_n)\subset F$,
\item[(iii)]
$\max\{\textrm{diam}(F)\colon F\in\xi_n\} \to 0$ as $n\to \infty$,
\item[(iv)]
$\left|f^{(q_n)}(x)-f^{(q_n)}(y)\right|<M$ for all $x,y\in C_n$.
\end{itemize}
The construction is carried out in such a way that the sets $C_n$ are unions of levels of towers with appropriately chosen sets in the base, in particular the diameters of these base sets converge to zero as $n$ tends to infinity. Therefore
$m(C_n\triangle T^{-1}(C_n))\to 0$ as $n\to\infty$.

Notice that the conditions $(ii)$ and $(iii)$ imply
\begin{equation*}
\sup_{x\in C_n} d(x,T^{q_n}x)\to 0 \textrm{ as }n\to \infty.
\end{equation*}
The condition $(iv)$ was used first by A.~V.~Kochergin in~\cite{Kochergin76}. He proved it to be a sufficient condition for a special flow to be not mixing, provided that there exist rigidity sets for the base automorphism, i.e. sets such that the conditions $(i)$, $(ii)$ and $(iii)$ are fulfilled. Moreover, in~\cite{Katok80} it was shown that for any function $h$ of bounded variation
\begin{itemize}
\item[(v)]
$| h^{(q_n)}(x)-h^{(q_n)}(y) |<M'$
\end{itemize}
for some constant $M$ for all $x,y \in C_n$.

From~$(iv)$ and~$(v)$ with $h=g$ it follows that
\begin{equation*}
\left\{ \int_{C_n}\left|(f+g)^{(q_n)}(x)-a_n\right|^2dm(x) \right\}_{n\in\mathbb{N}},
\end{equation*}
where $a_n=(f+g)^{(q_n)}(x_0)$ for some $x_0\in C_n$, is bounded. The distributions $\left( (f+g)^{(q_n)}(x)-a_n \right)_{\ast}(m|_{C_n})$ are uniformly tight and we may assume (passing to a subsequence if necessary) that
\begin{equation*}
\left( (f+g)^{(q_n)}(x)-a_n \right)_{\ast}(m|_{C_n})\to P
\end{equation*} 
weakly in $\mathcal{P}(\mathbb{R})$ for some measure $P$. From separability (passing again to a subsequence if needed), we deduce that $\{(T^{f+g})_{a_n}\}$ converges in the weak operator topology.


\subsection{The absence of self-similarities}
We will prove now Theorem~\ref{tw:g}. We will use the Lemma~\ref{lm:fl08}\cite{FL08} recalled in the introduction. Let us first prove a counterpart of Theorem~\ref{tw:g} expressed in terms of the special flow representation.
\begin{tw}\label{tw18}
Assume that $T^{f+g}$ is a special flow over an IET $T$ with balanced partition lengths, $f$ is given by equation~\eqref{fczyste} and satisfies~\eqref{symetry} and $g$ is a piecewise absolutely continuous function (continuous whenever $f$ is continuous), such that $f+g>0$. Then $T^{f+g}$ is not self-similar.
\end{tw}
\begin{proof}
By Theorem~\ref{tw15} we have that $(T^{f+g})_{a_n}$ converges weakly to the operator $a\int_{\mathbb{R}}(T^{f+g})_tdP(t)+(1-a)J$ for some $J\in J(T^f)$. From Theorem~\ref{tw12} the considered flow is not partially rigid. Therefore, from Lemma~\ref{lm:fl08} we conclude the absence of self-similarities.
\end{proof}

Theorem~\ref{tw:g} announced in the introduction now easily follows.
\begin{proof}[Proof of Theorem~\ref{tw:g}]
The claim follows directly from the discussion in Section~\ref{reduction}, by Lemma~\ref{kazdypjestb} and Theorem~\ref{tw18}.
\end{proof}

\subsection{The absence of spectral self-similarities}\label{se:6.3}
In this section we discuss the problem of the absence of spectral self-similarities. With minor modification we follow the approach proposed in~\cite{FL08}. To begin with, let us give a formal definition which is the spectral counterpart of the notion of the set of scales of self-similarities. By $M(L^2(X,\mu))$ we denote the convex set of Markov operators $V \colon L^2(X,\mu) \to L^2(X,\mu)$, i.e. $V$ is a positive operator such that $V(1)=1$ and $V^\ast (1)=1$. Let $\mathcal{V}=(V_t)_{t\in\mathbb{R}}$ be a continuous representation of $\mathbb{R}$ in $M(L^2(X,\mu))$. Representations $\mathcal{V}=(V_t)_{t\in\mathbb{R}}$ and $\mathcal{V}'=(V'_t)_{t\in\mathbb{R}}$ are said to be spectrally isomorphic if there exists a unitary operator $U \colon L^2(X,\mu) \to L^2(X,\mu)$ such that $U \circ V'_t = V_t \circ U$ for all $t\in\mathbb{R}$.
\begin{df}
The \emph{set of scales of spectral self-similarities} is given by
\begin{equation*}
I_{sp}(\mathcal{V})=\left\{ s\in\mathbb{R} \setminus \{0\}\colon \mathcal{V} \text{ and } \mathcal{V}_s \text{ are spectrally isomorphic}\right\}.
\end{equation*}
If $I_{sp}(\mathcal{V}) \subset \{-1,1\}$, we say that $\mathcal{V}$ has \emph{no spectral self-similarities}.
\end{df}
Let $R_s\colon \mathbb{R} \to \mathbb{R}$ stand for the rescaling map $R_s(t)=st$. Denote by $\mathcal{P}(\mathbb{R})$ the set of all probability Borel measures on $\mathbb{R}$. Let $P_{s}=(R_{s})_{\ast}(P)$.



\begin{uw}\label{rm:FL08}
As noticed in~\cite{FL08},
\begin{equation*}
\int_{\mathbb{R}} V_t d P_{s_n} \to I
\end{equation*}
whenever $P \in \mathcal{P}(\mathbb{R})$ and $s_n\to 0$.

\end{uw}

The next lemma is a modification of Lemma 6.3 in~\cite{FL08}. Let $\left\{V_t \colon t\in \mathbb{R} \right\}^d$ stand for the closure of $\left\{V_t \colon t\in \mathbb{R} \right\}$ in the weak operator topology.
\begin{lm}\label{zmodyfikowane}
Suppose that there exists $x\in I_{sp} (\mathcal{V}) \setminus \{-1,1\}$ and there exists $P\in\mathcal{P}(\mathbb{R})$ and $0<a\leq 1$ such that
\begin{equation*}
a\int_{\mathbb{R}} V_tdP(t) + (1-a)J \in \left\{V_t \colon t\in \mathbb{R} \right\}^d
\end{equation*}
for some $J \in M(L^2(X,\mu))$. Then
\begin{equation*}
a I + (1-a) K \in \{V_t \colon t\in\mathbb{R}\}^d
\end{equation*}
for some contraction $K$ on $L^2(X,\mu)$.
\end{lm}
\begin{proof}
Since $s \in I_{sp}(\mathcal{V})$, there exists a unitary operator $U \colon L^2(X,\mu) \to L^2(X,\mu)$ such that $U \circ  V_{st}=V_t \circ U$ for all $t \in \mathbb{R}$. Therefore,
\begin{equation*}
U^m \circ V_{s^mt}=V_t \circ U^m \text{ for every }t\in\mathbb{R} \text{ and }m\in\mathbb{Z}.
\end{equation*}
By the assumption, there exists a sequence $(t_n)$ such that $|t_n| \to +\infty$ and
\begin{equation*}
V_{t_n} \to a \int_{\mathbb{R}}V_t\ dP(t) + (1-a)J \text{ weakly}.
\end{equation*}
It follows that
\begin{multline*}
V_{s^mt_n}=U^{-m} \circ V_{t_n} \circ U^m \to a\int_{\mathbb{R}} U^{-m}\circ V_t \circ U^m dP(t) + (1-a)J_m
\\
=a\int_{\mathbb{R}} V_{s^mt}dP(t)+(1-a)J_m = a\int_{\mathbb{R}}V_t\ dP_{s^m}(t)+(1-a)J_m,
\end{multline*}
where $J_m=U^{-m} \circ J \circ U^m$. Hence
\begin{equation*}
a\int_{\mathbb{R}} V_t\ dP_{s^m}(t)+(1-a)J_m \in \{V_t \colon t\in \mathbb{R}\}^d.
\end{equation*}
Assume that $|s|<1$, in the case $|s|>1$ the proof follows by the same method by taking the sequence $(s^{-m})_{m=1}^{\infty}$ instead of $(s^m)_{m=1}^{\infty}$. By passing to a subsequence if necessary, we can assume that $J_m \to K$ weakly, where $K$ is a contraction.\footnote{Every Markov operator is a contraction, see e.g. A.~M.~Vershik~\cite{MR0476998}.} Since $s^m \to 0$ as $m \to +\infty$, by Remark~\ref{rm:FL08},
\begin{equation*}
a \int_{\mathbb{R}} V_t d\left(R_{s^m}\right)_{\ast}(P)(t) + (1-a)J_m \to aI+(1-a)K \text{ as }m\to +\infty.
\end{equation*}
Thus 
\begin{equation*}
a I+(1-a)K \in \left(\left\{V_t \colon t \in\mathbb{R} \right\}^d\right)^d= \{V_t \colon t\in\mathbb{R}\}^d.
\end{equation*}
\end{proof}

\begin{uw}
Notice that the only difference between Lemma~\ref{zmodyfikowane} and Lemma 6.3 in~\cite{FL08} is that the obtained operator $K$ is a contraction, not necessarily a Markov operator.
\end{uw}

The following theorem is a spectral counterpart of Theorem~6.4 in~\cite{FL08}. Notice that in the second part of the preceding theorem we need to assume that $1/2< a \leq 1$. For the role of $1/2$ see also Example~\ref{uw:spnotin} and Proposition~\ref{pr:last}.
\begin{tw}\label{tw:spco}
Let $\mathcal{T}=\{T_t\}_{t\in\mathbb{R}}$ be a measure-preserving flow on $(X,\mu)$ such that $\mathcal{T}$ is spectrally isomorphic to $\mathcal{T}\circ s$ for some $s \neq \pm 1$.
\begin{itemize}
\item
If $\int_{\mathbb{R}}T_tdP(t)$ belongs to $\{T_t\colon t \in\mathbb{R}\}^d$ for some $P\in \mathcal{P}(\mathbb{R})$ then $\mathcal{T}$ is rigid.
\item
If $a\int_{\mathbb{R}}T_tdP(t)+(1-a)J \in \{T_t\colon t\in\mathbb{R}\}^d$ for some $1/2< a \leq 1$, $P\in\mathcal{P}(\mathbb{R})$ and $J\in J(\mathcal{T})$ then $\mathcal{T}$ is partially rigid.
\end{itemize}
\end{tw}
\begin{proof}
The first part of the claim follows directly from Lemma~\ref{zmodyfikowane}. To prove the second part suppose that $a\int_{\mathbb{R}}T_tdP(t)+(1-a)J \in \{T_t\colon t\in\mathbb{R}\}^d$ for some $1/2< a \leq 1$. By Lemma~\ref{zmodyfikowane} for any measurable set $A\subset X$ we have
\begin{multline*}
\liminf_{n \to \infty} \mu(T_{t_n}A \cap A) =a \mu (A)+(1-a)\langle K\mathbbm{1}_A,\mathbbm{1}_A \rangle 
\\
\geq a\mu(A)- (1-a)\mu(A)=(2a-1)\mu(A).
\end{multline*}
Since $0< 2a-1\leq 1$, the proof is complete.
\end{proof}

\begin{wn}\label{wnspektralny}
If $\mathcal{T}$ is non-rigid and $\int_{\mathbb{R}}T_t\ dP(t)$ belongs to $\{T_t \colon t\in\mathbb{R}\}^d$ for some $P \in \mathcal{P}(\mathbb{R})$ then $\mathcal{T}$ has no spectral self-similarities. If $\mathcal{T}$ is not partially rigid and $a \int_{\mathbb{R}} T_t\ dP(t)+(1-a)J$ belongs to $\{ T_t\colon t \in\mathbb{R}\}^d$ for some $P\in\mathcal{P}(\mathbb{R})$, $1/2< a \leq 1$ and $J \in J(\mathcal{T})$ then $\mathcal{T}$ has no spectral self-similarities.
\end{wn}

\begin{ex}\label{przykladzik}
Consider a special flow $T^f$ built over a rotation on the circle $T \colon [0,1) \to [0,1)$ by $\alpha$: $Tx=x+\alpha$, where $\alpha$ is an irrational number with bounded partial quotients and under symmetric logarithmic function $f(x)=-a(\log\{x\})+\log\{-x\})+h(x)$, where $a>0$ and $h \colon [0,1) \to \mathbb{R}$ is an absolutely continuous function. By Theorem~\ref{tw12} $T^f$ is not partially rigid and therefore also not rigid. By Theorem~\ref{tw15} (see the discussion in Section~\ref{selfsim}) there exists a sequence $(a_n)$ such that $T_{a_n}^f$ converges weakly to the operator $\int_{\mathbb{R}}T_t^fdP(t)$ (rotation is a rigid transformation and as sets $D_n$ in Theorem~\ref{tw15} we can take the whole interval $[0,1)$ - this is why there is only one term in the limit operator). By Corollary~\ref{wnspektralny} it follows that $T^f$ has no spectral self-similarities.
\end{ex}

The flow in Example~\ref{przykladzik} doesn't belong to the family of flows on surfaces considered by us in this paper. However, there exist smooth flows on surfaces of any genus $\mathbf{g} \geq 2$ which yield this representation. To construct them, it is necessary to allow saddle connections. For more details we refer to~\cite{FL03}.

We will give now two examples showing that partial rigidity is not a spectral invariant. Let us begin by giving a common background for these two examples. Consider an ergodic automorphism $T\colon X\to X$ which is rigid and a cocycle $\varphi \colon X\to \mathbb{Z}_2=\{0,1\}$ such that automorphism $T_{\varphi}\colon X\times\mathbb{Z}_2 \to X\times \mathbb{Z}_2$ given by $T_{\varphi}(x,g)=(Tx,\varphi(x)+g)$ has Lebesgue spectrum on the space $L^2(X\times \mathbb{Z}_2)\ominus L^2(X)\otimes \mathbbm{1}$. Such a cocycle exists for any ergodic, rigid automorphism $T$ (see H.~Helson, W.~Parry~\cite{HP78}). Let $S\colon Y\to Y$ be a Bernoulli automorphism and consider $T \times S \colon X\times Y \to X\times Y$. Notice that $T\times S$ is not partially rigid, whereas $T_{\varphi}$ is partially rigid with rigidity constant $\alpha=1/2$ (see Corollary~1.2. in \cite{Ageev09}).

\begin{ex}\label{uw:spnotin}
Assume additionally that $T$ is an ergodic rotation on a compact abelian group $X$, which has an infinite, closed subgroup $X_0$ such that the quotient space $X/X_0$ is inifnite.\footnote{These assumptions are fulfilled e.g. by $X=\mathbb{T}\times \mathbb{T}$.} Then there exists a cocycle $\varphi\colon X \to \mathbb{Z}_2$ such that automorphism $T_{\varphi}$ has countable Lebesgue spectrum on $L^2(X\times \mathbb{Z}_2)\ominus L^2(X)\otimes \mathbbm{1}$. 

We claim that $T\times S$ has the same spectrum as $T_{\varphi}$. Indeed, we have
\begin{equation*}
L^2_0(X\times Y)=(L^2_0(X)\otimes \mathbbm{1}) \oplus (\mathbbm{1}\otimes L^2_0(Y)) \oplus (L^2_0(X)\otimes L^2_0(Y)),
\end{equation*}
and
\begin{equation*}
L^2_0(X\times \mathbb{Z}_2)= L^2_0(X) \otimes \mathbbm{1} \oplus \left(L^2_0(X\times \mathbb{Z}_2)\ominus L^2_0(X) \otimes \mathbbm{1}\right).
\end{equation*}
Notice that
\begin{itemize}
\item
on $L^2_0(X)\otimes \mathbbm{1}$ spectrum of automorphism $T\times S$ and spectrum of automorphism $T_\varphi$ is the same as spectrum of automorphism $T$ on $L^2_0(X)$,
\item
on $\mathbbm{1}\otimes L^2_0(Y)$ spectrum of automorphism $T\times S$ is the same as spectrum of automorphism $S$, i.e. Lebesgue with infinite multiplicity,
\item
on $L^2_0(X)\otimes L^2_0(Y)$ maximal spectral type of automorphism $T\times S$ is equal to $\sigma_T\ast \sigma_S=\sigma_T \ast \lambda_{\mathbb{T}}\equiv \lambda_{\mathbb{T}}$. 
\end{itemize}
Therefore $T\times S$ and $T_{\varphi}$ have the same spectrum whence they are spectrally isomorphic.
\end{ex}

\begin{ex}\label{uw:spnotin2}
We claim that under the assuptions listed directly before Example ~\ref{uw:spnotin} (without imposing additional properties on $X$ and $T$, i.e. in particular $T$ can be weakly mixing), automorphisms $T_{\varphi}\times T_{\varphi}$ and $T\times S \times T \times S$ have the same spectrum. Indeed, notice that
\begin{itemize}
\item
on 
\begin{multline*}
H_1:=(L^2_0(X)\otimes \mathbbm{1}\otimes L^2_0(X)\otimes \mathbbm{1})\oplus  (L^2_0(X)\otimes \mathbbm{1} \otimes \mathbbm{1}\otimes \mathbbm{1})\oplus\\
 \oplus (\mathbbm{1}\otimes \mathbbm{1} \otimes L^2_0(X)\otimes \mathbbm{1})
\end{multline*}
automorphism $T\times S\times T\times S$ has the same spectrum as automorphism $T\times T$ on $L^2_0(X\times X)$,
\item
on $H_2:=\mathbbm{1}\otimes L^2_0(Y)\otimes \mathbbm{1}\otimes\mathbbm{1}$ automorphism $T\times S\times T\times S$ has Lebesgue spectrum of infinite multiplicity,
\item
on
\begin{equation*}
L^2_0(X\times Y\times X\times Y) \ominus (H_1 \oplus H_2)
\end{equation*}
as in Example~\ref{uw:spnotin}, maximal spectral type of automorphism $T\times S\times T\times S$ is Lebesgue measure.
\end{itemize}
Moreover
\begin{itemize}
\item
on
\begin{multline*}
H:=(L^2_0(X)\otimes \mathbbm{1}\otimes L^2_0(X)\otimes \mathbbm{1})\oplus  (L^2_0(X)\otimes \mathbbm{1} \otimes \mathbbm{1}\otimes \mathbbm{1})\oplus\\
 \oplus (\mathbbm{1}\otimes \mathbbm{1} \otimes L^2_0(X)\otimes \mathbbm{1})
\end{multline*}
automorphism $T_\varphi \times T_\varphi$ has the same spectrum automorphism $T\times T$ on $L^2_0(X\times X)$,
\item
on $L^2_0(X\times \mathbb{Z}_2\times X \times \mathbb{Z}_2) \ominus H$ automorphism $T_\varphi\times T_{\varphi}$ has Lebesgue spectrum of infinite multiplicity.
\end{itemize}
Therefore $T\times S\times T\times S$ and $T_{\varphi}\times T_{\varphi}$ are spectrally isomorphic.

On the other hand, $T_\varphi$ partially rigid with rigidity constant $\alpha=1/2$ whence $T_\varphi\times T_{\varphi}$ is partially rigid with rigidity constant $\alpha=1/4$, whereas $T\times S\times T\times S$ is not partially rigid.
\end{ex}

The following proposition shows that a flow which is spectrally isomorphic to a flow which is partially rigid with the rigidity constant greater than $1/2$ is also partially rigid.
\begin{pr}\label{pr:last}
Let $\mathcal{T}=\{T_t\}_{t\in\mathbb{R}}$ and $\mathcal{S}=\{S_t\}_{t\in\mathbb{R}}$ be measurable flows on probability Borel spaces $(X,\mu)$ and $(Y,\nu)$ respectively. Suppose that $\mathcal{T}$ and $\mathcal{S}$ are spectrally isomorphic and that $\mathcal{T}$ is partially rigid along $\{t_n\}$ with rigidity constant $1/2 < a \leq 1$. Then $\mathcal{S}$ is also partially rigid along the same sequence.
\end{pr}
\begin{proof}
By assumption, there exists a unitary operator $U\colon L^2(X,\mu) \to L^2(Y,\nu)$ intertwining $\mathcal{T}$ and $\mathcal{S}$, i.e. such that for all $t\in\mathbb{R}$ 
\begin{equation*}
U\circ T_t=S_t\circ U.
\end{equation*}
Passing to a subsequence if necessary, by Remark~\ref{czesc} we obtain
\begin{align*}
\liminf_{n\to\infty} \langle S_{t_n}\mathbbm{1}_A,\mathbbm{1}_A \rangle&=\liminf_{n\to\infty} \langle U\circ T_{t_n}\circ U^{-1}\mathbbm{1}_A,\mathbbm{1}_A \rangle
\\
&=\liminf_{n\to\infty} \langle T_{t_n}\circ U^{-1}\mathbbm{1}_A,U^{-1}\mathbbm{1}_A \rangle
\\
&=a \langle U^{-1}\mathbbm{1}_A,U^{-1}\mathbbm{1}_A\rangle+(1-a)\langle KU^{-1}\mathbbm{1}_A,U^{-1}\mathbbm{1}_A\rangle
\\
&=a\langle \mathbbm{1}_A,\mathbbm{1}_A\rangle + (1-a)\langle KU^{-1}\mathbbm{1}_A,U^{-1}\mathbbm{1}_A\rangle
\\
&\geq a\mu(A)-(1-a)\mu(A)=(2a-1)\mu(A)
\end{align*}
which completes the proof since $2a-1>0$.
\end{proof}

\section*{Acknowledgements}
I would like to thank Professor M.~Lema\'nczyk, Professor K.~Fr\k{a}czek and Professor C.~Ulcigrai for valuable discussions and their encouragement. I would also like to thank the referees for the comments which provided insights that helped improve the paper.

\footnotesize
\bibliography{mybib1}

\end{document}